\documentclass[12pt,a4paper,twoside,reqno]{amsart}

\usepackage{amsmath, amssymb, amsthm, enumerate, graphicx}

\newcommand{\dotcup}{\ensuremath{\mathaccent\cdot\cup}}

\newcommand{\IN}{\mathbb{N}}

\newcommand{\IR}{\mathbb{R}}

\newcommand{\IH}{\mathbb{H}}

\newcommand{\ov}[1]{\overline{#1}}

\DeclareMathOperator{\Ric}{Ric}
\DeclareMathOperator{\tr}{tr}
\DeclareMathOperator{\id}{id}

\DeclareMathOperator{\Isom}{Isom}
\DeclareMathOperator{\diam}{diam}

\DeclareMathOperator{\inj}{inj}
\DeclareMathOperator{\Sym}{Sym}
\DeclareMathOperator{\vol}{vol}

\DeclareMathOperator{\supp}{supp}

\DeclareMathOperator{\DIV}{div}
\DeclareMathOperator{\Int}{Int}
\DeclareMathOperator{\Tor}{\mathbb{T}}

\newcommand{\Mat}[3]{\left( \begin{matrix} #1 & #2 \\ & #3 \end{matrix} \right)}
\newcommand{\sfrac}[2]{{\textstyle \frac{#1}{#2}}}
\newcommand{\twocoeff}[2]{ \big\{ \hspace{-1mm} \begin{smallmatrix} #1 \\ #2 \end{smallmatrix} \hspace{-1mm} \big\} }

\newtheorem{Theorem}{Theorem}[section]
\newtheorem{Lemma}[Theorem]{Lemma}
\newtheorem{Corollary}[Theorem]{Corollary}
\newtheorem{Proposition}[Theorem]{Proposition}

\numberwithin{equation}{section}
\theoremstyle{definition}

\title{Stability of hyperbolic manifolds with cusps under Ricci flow}
\author{Richard H Bamler}
\date{April 12, 2010}

\begin{document}
\begin{abstract}
We show that every finite volume hyperbolic manifold of dimension greater or equal to $3$ is stable under rescaled Ricci flow, i.e. that every small perturbation of the hyperbolic metric flows back to the hyperbolic metric again.
Note that we do not need to make any decay assumptions on this perturbation.

It will turn out that the main difficulty in the proof comes from a weak stability of the cusps which has to do with infinitesimal cusp deformations.
We will overcome this weak stability by using a new analytical method developed by Koch and Lamm.
\end{abstract}

\maketitle
\tableofcontents

\section{Introduction and statement of the result}
In this paper, we will prove the following theorem
\begin{Theorem} \label{Thm:main}
For any complete hyperbolic manifold $(M^n, \ov g)$ of finite volume and dimension $n \geq 3$, there is an $\varepsilon > 0$ such that the following holds: \\
If $g_0$ is another smooth metric on $M$ with
\[ (1- \varepsilon) \ov g \leq g_0 \leq (1+\varepsilon) \ov g, \]
then there is a solution $(g_t)_{t \in [0, \infty)}$ to the rescaled Ricci flow equation
\[ \dot{g}_t = - 2 \Ric_{g_t} - 2 (n-1) g_t \]
starting from $g_0$ which exists for all time and as $t \to \infty$ we have convergence $g_t \longrightarrow \ov g$ in the pointed smooth Cheeger-Gromov sense, i.e. there is a family of diffeomorphisms $\Psi_t$ of $M$ such that $\Psi_t^* g_t \longrightarrow \ov g$ in the smooth sense on every compact subset of $M$. 

Moreover, $\varepsilon$ can be chosen so that it only depends on an upper volume bound on $M$ for $n \geq 4$ resp. an upper diameter bound on the compact part $M_{cpt}$ of $M$ for $n = 3$ (see subsection \ref{subsec:Hypmfs} for more details).
\end{Theorem}

Observe, that the theorem is already known in the \emph{compact} case (see e.g. \cite{Ye}).
The \emph{finite volume} case is more general than the compact case since it allows the manifold to have cusps, and hence to be noncompact (for a geometric description of these manifolds see subsection \ref{subsec:Hypmfs}).
A similar stability result also holds in dimension $2$.
However, one has to take into account a finite dimensional deformation space of the hyperbolic structure of $M$.
The $2$ dimensional case has been treated by Giesen and Topping in \cite{GT} where they show much more, namely that for \emph{any} initial metric $g_0$ on \emph{any} surface $M$ supporting a conformally equivalent hyperbolic metric, the rescaled Ricci flow converges to the hyperbolic metric (see also \cite{JMS} for an earlier approach).

We would like to emphasize that the power of the theorem lies in the fact that we do not impose any decay assumptions on the perturbation $g_0 - \ov g$ at infinity.
In fact, the case in which $g_0 - \ov g$ is small and decays for instance exponentially with respect to the distance from a base point, can be treated by almost the same methods as used for the compact case.
The reason why we dropped the decay assumption comes from the fact that we want to allow cusp deformations (or ``trivial Einstein deformations'' as in \cite{Bam}) which arise arise from deformations of the cross-sectional flat structure of each cusp.

Here is a motivation why this problem is interesting:
In dimension $3$, finite volume hyperbolic manifolds appear as pieces in the long-time behaviour of the Ricci flow.
These pieces constitute exactly the hyperbolic pieces in the geometric decomposition of the given manifold and they are attached to the other geometric pieces along their cusps.
Perelman (\cite{Per}) has used this fact for his proof of the geometrization conjecture.
However, it is still not known whether Ricci flow exhibits the geometric decomposition of the manifold in a stronger sense.
Results into this direction have been achieved by Lott in \cite{Lot} where he shows that this is indeed the case under certain curvature and diameter assumptions which entail that the manifold consists of a single geometric piece.
In the case in which the geometric decomposition contains a hyperbolic piece which is not the only piece, i.e. if the hyperbolic piece has a cusp, much is still unknown.
Furthermore, it is also not known whether there are always only finitely many surgeries in the whole Ricci flow.
One approach of analyzing these questions would be to understand the Ricci flow well enough on each geometric piece and then apply a gluing argument to treat the general case.
In order to do this, stability results play an important role and our theorem gives a very strong statement for the hyperbolic pieces.

Another application of the Theorem would be to treat the question whether the Einstein metric $g_{\ov{\sigma}}$ constructed in \cite{Bam} can also be obtained from the almost Einstein metric (constructed in \cite[chp 3]{Bam}) by Ricci flow.
This might imply that the constant $\varepsilon$ in Theorem \ref{Thm:main} can be chosen only depending on the volume of $M$ in dimension $3$.
Moreover, it would demonstrate a first example for a gluing argument for Ricci flows.

Finally, we hope to use the methods presented in this paper to show an improved stability result for hyperbolic space $(\IH^n, \ov g)$, i.e. to analyze the long-time behaviour of Ricci flow $(g_t)$ if $g_0 - \ov g$ is sufficiently small.
The case in which $g_0 - \ov g$ has $1/r$ decay at infinity has been treated by the author in \cite{BamSymSp} (see also \cite{LiYin}, \cite{SSS2} for earlier results).
A result that does not assume any decay of the perturbation would enable us to choose $\varepsilon$ in Theorem \ref{Thm:main} even independent of the volume of $M$.

The paper is organized as follows: Section \ref{sec:Prelim} explains all necessary background material.
Then we give a brief sketch of the proof in section \ref{sec:Overview}.
In sections \ref{sec:cusp} and \ref{sec:invRF}, we prove a stability theorem for hyperbolic cusps which is then used for the proof of Theorem \ref{Thm:main} in section \ref{sec:whole}.

I would like to thank my advisor Gang Tian for his constant support and encouragement during this project and Hans-Joachim Hein for many helpful discussions.
Special thanks go to Tobias Lamm for presenting his paper \cite{KL} to me during his visit to Princeton in February 2009. 
Its methods are very crucial for the following proof.

\section{Preliminaries} \label{sec:Prelim}
\subsection{Hyperbolic manifolds} \label{subsec:Hypmfs}
We recall the \emph{thick-thin-decomposition} for hyperbolic manifolds (see e.g. \cite[p. 671]{Rat} or \cite[p. 89]{Kapovich})
\begin{Theorem}
There is a constant $\mu_n > 0$, the \emph{Margulis constant,} such that the following holds:
If $(M^n, \ov{g})$ is a finite volume hyperbolic manifold then $M$ can be decomposed into a \emph{thin part} $M_{thin}$ and a \emph{thick part} $M_{thick}$ with $M = M_{thin} \dotcup M_{thick}$ such that:
\begin{itemize}
 \item $\inj > \mu_n$ on $M_{thick}$ and $M_{thick}$ is relatively compact in $M$.
 \item $M_{thin}$ is a finite union of closed sets $N_1, \ldots, N_p$ and $N'_1, \ldots, N'_{p'}$ where 
 \begin{itemize}
 \item the $N_k$ are \emph{cusps} of the form $[0, \infty) \times (\Tor^{n-1} / \Gamma_k)$ for finite subgroups $\Gamma_k < \Isom \Tor^{n-1}$ with a warped product metric
 \[ \ov{g} = ds^2 + e^{-2s} g_{flat, \Tor^{n-1} / \Gamma_k}. \]
 In the case in which $\Gamma_k = \{ 1 \}$, we call $N_k$ \emph{standard}.
 \item and the $N'_k$ are covered by cylindrical neighborhoods around geodesics in hyperbolic space.
 \end{itemize}
Furthermore, we can choose the $N_k$ such that their boundaries are images of horospheres under the universal covering projection and such that $\inj = \mu_n$ at some point on $\partial N_k$.
\end{itemize}
Call $M_{cpt} = M_{thick} \cup N'_1 \cup \ldots \cup N'_{p'}$ the \emph{compact part} of $M$ and $M_{ncpt} = N_1 \cup \ldots \cup N_p$ its \emph{noncompact part}.

In every dimension, $\diam M_{thick}$ is bounded from above by a constant which only depends on an upper bound on $\vol M$ and in dimension $n \geq 4$, this is even true for $\diam M_{cpt}$.
Moreover, in every dimension the number of cusps as well as the geometry of the $\partial N_k$ is bounded by a constant only depending on an upper volume bound on $M$.
\end{Theorem}

The geometry of any such cusp $(N = [0, \infty) \times (\Tor^{n-1}/ \Gamma), \ov g)$ only depends on the geometry its (flat) boundary torus quotient $\partial N = \Tor^{n-1} / \Gamma$.
If $(N, \ov{g})$ is standard, we can choose euclidean coordinates $x_2, \ldots, x_n$ on $\Tor^{n-1}$ (up to an additive constant), let $s = x_1$ denote the coordinate on $[0, \infty)$ and write
\begin{equation} \label{eq:hypcusp} \ov{g} = ds^2 + e^{-2s} (dx_2^2 + \ldots + dx_n^2). \end{equation}

\subsection{(Modified) Ricci deTurck flow} \label{subsec:MRdT}
In this paper we will rather be interested in a modified version of Ricci deTurck flow, than Ricci flow itself.
This flow will agree with Ricci flow and non-modified Ricci deTurck flow modulo pull backs via continuous families of diffeomorphisms.

The rescaled Ricci flow equation reads
\begin{equation} 
\dot{g}^{RF}_t = - 2 \Ric_{g^{RF}_t} - 2 (n-1) g^{RF}_t. \label{eq:nRF}
\end{equation}
In order to define the (nonmodified) Ricci deTurck flow, we need to make use of a distinguished background metric $\ov{g}$ which we will always chose to be the given hyperbolic metric on $M$.
Define the divergence operator
\[ \DIV_{\ov g} : C^\infty(M; \Sym_2 T^*M) \longrightarrow C^\infty(M; T M), \quad h \mapsto - \sum_i (\ov\nabla_{\ov{e}_i} h(\ov{e}_i, \cdot))^{\ov{\#}} \]
where we sum over a local $\ov{g}$-orthonormal frame field $(\ov{e}_i)$ and the musical operator $\ov{\#}$ is also taken with respect to $\ov{g}$.
Set
\[ X'_{\ov g}(h) = \DIV_{\ov g} h + \tfrac12 \ov\nabla \tr_{\ov{g}} h. \]
Then the Ricci deTurck flow equation reads
\[ \dot{g}^{DT}_t = - 2 \Ric_{g^{DT}_t} - 2(n-1) g^{DT}_t - \mathcal{L}_{X'_{\ov g}(g^{DT}_t)} g^{DT}_t. \]
The advantage of Ricci deTurck flow over Ricci flow is that its linearization at $g_t = \ov g$ is strongly elliptic.
This fact has been used by deTurck to give a simplified proof for the short-time existence of Ricci flow (\cite{DeT}).

For our purposes we have to adjust the lower order terms of this flow equation slightly.
The reason for doing this is quite subtle and has to do with the disappearance of certain higher order terms when we consider a very special subclass of solutions, namely the ``invariant cusp'' solutions, in section \ref{sec:invRF}.
The significant effect of this modification will be pointed out on page \pageref{page:MRdF} there.
Loosely speaking, the modification ensures the stability of the cusps.

Any $g \in \Sym_2 T^* M$ can be interpreted as a self-adjoint endomorphism of $TM$ with respect to the background metric $\ov g$.
So for $g$ sufficiently close to $\ov g$, we can define $\log_{\ov g} g$ by the Taylor series of the logarithm in $g-\ov g$.
Now set
\[ X_{\ov g}(g) = \DIV_{\ov g} \log_{\ov g}(g) + \tfrac12 \ov\nabla \tr_{\ov g} \log_{\ov g}(g). \]
The modified Ricci deTurck equation is now defined as
\begin{equation} \dot{g}^{MDT}_t = - 2 \Ric_{g^{MDT}_t} - 2 (n-1) g^{MDT}_t - \mathcal{L}_{X_{\ov g}(g^{MDT}_t)} g^{MDT}_t. \label{eq:MRdTflow} \end{equation}

The following Proposition expresses the equivalence of modified Ricci deTurck flow and Ricci flow.
\begin{Proposition} \label{Prop:MDTisRF}
 Let $(g_t^{MDT})_{t \in [0,T)}$ be a smooth solution to the Ricci deTurck flow equation (\ref{eq:MRdTflow}) and assume that $|g_t^{MDT}-\ov g| < \varepsilon_0$ everywhere for some universal $\varepsilon_0 > 0$.
 Define the time dependent vector field $X_t = X_{\ov g} (g_t^{MDT})$.
 Then $X_t$ has a flow $(\Psi_t)_{t \in [0,T)}$, i.e. a family of diffeomorphisms $\Psi_t : M \to M$ such that
\[ \dot\Psi_t = X_t \circ \Psi_t \qquad \text{and} \qquad \Psi_0 = \id_M, \]
and $g_t = \Psi^*_t g_t^{MDT}$ solves the rescaled Ricci flow equation (\ref{eq:nRF}).
\end{Proposition}
\begin{proof}
For the existence of the flow $(\Psi_t)$ observe that we have $|X_t| \leq C t^{-1/2}$ by Corollary \ref{Cor:Shi}.
The fact that $g_t$ satisfies the rescaled Ricci flow equation can be checked easily.
\end{proof}

Hence, in order to establish Theorem \ref{Thm:main}, it suffices to prove
\begin{Proposition} \label{Prop:main}
For any complete hyperbolic manifold $(M^n, \ov g)$ of finite volume and dimension $n \geq 3$, there is an $\varepsilon > 0$ (which can be chosen as indicated in Theorem \ref{Thm:main}) such that the following holds: \\
If $g_0$ is another smooth metric on $M$ with $\Vert g_0 - \ov g \Vert_{L^\infty} < \varepsilon$, then there is a solution $(g_t)_{t \in [0,\infty)}$ to the modified Ricci deTurck flow equation (\ref{eq:MRdTflow}) which exists for all time and we have smooth convergence $g_t \to \ov g$ as $t \to \infty$ on each compact subset.
\end{Proposition}

Finally, we express equation (\ref{eq:MRdTflow}) in terms of the perturbation $h_t = g^{MDT}_t - \ov g$.
For simplicity, we will write $g$ for $g^{MDT}_t$ and $h$ for $h_t$.
We will use $\ov g$ as a background metric and assume $\overline\Ric_{ab} = -(n-1)\ov g_{ab}$.
Since we will never be dealing with covariant derivatives with respect to $g$, we will denote the covariant derivatives with respect to $\ov g$ by $\nabla$ rather than $\overline \nabla$.
First observe that
\begin{equation*}
\begin{split} 
 2 \Ric_{ab} &= - 2 (n-1) \ov g_{ab} - 2(n-1) h_{ab} + (L h)_{ab} \\
&\qquad + \ov g^{uv} (\nabla^2_{au} h_{bv} + \nabla^2_{bu} h_{av} - \nabla^2_{ab} h_{uv} ) \\
&\qquad + (g^{uv} - \ov g^{uv}) (\nabla^2_{ua} h_{bv} + \nabla^2_{ub} h_{av} - \nabla^2_{uv} h_{ab} - \nabla^2_{ab} h_{uv}) \\
 &\qquad + g^{uv} g^{pq} (\nabla_u h_{pa} \nabla_v h_{qb} - \nabla_p h_{ua} \nabla_v h_{qb} + \tfrac12 \nabla_a h_{up} \nabla_b h_{vq}) \\
 &\qquad + g^{uv}( - \nabla_u h_{vp} + \tfrac12 \nabla_p h_{uv}) g^{pq} (\nabla_a h_{qb} + \nabla_b h_{qa} - \nabla_q h_{ab}) 
\end{split}
\end{equation*}
where $L$ is called \emph{Einstein operator} with
\[ (L h)_{ab} = - \triangle h_{ab} - 2 \ov g^{uv} \ov g^{pq} \ov R_{aupb} h_{vq}. \]
The Lie derivative term can be computed as follows:
\begin{multline} 
 (\mathcal{L}_{X_{\ov{g}}(g)} g)_{ab} = X^u \nabla_u h_{ab} + g_{au} \nabla_b X^u + g_{bu} \nabla_a X^u, \\
 \text{where} \qquad X^u =  \ov{g}^{uv} \ov{g}^{pq} (- \nabla_p (\log g)_{qv} + {\textstyle \frac12} \nabla_v (\log g)_{pq}) \label{eq:Lieder}
\end{multline}
Hence the evolution equation for $h_t$ is
\begin{equation} \partial_t h_t + L h_t = R[h_t] + \nabla^* S[h_t], \label{eq:dhRS} \end{equation}
where with $g = \ov g + h$ and $X^u$ as above
\begin{alignat*}{1}
 R_{ab}[h] &=  - g^{uv} g^{pq} (\nabla_u h_{pa} \nabla_v h_{qb} - \nabla_p h_{ua} \nabla_v h_{qb} + \tfrac12 \nabla_a h_{up} \nabla_b h_{vq}) \\
 &\qquad - g^{uv}( - \nabla_u h_{vp} + \tfrac12 \nabla_p h_{uv}) g^{pq} (\nabla_a h_{qb} + \nabla_b h_{qa} - \nabla_q h_{ab}) - X^u \nabla_u h_{ab} \displaybreak[1]  \\
 &\qquad - \ov g^{sp} \ov g^{qv} \nabla_s h_{pq} (\nabla_a h_{bv} + \nabla_b h_{av} - \nabla_v h_{ab}) + \ov g^{up} \ov g^{vq} \nabla_a h_{pq} \nabla_b h_{uv} \\
 &\qquad + \nabla_a h_{bu} X^u + \nabla_b h_{au} X^u
\end{alignat*}
and $\nabla^* S_{ab} = - \nabla_s S^s_{ab}$ with
\begin{equation*} 
\begin{split}
  S^s_{ab}[h] &= (g^{sv} - \ov g^{sv})(\nabla_a h_{bv} + \nabla_b h_{av} - \nabla_v h_{ab}) - (g^{uv} - \ov g^{uv}) \delta_a^s \nabla_b h_{uv} \\
&\qquad + \delta_a^s (\ov g^{uv} \nabla_u h_{bv} - \tfrac12 \ov g^{uv} \nabla_b h_{uv} + g_{bu} X^u) \\
&\qquad + \delta_b^s (\ov g^{uv} \nabla_u h_{av} - \tfrac12 \ov g^{uv} \nabla_a h_{uv} + g_{au} X^u).
\end{split}
\end{equation*}
Observe that at every point, $S_{ab}^c$ is linear in $\nabla h$.
Moreover, if $h = 0$ (but not necessarily $\nabla h = 0$) at that point, we have $S_{ab}^c = 0$.
Hence if $|h|< 0.1$, we can bound
\begin{equation} |R_{ab}|[h] \leq C |\nabla h|^2, \qquad |S_{ab}^s|[h] \leq C |h| |\nabla h|. \label{eq:boundRS} \end{equation}
We will also make use of the following simpler identity: if $|h_t|< 0.1$ then by (\ref{eq:dhRS})
\begin{equation} \label{eq:Qtequation}
 |(\partial_t + L) h_t| \leq C( |\nabla h_t|^2 + |h_t| |\nabla^2 h_t| ).
\end{equation}

\subsection{The Einstein operator} \label{subsec:Einstop}
We will now analyze the Einstein operator $L$.
If $\ov g$ is the hyperbolic metric, then
\begin{equation} (L h)_{ab} = -\triangle h_{ab} - 2 h_{ab} + 2 \ov g^{ij} h_{ij} \ov g_{ab}. \label{eq:Lonhyp} \end{equation}
We can also derive a Weitzenb\"ock formula for $L$.
The formal adjoint of $\DIV_{\ov g}$ is
\[ \DIV_{\ov g}^*: C^\infty(M; TM) \to C^\infty(M;\Sym_2 T^*M), \quad X^a \mapsto {\textstyle \frac12} (\ov g_{bi} \nabla_a X^i + \ov g_{ai} \nabla_b X^i ). \]
Furthermore, let
\[ d : C^\infty(M; \Sym_2 T^* M) \to C^\infty(M; \Lambda_2 T^*M \otimes T^*M), \quad h_{ab} \mapsto \nabla_a h_{bc} - \nabla_b h_{ac}. \]
Its formal adjoint is
\begin{multline*} d^* : C^\infty(M; \Lambda_2 T^*M \otimes T^*M) \to C^\infty(M; \Sym_2 T^* M), \\ t_{abc} \mapsto - {\textstyle \frac12} (\ov{g}^{ij} \nabla_i t_{jab} + \ov{g}^{ij} \nabla_i t_{jba}). \end{multline*}
Then, using the assumption $\overline \Ric_{ab} = - (n-1) \ov g_{ab}$, we can compute
\[ (L h)_{ab} = (\DIV_{\ov g}^* \DIV_{\ov g} + d^* d)h_{ab} - \ov R_{astb} h_{st} + (n-1) h_{ab}. \]
Hence if $\ov g$ is the hyperbolic metric, then
\[ (L h)_{ab} = (\DIV_{\ov g}^* \DIV_{\ov g} + d^* d)h_{ab} + \ov g^{ij} h_{ij} \ov g_{ab} + (n-2) h_{ab}. \]
Thus in a setting where we can apply Stokes' Theorem, we have $L \geq n-2$ in the $L^2$-sense:
\begin{equation} \langle Lh, h \rangle = \Vert {\DIV_{\ov g}} \Vert_{L^2}^2 + \Vert d h \Vert_{L^2}^2 + \Vert {\tr_{\ov g} h} \Vert_{L^2}^2 + (n-2) \Vert h \Vert_{L^2}^2 \geq  (n-2) \Vert h \Vert_{L^2}^2. \label{eq:Lgeqnm2}
\end{equation}

Finally, we compute the action of the Einstein operator on $\Tor^{n-1}$-invariant sections on a hyperbolic cusp.
Let $(N = \IR \times \Tor^{n-1}, \ov g)$ be a cusp with coordinates $(s,x_2, \ldots x_n)$ and $\ov g$ defined as in (\ref{eq:hypcusp}).
Assume that $h \in \Sym_2 T^*N$ is $\Tor^{n-1}$-invariant, i.e. of the form 
\[ h = h_{11} dx_1 dx_1 +  e^{-s} \sum_{i=2}^n h_{1i} (dx_1 dx_i + dx_i dx_1) + e^{-2s} \sum_{i,j = 2}^n h_{ij} dx_i dx_j, \]
where the coefficients $h_{ij}$ depend only on $s$.
Then we can express
\[ Lh = (Lh)_{11} dx_1 dx_1 +  e^{-s} \sum_{i=2}^n (Lh)_{1i} (dx_1 dx_i + dx_i dx_1) + e^{-2s} \sum_{i,j = 2}^n (Lh)_{ij} dx_i dx_j, \]
where the coefficients $(Lh)_{ij}$ can be computed as follows $(i,j>1)$
\begin{subequations}
\begin{alignat}{3}
 -(Lh)_{11} &= h''_{11} &&- (n-1) h'_{11} &&- 2(n-1) h_{11} \label{eq:Linv1} \\
 -(Lh)_{1i} &= h''_{1i} &&- (n-1) h'_{1i} &&- n h_{1i} \label{eq:Linv2} \\
 -(Lh)_{ij} &= h''_{ij} &&- (n-1) h'_{ij} &&- 2 \delta_{ij} {\textstyle \sum_{k=2}^n} h_{kk} \label{eq:Linv3}
\end{alignat}
\end{subequations}

\subsection{A result from harmonic analysis}
Let $\Phi \in C^\infty (\IR \times \IR_+), \Phi(x,t) = (4 \pi t)^{-1/2} \exp (- \frac{x^2}{4t} )$ be the one dimensional heat kernel, i.e. $\partial_t \Phi = \partial^2_x \Phi$.
We will need the following result from harmonic analysis (see \cite{KrySob} or \cite{Ste}):
\begin{Lemma} \label{Lem:CZ}
Assume that $r > 0$ and set $\Omega = [-r,r] \times [0,r^2]$.
For every $f \in C^\infty_0 (\Omega)$ we can compute the convolution $\Phi'' * f$ and restrict it to $\Omega$ (prime denotes differentiation by $x$).
For every $1<p<\infty$ this induces a map
\[ \Phi'' * : \;\; L^p (\Omega) \longrightarrow L^p (\Omega) \]
and
\[ \Vert \Phi'' * f \Vert_{L^p (\Omega)} \leq C(p) \Vert f \Vert_{L^p (\Omega)}. \]
The same is true for the operator $\ov \Phi'' *$, where $\ov\Phi(x,t) = e^{-\zeta t} \Phi(x,t)$ for any $\zeta \geq 0$.
\end{Lemma}
\begin{proof}
We will only discuss the operator $\Phi'' *$ here.
The proof for $\ov \Phi''$ goes along the lines.

In the first step we prove the Lemma for $p=2$:
Let $f \in C^\infty_0 (\Omega)$ and $h = \Phi * f' = \Phi' * f \in C^\infty(\IR \times [0,r^2])$.
Since $\dot h - h'' = f'$ we obtain for every time slice
\[  \partial_t \int_{\IR} h^2 + 2 \int_{\IR} (h')^2 =  2 \int_{\IR}  h f' = - 2 \int_{\IR} h' f \leq \int_{\IR} (h')^2 + \int_{\IR} f^2. \]
Integrating both sides from $0$ to $r^2$ yields
\[ \Vert h' \Vert^2_{L^2(\Omega)} \leq  \Vert f \Vert^2_{L^2(\Omega)}. \]

In the second step we prove that the operator $\Phi'' *$ is weak $(1,1)$.
Assume again that $f \in C^\infty_0(\Omega)$ and set $k = \Phi'' * f$.
Let $\alpha > 0$ be arbitrary.
If $\alpha \leq \frac1{2r^3} \Vert f \Vert_{L^1(\Omega)}$, then trivially $|\{ |k| > \alpha \}| \leq | \Omega | = 2 r^3 \leq \alpha^{-1} \Vert f \Vert_{L^1(\Omega)}$.
Assume now $\alpha > \frac1{2r^3} \Vert f \Vert_{L^1(\Omega)}$ and consider the Calder\'on-Zygmund decomposition $f = g+b$ with
\begin{enumerate}[(a)]
\item $|g| < \alpha$ on $\Omega$ and $\Vert g \Vert_{L^1(\Omega)} \leq \Vert f \Vert_{L^1(\Omega)}$.
\item $\supp b \subset B = \bigcup_{i=1}^N Q_i$ where the $Q_i$ are parabolic domains of the form $Q_i = [x_i - r_i, x_i + r_i] \times [t_i, t_i + r^2_i] \subset \Omega$ which are pairwise disjoint except for intersection of their boundary.
\item $\int_{Q_i} b = 0$ and $\int_{Q_i} |b| \leq 2 |Q_i| \alpha$.
\item $|B| \leq \frac{8}{\alpha} \Vert f \Vert_{L^1(\Omega)}$.
\item $g,b$ are smooth on an open dense subset of $\Omega$ of full measure.
\end{enumerate}
Then $k= k_g + k_b$ where $k_g = \Phi'' * g$ and $k_b = \Phi'' * b$.
Moreover $|\{ |k| > \alpha \}| \leq |\{ |k_g| > \alpha/2 \}| + | \{ |k_b| > \alpha/2 \}|$.
Since $\Vert g \Vert_{L^2(\Omega)} \leq \alpha^{1/2} \Vert g \Vert^{1/2}_{L^1(\Omega)}$, we get using the first step that
\[ | \{ |k_g| > \alpha/2 \} | \leq \frac4{\alpha^2} \Vert k_g \Vert^2_{L^2(\Omega)} \leq \frac{4}{\alpha} \Vert g \Vert_{L^1(\Omega)} \leq \frac{4}{\alpha} \Vert f \Vert_{L^1(\Omega)}. \]

We will now analyze $k_b$.
For every $Q_i = [x_i - r_i, x_i + r_i] \times [t_i, t_i + r_i^2]$ set $Q'_i = [x_i - 2 r_i, x_i + 2 r_i] \times [t_i, t_i + 4 r_i^2] \cap \Omega$ and let $B' = \bigcup_{i=1}^N Q'_i$.
Obviously, $|Q'_i| \leq 8 |Q_i|$ and $|B'| \leq 8 |B| < \frac{64}{\alpha} \Vert f \Vert_{L^1(\Omega)}$.
Decompose $b = b_1 + \ldots + b_N$ where $b_i = \chi_{Q_i} b$.
Then $k_b = k_{b_1} + \ldots + k_{b_N}$.
Fix one $i$ for the moment and consider a point $(x,t) \in \Omega \setminus Q'_i$.
By (c) we have
\[ k_{b_i}(x,t) = \int_{Q_i} \bigl( \Phi''(x-x',t-t') - \Phi''(x-x_i,t-t_i) \bigr) b_i(x',t') d x' d t'. \]
Using the fact that the absolute value of the difference in the parentheses is bounded by
\begin{multline*}
 \sup_{(x',t') \in Q_i} \left( r_i | \Phi''' |(x-x',t-t') + r_i^2 | \partial_t \Phi'' |(x-x',t-t') \right) \displaybreak[1] \\
 \leq r_i \frac{C}{(t - t_i +r_i^2)^{3/2}} r_i^{-1} \exp(-c|x-x_i|/r_i) + r_i^2 \frac{C}{(t-t_i+r_i^2)^2} r_i^{-1} \exp(-c|x-x_i|/r_i),
\end{multline*}
we find
\[ |k_{b_i}|(x,t) \leq \int_{Q_i} |b| \cdot \left( \frac{C r_i}{(t-t_i+r_i^2)^{3/2}} + \frac{Cr_i^2}{(t-t_i+r_i^2)^2} \right) r_i^{-1} \exp(-c|x-x_i|/r_i) \]
Hence, since the second factor is bounded in $L^1$ independently of $r_i$, Young's inequality yields
\[ \int_{\Omega \setminus Q'_i} |k_{b_i}| \leq C \int_{Q_i} |b| \leq 2 C \alpha |Q_i|.  \]
And thus $\int_{\Omega \setminus B'} |k_b| \leq C \alpha |B| \leq C \Vert f \Vert_{L^1(\Omega)}$.
This implies $|\{ |k_b| > \alpha/2 \}| \leq |B'| + \frac{C}{\alpha} \Vert f \Vert_{L^1(\Omega)} \leq \frac{C}{\alpha} \Vert f \Vert_{L^1(\Omega)}$.
Putting both terms together, we finally get $| \{ |k| > \alpha \}| \leq \frac{C}{\alpha} \Vert f \Vert_{L^1(\Omega)}$.

Having established that the operator $\Phi'' *$ is strong $(2,2)$ and weak $(1,1)$, we conclude by the Marcinkiewicz interpolation theorem that it is strong $(p,p)$ for all $1< p \leq 2$.
The rest of the Lemma follows by duality:
For every $1<p\leq 2$ and conjugate $2 \leq p^* < \infty$, we have with $\Phi_- (x,t) = \Phi(x,-t)$
\[ \langle \Phi'' * f_1, f_2 \rangle_{\Omega} = \langle f_1, \Phi_-'' * f_2 \rangle_{\Omega} \leq C \Vert f_1 \Vert_{L^{p^*}(\Omega)} \Vert f_2 \Vert_{L^p(\Omega)}. \]
Thus, the result is also true for $p^*$.
\end{proof}

\subsection{Derivative bounds for linear and nonlinear parabolic equations}
We recall an a priori derivative estimate for linear or a certain type of nonlinear parabolic equations.
If $\Omega \subset \IR^n \times \IR$ denotes some parabolic neighborhood in space-time (e.g. $\Omega = B_r(0) \times [0,T]$), then we will denote by $C^{2m;m}(\Omega)$ the space of scalar functions on $\Omega$ which are $i$ times differentiable in spatial direction and $j$ times differentiable in time direction if $i + 2j \leq 2m$.
For $\alpha \in (0,\frac12)$, the corresponding H\"older space will be denoted by $C^{2m, 2\alpha; m, \alpha}(\Omega)$.

In order to present our results in a scale invariant way, we will use the following weights to define the H\"older norm on $C^{2m, 2\alpha; m, \alpha}(\Omega)$:
Assume 
\begin{equation*} r = \min \{ r' \; : \; \text{$\Omega \subset B_{r'}(p) \times [t-(r')^2, t]$ for some $p$, $t$} \} < \infty. \end{equation*}
Then set
\[ \Vert u \Vert_{C^{2m, 2\alpha; m, \alpha}(\Omega)} = \sum_{|\iota|+2k \leq 2m} r^{|\iota|+2k} (\Vert D^\iota \partial_t^k u \Vert_{C^0} + r^{2\alpha} [ D^\iota \partial_t^k u ]_{2\alpha,\alpha} ), \]
where $\iota$ runs over products of spatial derivatives.

Set $B_r = B_r(0) \subset \IR^n$.

\begin{Proposition} \label{Prop:Shi}
 Let $r > 0$ and consider the parabolic neighborhoods $\Omega = B_r \times [-r^2,0]$ and $\Omega' = B_{2r} \times [-4r^2,0]$.
 
 Assume that $u \in C^{2;1}(\Omega')$ satisfies the equation
 \begin{multline} (\partial_t - L) u = Q[u] = r^{-2} f_1 (r^{-1} x, u) \cdot u + r^{-1} f_2(r^{-1} x, u) \cdot \nabla u \\ + f_3(r^{-1} x, u) \cdot \nabla u \otimes \nabla u + f_4(r^{-1} x, u) \cdot u \otimes \nabla^2 u, \label{eq:Shiequ}
 \end{multline}
 where $f_1, \ldots, f_4$ are smooth functions in $x$ and $u$ such that $f_2, f_3, f_4$  can be paired with the tensors $u \otimes \nabla u$, $\nabla u \otimes \nabla u$ resp. $u \otimes \nabla^2 u$.
 Assume that the linear operator $L$ has the form
 \begin{equation} \label{eq:defofL}
  L u = a_{ij}(x) \partial_{ij}^2 u + b_i(x) \partial_i u + c(x) u.
 \end{equation}
 
 Now assume that we have the following bounds for $m \geq 1$, $\alpha \in (0,\frac12)$:
 \begin{equation} 
 \begin{split} \frac{1}{\Lambda} < a_{ij} < \Lambda, \quad \Vert a_{ij} \Vert_{C^{2m-2, 2\alpha; m-1, \alpha}(\Omega')} < \Lambda, \\ \quad \Vert b_{i} \Vert_{C^{2m-2, 2\alpha; m-1, \alpha}(\Omega')} < r^{-1} \Lambda, \quad \Vert c \Vert_{C^{2m-2, 2\alpha; m-1, \alpha}(\Omega')} < r^{-2} \Lambda. \label{eq:Lambdaest}
 \end{split}
 \end{equation}
 
 Then there are constants $\varepsilon_m > 0$ and $C_m < \infty$ depending only on $\Lambda$, $\alpha$, $n$, $m$ and the $f_i$ such that if
 \[ H = \Vert u \Vert_{L^\infty(\Omega')}  < \varepsilon_m, \]
 then
 \[ \Vert u \Vert_{C^{2m, 2\alpha; m, \alpha}(\Omega)} < C_m H . \]
 
 Moreover, the Proposition still holds if $u$ is vector-valued.
 In this case $a_{ij}$, $b_i$, $c$, $f_1, \ldots, f_4$ have to be tensors of the appropriate shape and we need to assume that for each $i, j$ the coefficient $a_{ij}$ is a multiple of the identity matrix.
\end{Proposition}

Observe that for $f_i = 0$, this includes the linear case.
In the following proof, we will for simplicity always assume that $u$ is a scalar function.
The vector-valued case follows by exactly the same arguments (note that we can even still use Lemma \ref{Lem:KryHoe} for the scalar case, since we can actually include the terms $b_i(x) \partial_i u$ and $c(x) u$ into the nonlinear terms involving $f_2$ resp. $f_1$).
In order to prove Proposition \ref{Prop:Shi}, we will need the following

\begin{Lemma} \label{Lem:KryHoe}
 Assume that $\Omega \subset \Omega'$ are defined as in Proposition \ref{Prop:Shi} and that (\ref{eq:defofL}) and (\ref{eq:Lambdaest}) hold.
 
 Then if $u \in C^{2;1}(\Omega')$ satisfies the equation
 \[ (\partial_t - L) u = f, \]
 we have the interior bound
 \[ \Vert u \Vert_{C^{2m,2\alpha;m,\alpha}(\Omega)} \leq C_m (r^2 \Vert f \Vert_{C^{2m-2, 2\alpha; m-1, \alpha}(\Omega')} + \Vert u \Vert_{C^0(\Omega')} ). \]
 Here $C_m$ depends only on $\Lambda$, $\alpha$ and $n$.
\end{Lemma}
\begin{proof}
 For $m=1$, the Lemma is exactly the same as Theorem 8.11.1 in \cite{KryHoe} and for $m>1$ it follows by differentiation.
\end{proof}

\begin{proof}[Proof of Proposition \ref{Prop:Shi}]
We derive a slightly stronger statement from Lemma \ref{Lem:KryHoe}.
In order to do this, we introduce a new weighted norm for $0 < \theta \leq 1$:
\[ \Vert u \Vert^{(\theta)}_{C^{2m, 2\alpha; m, \alpha}(\Omega)} = \sum_{|\iota|+2k \leq 2m} (r\theta)^{|\iota|+2k} (\Vert D^\iota \partial_t^k u \Vert_{C^0} + (r\theta)^{2\alpha} [ D^\iota \partial_t^k u ]_{2\alpha,\alpha} ). \]
Observe that for $\theta = 1$, this norm agrees with the previous norm.
Applying the Lemma to any ball $B_{\theta r} (p) \subset B_r$, we can deduce
\begin{flushleft}
\begingroup \it
\leftskip=1cm 
\noindent
 Assume we are in the setting of Lemma \ref{Lem:KryHoe} except that now $\Omega' = B_{(1+\theta)r} \times [-(1+\theta)^2 r^2,0]$.
 Then
 \[ \qquad \Vert u \Vert^{(\theta)}_{C^{2m,2\alpha;m,\alpha}(\Omega)} \leq C_m \bigr( (r\theta)^2 \Vert f \Vert^{(\theta)}_{C^{2m-2, 2\alpha; m-1, \alpha}(\Omega')} + \Vert u \Vert_{C^0(\Omega')}  \bigr). \]
\par
\endgroup
\end{flushleft}
Now consider the setting of Proposition \ref{Prop:Shi}.
By scaling invariance, we can assume $r=1$.
In the following, we will abbreviate every constant which only depends on $\Lambda$, $\alpha$, $n$, $m$ and the $f_i$ by $C$.

Set $r_k = \sum_{i=0}^k 2^{-i} = 2- 2^{-k}$, $\theta_k = \frac{r_{k+1}}{r_k} - 1$ and $\Omega_k = B_{r_k}(0) \times [-r_k^2,0]$.
By (\ref{eq:Shiequ})
\[ a_k := \Vert u \Vert_{C^{2m,2\alpha;m,\alpha}(\Omega_k)}^{(\theta_k)} \leq C \bigr( \theta_k^2 \Vert Q[u] \Vert_{C^{2m-2,2\alpha;m-1,\alpha}(\Omega_{k+1})}^{(\theta_{k+1})} + H \bigr). \]
Observe that since $\theta_k \to 0$, we have $a_k \to a_{\infty} = \Vert u \Vert_{C^0(\Omega')} \leq H$.
We now estimate $Q[u]$ in terms of $u$ using (\ref{eq:Shiequ}).
For this note that for $i = 1, \ldots, 4$
\[ \Vert f_i(x,u) \Vert^{(\theta_{k+1})}_{C^{2m-2,2\alpha; m-1, \alpha}(\Omega_{k+1})} \leq C \big( 1 + \big( \Vert u \Vert^{(\theta_{k+1})}_{C^{2m-2,2\alpha;m-1,\alpha}(\Omega_{k+1})} \big)^{2m-1} \big) \leq C (1+a_{k+1}^{2m-1}). \]
So
\begin{multline*}
 \Vert f_1 \cdot u \Vert^{(\theta_{k+1})}_{C^{2m-2,2\alpha;m-1,\alpha}(\Omega_{k+1})}
  \leq C \Vert f_1 \Vert^{(\theta_{k+1})}_{C^{2m-2,2\alpha;m-1,\alpha}(\Omega_{k+1})} \Vert u \Vert^{(\theta_{k+1})}_{C^{2m-2,2\alpha;m-1,\alpha}(\Omega_{k+1})} \\
  \leq C(a_{k+1} + a_{k+1}^{2m}).
\end{multline*}
Similarly
\begin{alignat*}{1}
 \Vert f_2 \cdot \nabla u \Vert^{(\theta_{k+1})}_{C^{2m-2,2\alpha;m-1,\alpha}(\Omega_{k+1})} &\leq C (1+a_{k+1}^{2m-1}) \theta^{-1}_{k+1} \Vert u \Vert^{(\theta_{k+1})}_{C^{2m,2\alpha;m,\alpha} (\Omega_{k+1})} \\
 & \leq C \theta^{-1}_{k+1} (a_{k+1} + a_{k+1}^{2m}), \\
 \Vert f_3 \cdot \nabla u \otimes \nabla u \Vert^{(\theta_{k+1})}_{C^{2m-2,2\alpha;m-1,\alpha}(\Omega_{k+1})} &\leq C \theta^{-2}_{k+1} (a_{k+1}^2 + a_{k+1}^{2m+1}), \\
 \Vert f_4 \cdot  u \otimes \nabla^2 u \Vert^{(\theta_{k+1})}_{C^{2m-2,2\alpha;m-1,\alpha}(\Omega_{k+1})} &\leq C \theta^{-2}_{k+1} (a_{k+1}^2 + a_{k+1}^{2m+1}).
\end{alignat*}
We conclude
\[ \Vert Q[u] \Vert^{(\theta_{k+1})}_{C^{2m-2,2\alpha;m-1,\alpha}(\Omega_{k+1})} \leq C \big( \theta^{-1}_{k+1} a_{k+1} + \theta^{-2}_{k+1} a_{k+1}^2 + \theta_{k+1}^{-1} a_{k+1}^{2m} + \theta^{-2}_{k+1} a^{2m+1}_{k+1} \big). \]
Hence
\[ a_k \leq C(\theta_{k+1} a_{k+1} + a_{k+1}^2 + a_{k+1}^{2m} + a_{k+1}^{2m+1} + H). \]
So the quantity $b_k = a_k/H$ satisfies the following recursion inequality
\begin{equation} b_k \leq C_0 ( \theta_{k+1} b_{k+1} + H b_{k+1}^2 + H^{2m-1} b_{k+1}^{2m} + H^{2m} b_{k+1}^{2m+1} + 1). \label{eq:bn}
\end{equation}
Assume that $C_0 > 1$, set $\varepsilon_m = \varepsilon = \frac1{16 C_0^2}$ and choose $k_0$ such that $\theta_{k + 1} < \varepsilon$ for all $k \geq k_0$.
Since we assumed that $H < \varepsilon$, we get for $k \geq n_0$
\[ b_k \leq \frac1{16} b_{k+1} + \frac1{16 C_0} b_{k+1}^2 + \frac1{16^{2m-1} C_0^{4m-3}} b_{k+1}^{2m} + \frac1{16^{2m} C_0^{4m-1}} b_{k+1}^{2 m +1} + C_0. \]
So if $b_{k+1} < 2C_0$, it follows that $b_k < 2C_0$, too.
Hence by induction and the fact that $b_k \to a_\infty/H \leq 1 < 2 C_0$, it follows that $b_{k_0} < 2C_0$.

Finally, using (\ref{eq:bn}), we can derive a bound $C'$ for $b_0$.
So $a_0 \leq C' H$.
This finishes the proof.
\end{proof}

We will frequently make use of the following consequence of Proposition \ref{Prop:Shi}.
\begin{Corollary} \label{Cor:Shi}
 Let $\tau > 0$ and assume that $(\ov{g} + h_t)_{t \in [0,\tau)}$ satisfies either the modified Ricci deTurck flow equation (\ref{eq:MRdTflow}) or the linearized flow equation $\partial_t h_t + L h_t = 0$ (see (\ref{eq:dhRS})) on a domain $D \subset M$, where $M^n$ denotes any complete Riemannian manifold with boundary.
 Assume moreover, that the $\tau^{1/2}$ tubular neighborhood $D'$ of $\Omega$ does not meet $\partial M$.

 Then for any $m$, there exist constants $\varepsilon_m > 0, C_m < \infty$ depending only on $m$, $n$, $\tau$ and bounds on the curvature tensor of $M$ as well as its derivatives, such that if
\[ H = \Vert h \Vert_{L^\infty(D' \times [0,\tau))} < \varepsilon_m, \]
then
\[ \Vert \nabla^m h_t \Vert_{L^\infty(D)} < C_m t^{-m/2} H \qquad \text{for all $t \in [0,\tau)$}. \]

For the linearized flow equation, we do not have to assume the bound $H < \varepsilon_m$.
\end{Corollary}
Observe that $\varepsilon_m, C_m$ are in particular independent of the injectivity radius of $M$.
\begin{proof}
 At each point $p \in D$ pass over to a local cover and consider the domains $\Omega = B_r(p) \times [3r^2,4r^2] \subset B_{2r}(p) \times [0,4r^2] = \Omega'$ for $0 < r < \frac12 T^{1/2}$.
 Proposition \ref{Prop:Shi} then yields the desired result.
 
 In the case of the linearized flow equation, we can analyze the flow $(\delta h_t)$ for sufficiently small $\delta > 0$.
\end{proof}

\subsection{Short-time existence}
Equation (\ref{eq:dhRS}) implies that the modified Ricci deTurck flow equation (\ref{eq:MRdTflow}) is strongly parabolic if $h_t$ is small enough.
We quote a general short-time existence result which follows by a standard inverse function theorem argument.
For more details see \cite{Shi}, \cite{LS}, \cite[sec 4]{SSS1} and \cite{Bamler-thesis}.

\begin{Proposition}[Short-time existence] \label{Prop:shortex}
 Let $(M,\ov g)$ be an arbitrary Einstein manifold of bounded curvature and with Einstein constant $\lambda = -n+1$ and let $m_0 \in \IN$.
 Then there are $\varepsilon_{s.e.}, \tau_{s.e.} > 0$, $C_{s.e., m} < \infty$ such that the following holds: \\
Let $g_0$ be another smooth metric on $M$ such that
\[ \Vert g_0 - \ov g \Vert_{L^{\infty}(M)} < \varepsilon_{s.e.}, \]
then there is a smooth solution $(g_t) \in C^{\infty}(M \times [0,\tau_{s.e.}])$ to the modified Ricci deTurck flow equation (\ref{eq:MRdTflow}) with initial metric $g_0$.
Moreover, we have the bound
\[ 
\Vert g_t - \ov g \Vert_{L^\infty (M \times [0,\tau_{s.e.}])} \leq C_{s.e.,0}  \Vert g_0 - \ov g \Vert_{L^\infty(M)}
\]
and for every $m \leq m_0$ we have
\[ 
\Vert g_t - \ov g \Vert_{C^{2m;m}(M \times [\eta \tau_{s.e.},\tau_{s.e.}])} \leq C_{s.e.,m} \eta^{-m} \Vert g_0 - \ov g \Vert_{L^\infty(M)}  \quad \text{for all $\eta \in (0,1]$}.
\]

Moreover, the solution $(g_t)_{t \in [0,\tau_{s.e.}]}$ is unique amongst all solutions $(g'_t)_{t \in [0,\tau']}$ for which $\Vert g'_t - \ov g \Vert_{L^\infty(M \times [0,\tau'])} < C_{s.e., 0} \varepsilon_{s.e.}$.
\end{Proposition}

\section{Outline of the proof} \label{sec:Overview}
We give a brief sketch of the proof.
The linearization of the modified Ricci deTurck equation (\ref{eq:MRdTflow}) or (\ref{eq:dhRS}) in terms of the perturbation $h_t = g_t - \ov g$ reads $(\partial_t + L) h_t = 0$.
By (\ref{eq:Lgeqnm2}), the Einstein operator $L$ is strictly positive and hence the linearized flow  is indeed strongly attractive in the $L^2$-sense (i.e. the $L^2$-norm of every solution decays exponentially for $t \to \infty$).
In the case in which $M$ has no cusps and hence the injectivity radius is uniformly bounded from below, it easily follows that $L$ is actually strongly attractive in the $L^\infty$-sense.
It is then possible to show that also the nonlinear flow equation is strongly attractive.

However, in the case of manifolds with cusps, we lose the $L^\infty$-attractiveness.
The reason for this is the following:
Look at a very long part of a cusp which which is very far from the compact part of $M$ and consider a perturbation $h$ which is supported in this region and $\Tor^{n-1}$-invariant (see the end of subsection \ref{subsec:Einstop}).
It is possible to choose $h_{ij}$ such that its derivatives in the $s$-direction are very small, but such that at some point, say the $23$ entry attains a value bigger than $\frac1{1000}$. 
Then by looking at (\ref{eq:Linv1})-(\ref{eq:Linv3}), we expect $h_{23}$ to decay very slowly in time.
The geometric reason behind this weak stability is that hyperbolic cusps admit so-called trivial Einstein deformations (see \cite{Bam}), i.e. certain metric deformations which still satisfy the Einstein equation and which correspond to deformations of the flat structure on $\Tor^{n-1}$.
In our case, $h$ approximates such a trivial Einstein deformation.

Hence, the most important part of the proof is to show that despite this slow decay, we still have longtime bounds for solutions of of the nonlinear equation (\ref{eq:dhRS}) on a hyperbolic cusp.
This discussion is started in section \ref{sec:cusp}, where the solution is split into two components: one which is $\Tor^{n-1}$-invariant, i.e. constant along the cross-sectional tori (and hence its linearization is described by (\ref{eq:Linv1})-(\ref{eq:Linv3})) and one which averages out to $0$ over the cross-sections.
The flow equation can be expressed as a coupled system of flow equations in those two components.
It will then be shown that the linearization of the equation describing the second component is strongly attractive.

So it remains to analyze solutions to the equation describing the first component.
This equation is equivalent to (\ref{eq:MRdTflow}) or (\ref{eq:dhRS}) for $h_t$ being $\Tor^{n-1}$-invariant, but with an extra input term.
It can be reduced to a system of nonlinear parabolic equations in two variables $s$ and $t$ only.
We will discuss it in section \ref{sec:invRF}.
Here it becomes important to analyze the nonlinear term of the flow equation very carefully and our modification of the Ricci deTurck flow will turn out to be essential (we will point out when the modification becomes important on page \pageref{page:MRdF}).
Once we have described the algebraic structure of the nonlinear term, we apply analytical tools which were developed by Koch and Lamm in \cite{KL} and which we need to adapt to our situation.
Note that this part is actually the heart of the proof.

Finally in section \ref{sec:whole}, we establish the stability of the whole manifold $M$.
Since we have a uniform lower bound on the injectivity radius on the thick part of $M$, we can use the same arguments as in the no-cusp case there.
We then have to incorporate the longtime bounds on the cusps.
Here we have to carefully choose the border between the thick and the thin part of $M$ depending on the time.
We note that in this section we give a detailed description of how the convergence takes place.

In the following $C$, will always denote a dynamic constant which only depends on the quantities which are indicated in the beginning of each section.
For simplicity, we assume that $C > 1$.
Moreover, $\sigma > 0$ will denote a constant which we will have to choose sufficiently small.
It will always be clear that we can fix $\sigma$ first and then choose $C$ depending on it.

\section{Modified Ricci deTurck flow on a cusp} \label{sec:cusp}
\subsection{Introduction}
In this part we consider the following setting:
Let $\Tor = \Tor^{n-1}$ be a flat torus, $\Tor / \Gamma$ a finite quotient and consider the corresponding hyperbolic cusp $(N = [0,\infty) \times (\Tor / \Gamma), \ov g)$.
If the cusp is standard (i.e. $\Gamma =  \{ 1 \}$) we can choose coordinates $(s,x_2, \ldots, x_n)$ such that (see subsection \ref{subsec:Hypmfs})
\[ \ov{g} = d s^2 + e^{-2s} (dx_2^2 + \ldots + dx_n^2). \]
Note that $N$ is contained in the complete hyperbolic cusp $N' = (\IR \times (\Tor / \Gamma), \ov g)$.
Denote by $B_{\sigma}(\partial N) = [0, \sigma) \times (\Tor / \Gamma)$ the tubular neighborhood of radius $\sigma$ around $\partial N$ in $N$.

In this section we will prove
\begin{Theorem} \label{Thm:RdToncusp}
Let $T > (10\sigma)^2$ be some maximal time and consider a solution $(g_t)_{t \in [0,T)}$ to the modified Ricci deTurck flow equation (\ref{eq:MRdTflow}) on $(N,\ov g)$.
Assume moreover, that $g_t$, $\nabla g_t$, $\nabla^2 g_t$ and $\nabla^3 g_t$ are uniformly bounded on $N \times [0,T)$ by some constant. \\
Then for any $\delta > 0$ there are constants $\varepsilon_{cusp} > 0$ and $C_{cusp} < \infty$, both depending only and continuously on the geometry of $\Tor / \Gamma$ and on $\delta$, such that if
\[ H = \sup_{\stackrel{ \scriptstyle B_{10\sigma}( \partial N ) \times [0,T) }{\cup N \times [0, (10\sigma)^2) } } e^{\delta t} \big( | g_t - \ov g | + |\nabla g_t| \big) < \varepsilon_{cusp}, \]
then $\Vert g_t - \ov g \Vert_{L^\infty(N \times [0,T))} < C_{cusp} H$.
\end{Theorem}
Since we can pass to a finite cover, we will assume that $(N, \ov{g})$ is standard.

The idea of the proof is the following:
We split $g_t$ into a sum of two components, namely its invariant component $g^{inv}_t$ which is constant along all cross-sectional tori $\{ s \} \times \Tor$ and its oscillatory component $g^{osc}_t$ having the property that the integral along all such tori vanishes.
We can then express the flow equation (\ref{eq:dhRS}) as a coupled system of equations in $h^{inv}_t = g^{inv}_t - \ov{g}$ and $h^{osc}_t = g^{osc}_t$.

The analysis of the equation for $g^{inv}_t$ is the most crucial part of the proof and is deferred to section \ref{sec:invRF}.
In this section, we will mainly focus on the equation for $h^{osc}_t$.
In subsection \ref{subsec:finalcusp}, it will turn out that this equation describes a strong equilibrium, i.e. solutions are expected to decay to zero rapidly.
Furthermore, the coupling between both equations will be analyzed.

In this section, $C$ will always denote a dynamic constant depending only and continuously on $\delta$ and the geometry of $\Tor$.

\subsection{The invariant and the oscillatory component of the flow} \label{subsec:invosc}
For every tensor-field $h$ on $N$, we define its \emph{invariant component} $h^{inv}$ and its \emph{oscillatory component} $h^{osc}$ by
\[ h^{inv}(s, x_2, \ldots, x_n) = \frac1{\vol \Tor} \int_{\Tor} h(s, x_2, \ldots, x_n) dx_2 \cdots dx_n, \qquad h^{osc} = h- h^{inv}. \]
Note that $h^{inv}$ only depends on $s$, and that $h^{inv}$ and $h^{osc}$ are orthogonal to each other in the $L^2$-sense.
Furthermore, observe that ${\ov g}^{inv} = \ov g$, ${\ov g}^{osc} = 0$.
We can split the flow $(g_t)$ into a sum of the flows $(g^{inv}_t)$ and $(g^{osc}_t)$ and respectively for the perturbation $h_t = g_t - \ov g$, we have the decomposition $h_t = h^{inv}_t + h^{osc}_t$.

Equation (\ref{eq:dhRS}) can be expressed by equations in $h^{inv}_t$ and $h^{osc}_t$:
\begin{subequations}
\begin{alignat}{2}
 \partial_t h^{inv}_t + L h^{inv}_t &= R^{inv}[h^{inv}_t &&+ h^{osc}_t] + \nabla^* S^{inv} [h^{inv}_t + h^{osc}_t] \label{eq:evolhsplita} \\
 \partial_t h^{osc}_t + L h^{osc}_t &= R^{osc}[h^{inv}_t &&+ h^{osc}_t] + \nabla^* S^{osc} [h^{inv}_t + h^{osc}_t]. \label{eq:evolhsplitb}
\end{alignat}
\end{subequations}
Analogously to (\ref{eq:boundRS}), we can derive the following pointwise bounds if we assume $|h|< 0.1$
\begin{subequations}
\begin{alignat}{1}
 |R[h^{inv} + h^{osc}] - R[h^{inv}] | &\leq C ( |h^{osc} | | \nabla h |^2 + | \nabla h^{osc} | | \nabla h | ), \label{eq:Rinvest1} \\
 |S[h^{inv} + h^{osc}] - S[h^{inv}] | &\leq C( |h^{osc}| | \nabla h | + | \nabla h^{osc} |  | h | ). \label{eq:Rinvest2}
\end{alignat}
\end{subequations}

\subsection{The invariant component}
Set $I^{inv}_t = R^{inv}[h^{inv}_t + h^{osc}_t] - R[h^{inv}_t]$ and $J^{inv}_t = S^{inv} [h^{inv}_t + h^{osc}_t] - S[h^{inv}_t]$ and rewrite (\ref{eq:evolhsplita}) as
\begin{equation} \partial_t h^{inv}_t + L h^{inv}_t = R[h^{inv}_t] + \nabla^* S [h^{inv}_t] + I^{inv}_t + \nabla^* J^{inv}_t. \tag{\ref{eq:evolhsplita}$'$} \label{eq:RdTIJ}
\end{equation}

We can view (\ref{eq:RdTIJ}) as a modified Ricci deTurck flow equation with an extra input term $I^{inv}_t + \nabla^* J^{inv}_t$.
Observe that all quantities in this equation are invariant.

The following theorem gives us control over $h^{inv}_t$ in terms of bounds on $h^{inv}_t$ near the parabolic boundary $\partial N \times [0,T) \cup N \times \{ 0 \}$ and certain bounds on $I^{inv}_t$ and $J^{inv}_t$.
We defer its proof to section \ref{sec:invRF}.

\begin{Theorem} \label{Thm:cuspinv}
Let $T > (9\sigma)^2$ be some maximal time and consider an invariant solution $(h^{inv}_t)_{t \in [0,T)}$ to the modified Ricci deTurck flow equation (\ref{eq:RdTIJ}) with an extra ``input term'' $I^{inv}_t + \nabla^* J^{inv}_t$ on $(N, \ov g)$.
Moreover, assume that $h^{inv}_t$ and $\nabla h^{inv}_t$ are uniformly bounded on $N \times [0,T)$ by some constant.\\
Then for every $\delta > 0$ there are constants $\varepsilon_{inv} > 0$ and $C_{inv} < \infty$ depending only on $\delta$ and $n$ such that if
\begin{multline*}
 H = \sup_{\stackrel{\scriptstyle B_{9\sigma}(\partial N) \times [0,T)}{\cup N \times [ 0, (9\sigma)^2 )  }} \big( |h_t^{inv}| + |\nabla h_t^{inv}| \big) 
+ \sup_{(x,t) \in N \times [0,T)} e^{ s(x) + \delta t}  \big( |I_t^{inv}|(x,t)  \\
 + |J_t^{inv}|(x,t) \big) < \varepsilon_{inv},
\end{multline*}
then $\Vert h^{inv}_t \Vert_{L^\infty(N \times [0,T))} \leq C_{inv}H$.
\end{Theorem}

\subsection{The heat kernel on the cusp}
Recall that $N \subset N'$, where $N'$ is the complete hyperbolic cusp.
Let $E = \Sym_2 T^*N'$ be the vector bundle of symmetric $2$-forms over $N'$.
The Einstein operator $L$ is a second order differential operator acting on sections of $E$.
Let $(k_t) \in C^{\infty}(N' \times N' \times \IR_+; E \boxtimes E^*)$ be the heat kernel associated to $L$ on $N'$, i.e. for all $y \in N'$
\[
(\partial_t + L)k_t(\cdot, y) = 0 \qquad \text{and} \qquad
k_t(\cdot, y) \xrightarrow[t \to 0]{} \id_{E_y} \delta_y.
\]
We denote the derivatives of $k_t$ with respect to the first variable by $\nabla_1 k_t$ and those with respect to the second by $\nabla_2 k_t$.
Observe that we have the following symmetry property:
\[ \nabla_1^{m_1} \nabla_2^{m_2} k_t(x,y) = \nabla_2^{m_1} \nabla_1^{m_2} k_t^*(y,x). \]
Moreover, the convolution property holds:
\[ \nabla_1^{m_1} \nabla_2^{m_2} k_{t_1+t_2}(x,y) = \int_{N'} \nabla_1^{m_1} k_{t_1} (x,z) \nabla_2^{m_2} k_{t_2}(z,y) dz. \]

The torus $\Tor$, viewed as a Lie group, acts isometrically on $N'$ and $E$ by multiplication on the $\Tor$-factor.
So the heat kernel $k_t$ is equivariant with respect to this action, i.e. $g_*^{-1} \circ k_t(g.x,g.y) \circ g_* = k_t(x,y)$ for all $g \in \Tor$.
Hence the oscillatory part of $k_t(x,y)$ with respect to $x$ is the same as with respect to $y$ and we can write $k^{osc}_t(x,y)$ without ambiguity.

The following bounds for $k_t$ and $k_t^{osc}$ hold:

\begin{Lemma} \label{Lem:oschk}
\begin{enumerate}[(a)]
 \item We have for $x \in N$ and $s = s(x) \geq 0$
\[ 
\Vert k_t(x, \cdot) \Vert_{L^1([s-\sigma,s+\sigma] \times \Tor \times [0,\sigma^2])},  \;\;
\Vert \nabla_2 k_t(x, \cdot) \Vert_{L^1([s-\sigma,s+\sigma] \times \Tor \times [0,\sigma^2])} < C. 
\]
 \item Let $x, y \in N$. If $t \geq \sigma^2$ or $| s(x) - s(y) | \geq \sigma$, then
\[ |k_t^{osc}|(x,y), \; |\nabla_{2} k_t^{osc}| (x,y) < C \exp( -s(x) - s(y) - (n-2)t ). \]
\end{enumerate}
\end{Lemma}
\begin{proof}
By Kato's inequality and (\ref{eq:Lonhyp})
\begin{equation} (\partial_t - \triangle - 2 ) |k_t| (\cdot, y) \leq 0. \label{eq:Kato} \end{equation}
Hence $|k_t|(x,y) \leq e^{2 t} \Phi_t (x,y)$, where $\Phi_t$ is the \emph{scalar} heat kernel on $N'$.
By the heat kernel estimate from \cite[Corollary 3.1]{LiYau}, we can derive a bound on $\Phi_t$ which implies for $t < 1$
\begin{equation} |k_t|(x,y) < C \big( {\vol B_{\sqrt{t}} (x)} \big)^{-1/2} \big( {\vol B_{\sqrt{t}} (y)} \big)^{-1/2}  \exp \Big( { - \frac{d^2(x,y)}{5t} } \Big). \label{eq:CLY} \end{equation}
Using the a priori derivative estimate from Corollary \ref{Cor:Shi}, we obtain
\begin{equation*} 
|\nabla_1^{m_2} k_t |(x,y) < C_{m_2} t^{-m_2/2} \big( {\vol B_{\sqrt{t}} (x)} \big)^{-1/2} \big( {\vol B_{\sqrt{t}} (y)} \big)^{-1/2}  \exp \Big({ - \frac{d^2(x,y)}{6t} }\Big). 
\end{equation*}
By the symmetry property, the same bounds hold for $|\nabla_2^{m_2} k_t|(x,y)$ and since this expression satisfies the linear equation $(\partial_t + L) \nabla_2^{m_2} k_t(\cdot, y) = 0$, we can apply Corollary \ref{Cor:Shi} again to obtain
\begin{equation} |\nabla_1^{m_1} \nabla_2^{m_2} k_t |(x,y) < C_{m_1,m_2} t^{-(m_1+m_2)/2} \big( {\vol B_{\sqrt{t}} (x)} \big)^{-1/2} \ldots \exp \Big({ - \frac{d^2(x,y)}{8t} } \Big) \label{eq:CLYderder} \end{equation}

Observe that (\ref{eq:Kato}) implies that $e^{-2t} \Vert k_t(\cdot, y) \Vert_{L^1(N')}$ is montonically nonincreasing in $t$.
Moreover its limit as $t \to 0$ is equal to $1$.
Hence the quantity is uniformly bounded by $1$ and by the symmetry property, we get $\Vert k_t(x, \cdot) \Vert_{L^1(N')} \leq C$ for $t < 1$.
This establishes the bound on the first quantity of part (a).

For the bound on $\Vert \nabla_2 k_t(x, \cdot) \Vert_{L^1(\ldots)}$, we have to use (\ref{eq:CLYderder}).
Observe here that for $y \in [s-\sigma,s+\sigma] \times \Tor$ and $t \leq \sigma^2$, we have
\[  \vol B_{\sqrt{t}}(x), \; \vol B_{\sqrt{t}}(y) \geq c \min \{ t^{n/2}, \exp(-(n-1)s) t^{1/2} \}. \]
So if $t^{1/2} \leq \exp(-s)$, we find
\[ \Vert \nabla_2 k_t(x,\cdot) \Vert_{L^1([s-\sigma,s+\sigma] \times \Tor)} \leq C t^{-(n+1)/2} \int_{[s-\sigma,s+\sigma] \times \Tor} \exp \Big({ - \frac{d^2(x,y)}{8t}} \Big) dy \leq C t^{-1/2}. \]
And for $t^{1/2} \geq \exp(-s)$
\begin{multline*} \Vert \nabla_2 k_t(x,\cdot) \Vert_{L^1(\ldots)} \leq C t^{-1} \int_{[s-\sigma,s+\sigma] \times \Tor} \exp \Big({- \frac{(s(x)-s(y))^2}{8t} }\Big) \exp( (n-1) s) dy \\
\leq C t^{-1} \int_{s-\sigma}^{s+\sigma} \exp \Big({-\frac{(s(x)-s')^2}{8t}}\Big) ds' \leq C t^{-1/2}.
\end{multline*}
This establishes the bound on the second quantity of part (a).

Integrating (\ref{eq:CLY}) over $N'$ for $t = \sigma^2/10$ and using
\begin{equation} \label{eq:volest} 
\vol B_{\sqrt{t}} (z) \geq c \min \{ t^{n/2}, \exp (-(n-1)s(z)) t^{1/2} \} 
\end{equation}
gives us furthermore for $y \in N$
\[ \Vert k_{\sigma^2/10}(\cdot, y) \Vert_{L^2(N')} < C \exp \bigl( \sfrac12 (n-1) s(y) \bigr). \]
By (\ref{eq:Lgeqnm2})
\[ \partial_t \Vert k_t(\cdot, y) \Vert_{L^2(N')}^2 = - 2 \langle L k_t(\cdot,y), k_t(\cdot,y) \rangle
\leq -2(n-2) \Vert k_t(\cdot, y) \Vert_{L^2(N')}^2. \]
So for $t \geq \sigma^2/10$, we have
\[ \Vert k_t(\cdot, y) \Vert_{L^2(N')} = \Vert k_t(y, \cdot) \Vert_{L^2(N')} < C \exp \bigl( \sfrac12(n-1) s(y) - (n-2) t \bigr). \]

Using the convolution property, we can derive an $L^\infty$-bound from this $L^2$-bound for $t \geq \sigma^2/5$ and $x, y \in N$
\begin{multline}
 |k_t|(x,y) = \left| \int_N k_{t/2}(x,z) k_{t/2}(z,y) dz \right| \leq \Vert k_{t/2} (x, \cdot) \Vert_{L^2(M)} \Vert k_{t/2}(\cdot, y) \Vert_{L^2(M)} \\ 
 < C \exp \bigl( \sfrac12 (n-1) ( s(x) + s(y) ) - (n-2)t \bigr). \label{eq:ktnminus2}
\end{multline}
The a priori estimate from Corollary \ref{Cor:Shi} now gives us bounds on the derivatives of $k_t$ for $t \geq \sigma^2 / 2$:
\[ |\nabla_1^{m_2} k_t|(x,y) < C_{m_2} \exp \left( \sfrac12 (n-1) ( s(x) + s(y) ) - (n-2)t \right) \]
By the symmetry property and Corollary \ref{Cor:Shi} again we get for $t \geq \sigma^2$ (see the derivation of (\ref{eq:CLYderder}))
\begin{equation} \label{eq:m1m2k}
|\nabla_1^{m_1} \nabla_2^{m_2} k_t|(x,y) < C_{m_1, m_2} \exp \left( \sfrac12 (n-1) ( s(x) + s(y) ) - (n-2)t \right). 
\end{equation}

We now apply the following trick to get essentially better bounds on the oscillatory part of $k_t$:
For any tensor field $h$ we can estimate its oscillatory component $h^{osc}$ by its higher derivatives along cross-sectional tori $\{ s \} \times \Tor$:
\begin{equation} \Vert h^{osc} \Vert_{L^\infty(\{ s \} \times \Tor)} \leq C_m \exp(-m s) \Vert \nabla^m h \Vert_{L^\infty( \{ s \} \times \Tor)}. \label{eq:trick} \end{equation}
This follows by $m$-fold integration and the fact that $\diam ( \{ s \} \times \Tor ) < C e^{- s}$.

Now consider $x, y \in N$ and assume $s(y) \geq s(x)$ (if not, interchange $x$ and $y$).
Assume first that $t \geq \sigma^2$.
We apply (\ref{eq:trick}) to (\ref{eq:m1m2k}) along $\{ s(y) \} \times \Tor$ with $m = n+1$ and get
\[ | \nabla_1^{m_1} \nabla_2^{m_2} k^{osc}_t|(x,y) < C_{m_1, m_2} \exp \left( - s(x) - s(y) - (n-2)t \right). \]
On the other hand, if $t < \sigma^2$, but $| s(x) - s(y) | \geq \sigma$, we apply the same argument to (\ref{eq:CLYderder}) and (\ref{eq:volest}) and obtain
\begin{multline*} 
| \nabla_1^{m_1} \nabla_2^{m_2} k^{osc}_t|(x,y) < C_{m_1, m_2} t^{-(n+m_1+m_2+n+1)/2} \exp \left( - \frac{\sigma^2}{8t} - s(x) - s(y) \right) \\
 < C_{m_1, m_2} \exp ( - s(x) - s(y) )
\end{multline*}
Hence, we have established part (b) of the Lemma.
\end{proof}

\subsection{Representing $h_t$} \label{subsec:representh}
We can use the heat kernel $k_t$ to represent $h_t = g_t - \ov g$.
Choose a smooth function $\widetilde\varphi : \IR \to [0,1]$ with $\widetilde\varphi \equiv 0$ on $(-\infty, \frac12]$ and $\widetilde\varphi \equiv 1$ on $[1,\infty)$ and define $\varphi \in C_0(N)$ by $\varphi(x) = \widetilde\varphi(s(x)/\sigma)$.

Let $(x_0,t_0) \in N \times [0,T)$.
Since we assumed $h_t$ and $\nabla h_t$ to be uniformly bounded over $N \times [0,T)$, we can use integration by parts and (\ref{eq:dhRS}) to find that for $0 \leq t < t_0$
\begin{multline*}
 \partial_t \int_N \varphi^2  k_{t_0-t}(x_0,x) h_t(x) dx 
=  \int_N \varphi^2 \bigl[ L^*_x k_{t_0-t}(x_0,x) h_t(x) \\
  \hfill +  k_{t_0-t}(x_0,x) \bigl(-L h_t + R[h_t] + \nabla^* S[h_t] \bigr)(x) \bigr] dx \;\; \\
 \qquad = \int_N \varphi^2 \bigl( k_{t_0-t}(x_0,x) R[h_t](x) + \nabla_2 k_{t_0-t} (x_0,x) * S[h_t](x) \bigr) dx \hfill \\
 \qquad + \int_N \bigl[ - 2(\varphi \triangle \varphi + |\nabla \varphi|^2) k_{t_0-t}(x_0,x) h_t(x) - 4 \varphi \nabla \varphi * k_{t_0-t}(x_0,x) * \nabla h_t(x) \\
 \hfill + 2 \varphi \nabla \varphi * k_{t_0-t}(x_0,x) * S[h_t](x) \bigr] dx .
\end{multline*}
Integrating this over $t$ from $0$ to $t_0$ yields $h_t = h_t^* + h_t^{**}$ where
\begin{equation} \label{eq:hstar}
 h_{t_0}^*(x_0) = \int_{N\times[0,t_0]} \varphi^2 \bigl( k_{t_0-t}(x_0,x) R[h_t] + \nabla_2 k_{t_0-t} (x_0,x) * S[h_t] \bigr) dx dt 
\end{equation}
and
\begin{multline} 
h_{t_0}^{**}(x_0) = (1-\varphi^2) h_{t_0}(x_0) + \int_N \varphi^2 k_{t_0}(x_0,x) h_0(x) dx \displaybreak[1] \\
\qquad  + \int_{B_\sigma (\partial N) \times [0,t_0]} \bigl[ - 2(\varphi \triangle \varphi + |\nabla \varphi|^2) k_{t_0-t}(x_0,x) h_t(x) \hfill \displaybreak[1] \\
\qquad - 4 \varphi \nabla \varphi * k_{t_0-t}(x_0,x) * \nabla h_t(x) + 2 \varphi \nabla \varphi * k_{t_0-t}(x_0,x) * S[h_t](x) \bigr] dx dt. \hfill \label{eq:hstarstar}
\end{multline}
The following Lemma gives a bound on $(h_t^{**})^{osc}$ in terms of $H$.
\begin{Lemma} \label{Lem:hstarstar}
 Assume that $|h| < 0.1$ everywhere on $N \times [0,T)$, $H$ is defined as in Theorem \ref{Thm:RdToncusp} and $\delta < n-2$.
 
 If $(x_0,t_0) \in N \times [0,T)$ with $s(x_0) \geq 2\sigma$ and $t_0 \geq \sigma^2$, then
 \[ | (h^{**}_{t_0})^{osc} | (x_0) \leq C H \exp (- s(x_0) - \delta t_0). \]
\end{Lemma}
\begin{proof}
Obviously, the the first term in (\ref{eq:hstarstar}) vanishes.
As for the second term we use Lemma \ref{Lem:oschk} and the bound $|h_0| \leq CH$ to find that
\begin{multline*} \left| \int_N \varphi^2 k_{t_0}^{osc}(x_0,x) h_0(x) dx \right| \leq C H \int_N \exp(-s(x_0) - s(x) - (n-2) t_0 ) dx \\ \leq C H \exp(-s(x_0) - (n-2)t_0). 
\end{multline*}
Now concerning the third term we use the bounds for $h$ and $\nabla h$ on $B_\sigma (\partial N) \times [0,T)$ and (\ref{eq:boundRS}) to conclude $|S_t| \leq CH e^{-\delta t}$ on on $B_\sigma (\partial N) \times [0,T)$.
Thus the third term is bounded by
\begin{multline*} 
\int_{B_\sigma(\partial N) \times [0,t_0]} C H \exp(-s(x_0) - (n-2)(t_0-t) - \delta t) dxdt \\
 \leq CH \exp (-s(x_0) - \delta t_0). \qedhere
\end{multline*}
\end{proof}

\subsection{Final argument} \label{subsec:finalcusp}
We can now use these results and Theorem \ref{Thm:cuspinv} to prove Theorem \ref{Thm:RdToncusp}.
\begin{proof}[Proof of Theorem \ref{Thm:RdToncusp}.]
Assume in the following that at least $H < 0.1$.
For any $0 < T' \leq T$, we set
\[
  \theta_{T'} = \Vert h \Vert_{L^\infty(N \times [0,T'))} \quad \text{and} \quad \eta_{T'} = \sup_{(x,t) \in N \times [0,T')} e^{s(x) + \delta t} | h^{osc}| (x,t) .
\]
Observe that by the bound on $\nabla h$ in $N \times [0, \sigma^2]$ and (\ref{eq:trick}), we have $\eta_{\sigma^2} < C H$.

By Corollary \ref{Cor:Shi} and the hypothesis of the theorem, there is some universal $\varepsilon_0 > 0$ such that
\begin{equation} \label{eq:thetaShiestimate}
 \text{if} \quad \theta_{T'} < \varepsilon_0, \quad \text{then} \quad \Vert \nabla h \Vert_{L^\infty (N \times [0,T'))} \leq C ( \theta_{T'} + H).
\end{equation}
Next, we prove that (after possibly reducing $\varepsilon_0$), we have
\begin{equation} \label{eq:etaShiestimate}
 \text{if} \quad \theta_{T'}, \eta_{T'} < \varepsilon_0, \quad \text{then} \quad | \nabla h^{osc} |(x,t) \leq C e^{-s(x)-\delta t} (\eta_{T'} + H)
\end{equation}
for all $(x,t) \in N \times [\sigma^2,T')$.

In order to do this, we use the following trick:
Let $h'_t = g^* h_t$ be the pullback of $h_t$ via an isometry $g : N \to N$ arising from a translation along $\Tor$.
Then $(h'_t + \ov{g})$ still satisfies the modified Ricci deTurck flow equation (\ref{eq:MRdTflow}).
In local coordinates, we can write as in (\ref{eq:Shiequ}) of Proposition \ref{Prop:Shi}
\[ (\partial_t - L) h = f_1(x, h) \cdot h  + f_2 (x,h) \cdot \nabla h + f_3(x,h) \cdot \nabla h \otimes \nabla h + f_4(x,h) \cdot h \otimes \nabla^2 h \]
and the same for $h'_t$.
Let $b > 0$ be a constant that we will determine later and set $d_t = b(h_t - h'_t)$.
Then $d_t$ satisfies the following evolution equation:
\begin{multline*}
 (\partial_t - L) d = f_1' (x, h, h') \cdot d + f_2 (x, h) \cdot \nabla d + (f_2' (x, h, h') \cdot d) \cdot \nabla h' \displaybreak[1] \\ 
 + f_3 (x, h) \cdot \nabla d \otimes \nabla h + f_3 (x, h) \cdot \nabla h' \otimes \nabla d \displaybreak[1] \\
 + (f_3' (x, h, h') \cdot  d) \cdot \nabla h' \otimes \nabla h' \displaybreak[1] \\
+ f_4(x,h) \cdot d \otimes \nabla^2 h + f_4 (x, h) \cdot h' \otimes \nabla^2 d + f_4' (x,h,h') \cdot d \otimes \nabla^2 h'.
\end{multline*}
Observe here that the quantities $f_1', \ldots, f_4''$ do not depend on $b$.

The calculation above shows that the vector $(h_t, h'_t, d_t)$ satisfies a parabolic equation of the form (\ref{eq:Shiequ}) (here, we group $f_2, f_2' \cdot d$ and $f_3, f_3, f'_3 \cdot d$ and $f_4, f_4, f'_4$).
Hence, we can apply Proposition \ref{Prop:Shi} and the reasoning of Corollary \ref{Cor:Shi} to obtain that if $\widetilde{H} = \max \{ \Vert h \Vert_{L^\infty (\Omega')}, \Vert h' \Vert_{L^\infty(\Omega')}, b \Vert h - h' \Vert_{L^\infty(\Omega')} \} \leq \widetilde{\varepsilon}$, then amongst others $b \Vert \nabla h - \nabla h' \Vert_{L^\infty(\Omega)} < \widetilde{C} \widetilde{H}$.
Choosing $b = \widetilde{\varepsilon} \Vert h-h' \Vert_{L^\infty(\Omega')}^{-1}$ and averaging over all pullbacks $h' = g^* h$ of isometries $g : N \to N$, yields
\[ \text{if} \quad \Vert h \Vert_{L^\infty (\Omega')}, \Vert h' \Vert_{L^\infty(\Omega')} < \widetilde{\varepsilon}, \quad \text{then} \quad \Vert \nabla h^{osc} \Vert_{L^\infty(\Omega)} < C \Vert h^{osc} \Vert_{L^\infty(\Omega')}. \]
With this estimate, we can establish (\ref{eq:etaShiestimate}).

Using (\ref{eq:thetaShiestimate}) and (\ref{eq:etaShiestimate}), we can bound $I^{inv}$ and $J^{inv}$ by (\ref{eq:Rinvest1}) and (\ref{eq:Rinvest2}) if $\theta_{T'}, \eta_{T'} < \varepsilon_0$:
\[ |I^{inv}|(x,t), \; |J^{inv}|(x,t) \leq C e^{-s(x) - \delta t} (\eta_{T'} + H)(\theta_{T'} + H) \]
for all $(x,t) \in N \times [\sigma^2,T')$.

Now Theorem \ref{Thm:cuspinv} applied to $N \times [\sigma^2, \infty)$ implies that if $H + C (\eta_{T'} + H) ( \theta_{T'} + H) < \varepsilon_{inv}$, we have a uniform bound on $h^{inv}$ and together with the bound $\Vert h^{osc} \Vert_{L^\infty(N \times [0,T'))} \leq \eta_{T'}$ and the hypothesis of the theorem this means
\begin{equation} \theta_{T'} \leq C (\eta_{T'} + H) (\theta_{T'} + H) + C H + C \eta_{T'} \leq C (\eta_{T'} + \theta_{T'}^2 + H). \label{eq:theta} \end{equation}

Next, we show that if $\theta_{T'}, \eta_{T'} < \varepsilon_0$, then $\eta_{T'}$ satisfies the bound
\begin{equation} \eta_{T'} \leq C (\eta_{T'}^2 + \theta_{T'}^2 + H). \label{eq:eta} \end{equation}
So we will need to bound $h^{osc}(x_0,t_0)$ for all $(x_0,t_0) \in N \times [0,T')$.
First observe that if $t_0 < (10 \sigma)^2$, then the quantity is bounded by $C H e^{-s(x_0)}$ what follows from (\ref{eq:trick}) and the bound $|\nabla h| \leq C H$ on $N \times [0,(10 \sigma)^2)$.
Furthermore $h^{osc}(x_0,t_0)$ is also bounded by $C H e^{-\delta t_0}$ for $(x_0, t_0) \in B_{10 \sigma}(\partial N) \times [0,T')$.

Now assume $s_0 = s(x_0) \geq 10 \sigma$ and $t_0 \geq (10\sigma)^2$.
Using the decomposition $h^{osc} = (h^*)^{osc} + (h^{**})^{osc}$ corresponding to $N \times [\sigma^2, T)$ and Lemma \ref{Lem:hstarstar}, we find that it suffices to bound $ (h^*)^{osc} (x_0,t_0)$.
If we take the oscillatory component on both sides of (\ref{eq:hstar}), we find that the component $k_{t_0-t}^{osc}(x_0,x) R^{inv}[h_t] + \nabla k_{t_0-t}^{osc} (x_0,x) S^{inv} [h_t]$ cancels out by the integration and we are left with
\[  (h_{t_0}^*)^{osc} (x_0) = \int_{N \times[\sigma^2,t_0]} \varphi^2 \bigl( k_{t_0-t}^{osc}(x_0,x) R^{osc}[h_t] + \nabla k_{t_0-t}^{osc} (x_0,x) S^{osc} [h_t] \bigr) dx dt. \]
We first estimate the terms $R^{osc}[h_t]$ and $S^{osc}[h_t]$ appropriately.
By (\ref{eq:Rinvest1}), (\ref{eq:Rinvest2}) as well as (\ref{eq:thetaShiestimate}), (\ref{eq:etaShiestimate}), we can estimate for $(x,t) \in N \times [0,T')$
\begin{alignat*}{1} 
|R^{osc}[h]|(x,t) &= | (R[h] - R[h^{inv}])^{osc} |(x,t) \leq C e^{-s(x) - \delta t}  (\eta_{T'} + H)( \theta_{T'} + H ), \displaybreak[1] \\
|S^{osc}[h]|(x,t) &= | (S[h] - S[h^{inv}])^{osc} |(x,t) \leq C e^{-s(x) - \delta t}  (\eta_{T'} + H)( \theta_{T'} + H ).
\end{alignat*}
Now we split the domain $N \times [\sigma^2, t_0]$ into $\Omega = [s_0 - \sigma, s_0 + \sigma] \times \Tor \times [t_0-\sigma^2,t_0]$ and $N \times [\sigma^2 ,t_0] \setminus \Omega$ and use Lemma \ref{Lem:oschk} to conclude
\begin{alignat*}{1}
 |(h_{t_0}^*)^{osc}|(x_0) &\leq \int_{\Omega} \varphi^2 \bigl( k_{t_0-t}(x_0,x) |R^{osc}| [h_t] + \nabla k_{t_0-t}(x_0,x) | S^{osc} |[h_t] \bigr) dx dt \\ 
& \qquad + \int_{N \times [\sigma^2,t_0] \setminus \Omega} C e^{-s_0 - s(x) - (n-2)(t_0-t)} \bigl( |R^{osc}| [h_t] + |S^{osc}| [h_t] \bigr) dx dt \displaybreak[1] \\
&  \leq C e^{-s_0-\delta t_0} (\eta_{T'} + H)(\theta_{T'} + H) + C  (\eta_{T'} + H)(\theta_{T'} + H) \times   \\
& \qquad  \int_0^{t_0} \int_0^\infty e^{-s_0 - s - (n-2)(t_0-t)} e^{-(n-1)s} e^{-s - \delta t} ds dt \displaybreak[1] \\
&  \leq C e^{-s_0 - \delta t_0} (\eta_{T'} + H)(\theta_{T'} + H).
\end{alignat*}
This establishes (\ref{eq:eta}).

Putting (\ref{eq:eta}) and (\ref{eq:theta}) together (and possibly reducing $\varepsilon_0$ again), we conclude that if $\theta_{T'} + \eta_{T'} + H < \varepsilon_0$, then we have
\[ \theta_{T'} + \eta_{T'} \leq C_0 (\theta_{T'} + \eta_{T'})^2 + C_0 H \]
for some uniform constant $C_0$ which is independent of $T'$.
Moreover by the hypothesis of the Theorem, we have $\theta_{\sigma^2} + \eta_{\sigma^2} \leq C_1 H$.
Set $\varepsilon = \min \{ (2 C_0)^{-1}, \frac{\varepsilon_0}{2} \}$, $\varepsilon_{cusp} = \min \{ (2C_0)^{-1}, C_1^{-1}, 1 \} \varepsilon$ and assume that $H < \varepsilon_{cusp}$.
Hence $\theta_{\sigma^2} + \eta_{\sigma^2} < \varepsilon$.
Now if $\theta_T + \eta_T \geq \varepsilon$, then there would be some time $T' \in (\sigma^2, T]$ with $\theta_{T'} + \eta_{T'} = \varepsilon$ (note that we can use higher derivative estimates and (\ref{eq:trick}) to conclude that $\eta_{T'}$ is continuous in $T'$) and hence
\[ \varepsilon < C_0 (\varepsilon^2 + \varepsilon_{cusp}) \leq \tfrac12 \varepsilon + \tfrac12 \varepsilon, \]
a contradiction.
So $\theta_T + \eta_T < \varepsilon$ and we conclude
\[ \theta_T + \eta_T \leq 2 C_0 H. \qedhere \]
\end{proof}

\section{Invariant modified Ricci deTurck flow on the cusp} \label{sec:invRF}
\subsection{Calculations} \label{subsec:Calculations}
In this section we are concerned with the proof of Theorem \ref{Thm:cuspinv}.
As in the last section or in subsection \ref{subsec:Hypmfs}, denote by $(N = [0,\infty) \times \Tor, \ov g)$ a hyperbolic cusp with coordinates $s=x_1, x_2, \ldots, x_n$ and metric 
\[ \ov g = ds^2 + e^{-2s} (dx_2^2 + \ldots + dx_n^2). \]
Since we will only be dealing with ($\Tor$-)invariant tensor fields in this section, we will abbreviate the perturbation $(h^{inv}_t)$ resp. the metric $(g^{inv}_t) = (\ov g + h^{inv}_t)$ by $(h_t)$ resp. $(g_t)$ and the input terms $(I^{inv}_t)$ resp. $(J^{inv}_t)$ by $(I_t)$ resp. $(J_t)$.
Furthermore, we always assume $|h_t|<0.1$.

Our first goal is to express the modified Ricci deTurck flow equation (\ref{eq:RdTIJ}) in terms of the coordinate entries of $h_t$ as well as $I_t$ and $J_t$.
Recall from (\ref{eq:MRdTflow}), that this equation can be written as
\begin{equation} \label{eq:evoleq}
\dot h_t = - 2 \Ric_{g_t} - 2(n-1) g_t - \mathcal{L}_{X_{\ov g}(g_t)} g_t + I_t + \nabla^* J_t. 
\end{equation}

For the moment fix some time $t$ and write $h = h_t$, $I = I_t$ and $J = J_t$.
We express $h$ as
\[ h = A d s^2 + e^{-s} V_i (d x_i d s + d s d x_i) + e^{-2s} M_{ij} d x_i d x_j \]
where $A, V_i, M_{ij}$ only depend on $s$.
We can visualize $h$ in block matrix form
\[ h = \Mat{A}{e^{-s} V_i}{e^{-2s} M_{ij}} \]
where we omit the lower left entry, since it equals the transpose of the upper right one.

The first covariant derivatives of $h$ are ($k > 1$, a prime will always denote differentiation with respect to $s$)
\begin{alignat*}{1}
\nabla_1 h &= \Mat{A'}{e^{-s} V'_i}{e^{-2s} M'_{ij}} \displaybreak[1] \\ 
e^{s} \nabla_k h &= \Mat{2 V_k}{e^{-s} (M_{ki} - \delta_{ki} A)}{e^{-2s} ( - \delta_{ki} V_j - \delta_{kj} V_i ) }
\end{alignat*}
and the second covariant derivatives of $h$ are ($k,l > 1$, $k \not= l$)
\begin{alignat*}{1}
\nabla_{11} h &= \Mat{A''}{e^{-s} V''_i}{e^{-2s} M''_{ij}} \\
e^{s} \nabla_{1k} h &= \Mat{2 V'_k}{e^{-s}(M'_{ki} - \delta_{ki} A')}{e^{-2s}(-\delta_{ki} V'_j - \delta_{kj} V'_i)} \\
e^{s} \nabla_{k1} h &= \Mat{2 V'_k+2 V_k}{e^{-s}( M'_{ki} - \delta_{ki} A' + M_{ki} - \delta_{ki} A)}{e^{-2s}(-\delta_{ki} V'_j - \delta_{kj} V'_i - \delta_{ki} V_j - \delta_{kj} V_i)} \\
e^{2s} \nabla_{lk} h &= \Mat{2 M_{kl} }{e^{-s} ( - \delta_{ki} V_l - 2 \delta_{li} V_k)}{e^{-2s} (- \delta_{li} M_{kj} - \delta_{lj} M_{ki} + \delta_{li} \delta_{kj} A + \delta_{lj} \delta_{ki} A)} \\
e^{2s} \nabla_{kk} h &= \Mat{2 M_{kk} - 2 A - A'}{e^{-s}(-V_i - 3 \delta_{ki} V_k - V'_i)}{e^{-2s} ( - \delta_{ki} M_{kj} - \delta_{kj} M_{ki} + 2 \delta_{ki} \delta_{kj} A - M'_{ij})}
\end{alignat*}

Recall from subsection \ref{subsec:MRdT} and (\ref{eq:Lonhyp}), that we can express the Ricci curvature of $g$ in terms of $h$ by the following formula:
\begin{equation} \label{eq:Ric2}
\begin{split} 
 2 \Ric_{ab} &= - 2(n-1) \ov g_{ab} - 2 n h_{ab} + 2 \ov g^{uv} h_{uv} \ov g_{ab} \\
&\qquad + \ov g^{uv} (\nabla^2_{au} h_{bv} + \nabla^2_{bu} h_{av} - \nabla^2_{uv} h_{ab} - \nabla^2_{ab} h_{uv}) \\
 &\qquad\qquad + (g^{uv} - \ov g^{uv}) (\nabla^2_{ua} h_{bv} + \nabla^2_{ub} h_{av} - \nabla^2_{uv} h_{ab}  - \nabla^2_{ab} h_{uv})\\
 &\qquad + g^{uv} g^{pq} (\nabla_u h_{pa} \nabla_v h_{qb} - \nabla_p h_{ua} \nabla_v h_{qb} + \sfrac12 \nabla_a h_{up} \nabla_b h_{vq}) \\
 &\qquad + g^{uv}( - \nabla_u h_{vp} + \sfrac12 \nabla_p h_{uv}) g^{pq} (\nabla_a h_{qb} + \nabla_b h_{qa} - \nabla_q h_{ab}) 
\end{split}
\end{equation}
Moreover, by (\ref{eq:Lieder})
\begin{multline*} 
\mathcal{L}_{ab} = (\mathcal{L}_{X_{\ov{g}}(g)} g)_{ab} = X^u \nabla_u h_{ab} + g_{au} \nabla_b X^u + g_{bu} \nabla_a X^u \\
 \text{where} \qquad X^u =  \ov{g}^{uv} \ov{g}^{pq} (- \nabla_p (\log g)_{qv} + \sfrac12 \nabla_v (\log g)_{pq}).
\end{multline*}
It is clear that the $e^{-s}$-terms in both equations cancel in such a way that there is no such factor in the expression for $\Ric_{11}$ and $\mathcal{L}_{11}$, an $e^{-s}$ factor in the expression for $\Ric_{1b} = \Ric_{b1}$ and $\mathcal{L}_{1b} = \mathcal{L}_{b1}, (b>1)$ and an $e^{-2s}$ factor in the expression for $\Ric_{ab}$ and $\mathcal{L}_{ab}, (a,b>1)$.
So without loss of generality, we can simplify our calculations by considering the case $s=0$. 

We will only be interested in the structure of the evolution equation (\ref{eq:evoleq}) for $(h_t)$ rather than its explicit terms.
Our idea is that $M$ will be the main term in the nonlinear part and the influence of $A, V$ is very small.
Having that in mind, we decompose
\[ h = \widehat{h} + \check{h} = \Mat{0}{0}{M} + \Mat{A}{V}{0} \quad \text{and} \quad \widehat g = \ov g + \widehat h. \]
Let $\widehat{\Ric}_{ab}$ be the Ricci tensor corresponding to $\widehat g$.
Note that for symmetry reasons $\widehat{\Ric}_{1b} = 0$ for $b > 1$.
In the first step, we estimate $\Ric_{ab} - \widehat{\Ric}_{ab}$.
This difference has the following algebraic structure: It is a sum of terms $\mathcal{X}$ which can be categorized into the following types
\begin{enumerate}[(i)]
\item $\mathcal{X}$ doesn't depend on any derivative of $A, V, M$.
If $\mathcal{X}$ vanishes of order $1$, say for $A = 0$, then we write $\mathcal{X}=1 * A$.
This implies $|\mathcal{X}| \leq C |A|$.
If it only vanishes for $A=V=0$, we write in a sloppy way $\mathcal{X}=1*(A+V)$, meaning $|\mathcal{X}| \leq C(|A|+|V|)$ etc.
If $\mathcal{X}$ vanishes of order $2$, e.g. if $|\mathcal{X}| \leq C |A| |V|$, we write $\mathcal{X} = A*V$.
\item $\mathcal{X}$ depends linearly on $A', V', M'$, but the coefficients of this linear form might depend nonlinearly on $A, V, M$.
We will abbreviate those terms by $1 * A', 1 * V', 1 * M'$ or just sloppy by $1 * (A' + V' + M')$.
If all coefficients even vanish for, say $A=0$, we write $\mathcal{X} = A * A'$ etc.
\item $\mathcal{X}$ depends bilinearly on $A', V', M'$, but the coefficients might depend nonlinearly on $A, V, M$.
We will abbreviate those terms by $A' * A', A' * V', \ldots, M' * M'$ or more general by $(A' + V' + M')*(A' + V' + M')$.
\item $\mathcal{X}$ depends linearly on $A'', V'', M''$, but the coefficients might depend nonlinearly on $A, V, M$.
\end{enumerate}

We will first determine all terms of type (iv) in $\Ric_{ab} - \widehat{\Ric}_{ab}$.
Those are only produced whenever there is a $\nabla^2_{11} h_{ij}$ term, so ($a,b > 1$)
\begin{alignat*}{1}
 2 \Ric_{11} - 2 \widehat{\Ric}_{11} &\mathop{\equiv}_{\text{(iv) terms}} 2 g^{1v} \nabla_{11}^2 h_{1v} - g^{11} \nabla^2_{11} h_{11} - g^{uv} \nabla^2_{11} h_{uv}   \\
 &\qquad\quad -2 \widehat g^{1v} \nabla_{11}^2 \widehat h_{1v} + \widehat g^{11} \nabla^2_{11} \widehat h_{11} + \widehat g^{uv} \nabla^2_{11} \widehat h_{uv} \displaybreak[1]  \\ 
 &\qquad =  \sum_{u,v =2}^n  (\widehat g^{uv} - g^{uv}) M''_{uv} \\ \displaybreak[1]
 &\qquad = \biggl( \sum_{u,v =2}^n  (\widehat g^{uv} - g^{uv}) M'_{uv} \biggr)' -  \sum_{u,v =2}^n  (\widehat g^{uv} - g^{uv})' M'_{uv} \displaybreak[2] \\
 2 \Ric_{1b} - 2 \widehat \Ric_{1b} &\mathop{\equiv}_{\text{(iv) terms}} (g^{1v} \nabla^2_{11} h_{bv} - g^{11} \nabla^2_{11} h_{b1}) = \sum_{v=2}^n g^{1v} M''_{bv} \\
 &\qquad = \biggl( \sum_{v=2}^n g^{1v} M'_{bv} \biggr)' - \sum_{v=2}^n (g^{1v})' M'_{bv} \displaybreak[2] \\
 2 \Ric_{ab} - 2 \widehat \Ric_{ab} &\mathop{\equiv}_{\text{(iv) terms}} - g^{11} \nabla^2_{11} h_{ab} + \widehat g^{11} \nabla^2_{11} \widehat h_{ab} = (1 - g^{11}) M''_{ab} \\ 
 &\qquad = \left( (1- g^{11}) M'_{ab} \right)' + (g^{11})' M'_{ab}
\end{alignat*}
Now observe that if at some point we have $A=A'=A''=0$ and $V = V' = V'' =0$, then $\Ric_{ab} - \widehat{\Ric}_{ab} = 0$.
So the sum of all terms of type (ii) in $\Ric_{ab} - \widehat{\Ric}_{ab}$ which are of the form $1*M'$, is even of the form $(A+V)*M'$.
Hence for $a,b \geq 1$
\begin{multline*}
 \Ric_{ab} - \widehat \Ric_{ab} = S_1' + 1*(A+V) + 1*(A'+V') + (A+V)* M' \\
 + (A' + V' + M') * (A' + V' + M').
\end{multline*}
where the divergence term has the form $S_1 =  (A+V)*M'$.

Secondly, we express $\widehat \Ric_{ab}$ in terms of $M$.
We use again (\ref{eq:Ric2}) and substiute $h$ and $g$ by $\widehat h$ and $\widehat g$.
Denote by $\widehat T^1_{ab}, \widehat T^2_{ab}, \widehat T^3_{ab}$ the expression in the first three lines, the fourth line and the fifth line on the right hand side.
Then we compute ($a,b > 1$)
\begin{alignat*}{1}
 \widehat T^1_{11} &= -2(n-1) + (E+M)^{uv} (- M''_{uv} + 2 M'_{uv} ) \\
 \widehat T^1_{ab} &= -2(n-1)E_{ab} - M''_{ab} + 2 M'_{ab} - 2 M_{ab} + (E+M)^{uv} \\
  & \qquad\qquad  ( - 2 E_{uv} M_{ab} + 2 E_{au} M_{vb} - E_{ua} M'_{vb} - E_{ub} M'_{va} + E_{uv} M'_{ab} + E_{ab} M'_{uv} ) \\
  \widehat T^2_{11} &= \sfrac12 (E+M)^{uv} (E+M)^{pq} M'_{up} M'_{vq} = M' * M' \\
  \widehat T^2_{ab} &= (E+M)^{uv} \left( M'_{au} M'_{bv} - M'_{au} M_{vb} - M'_{bu} M_{va} + 2 M_{au} M_{bv} \right) \\
  \widehat T^3_{11} &= 0 \\
  \widehat T^3_{ab} &= -(2 M_{ab} - M'_{ab} ) (E+M)^{uv} (M_{uv} - \sfrac12 M'_{uv}) 
\end{alignat*}
So $2\widehat\Ric = \widehat T^1 + \widehat T^2 + \widehat T^3$ equals
\begin{alignat*}{1}
  2\widehat\Ric_{11} &= - 2(n-1) - \tr (E+M)^{-1} M'' + 2 \tr (E+M)^{-1} M' + M' * M' \displaybreak[1]  \\
  2\widehat\Ric_{ab} &= -2(n-1)(E+M) - M'' + (n-1) M' \\
  & \qquad + \left( \tr (E+M)^{-1} M' \right) (E + M) + M' * M'
\end{alignat*}

We will now carry out the same analysis for the Lie-derivative term $\mathcal{L}_{ab}$.
Set $\widehat{\mathcal{L}}_{ab} = (\mathcal{L}_{X_{\ov{g}}(\widehat g)} \widehat g)_{ab}$.
It is easy to see that $\mathcal{L}_{ab}$ and $\widehat{\mathcal{L}}_{ab}$ can also be expressed as a sum of terms of type (i)-(iv).
Observe also, that terms of type (iv) only occur in $\mathcal{L}_{ab} - \widehat{\mathcal{L}}_{ab}$ if $a$ or $b$ equals $1$.
In this case we determine $(b>1)$
\begin{alignat*}{1}
\mathcal{L}_{11} - \widehat{\mathcal{L}}_{11} &\hspace{-2mm}\mathop{\equiv}_{\text{(iv) terms}} \hspace{-3mm} - 2 g_{1u} \nabla^2_{11} (\log g)_{1u} + g_{11} \nabla^2_{11} (\log g)_{pp} + 2 \nabla^2_{11} (\log \widehat g)_{11} - \nabla^2_{11} (\log \widehat g)_{pp} \\
&  = - 2 (1+A) (\log g)''_{11} - 2 V_u (\log g)''_{1u} + A \tr (\log g)'' + \tr (\log g - \log \widehat g)'' \displaybreak[0] \\
& = \bigr(-2 (\log g)'_{11} + 2A' - 2 A (\log g)'_{11} - 2 V_u (\log g)'_{1u} \\
& \qquad + A \tr (\log g)' + \tr (\log g - \log \widehat g)' - A' \bigr)' \\
& \qquad - A'' + (A'+V'+M')*(A'+V'+M') \displaybreak[3]  \\
\mathcal{L}_{1b} - \widehat{\mathcal{L}}_{1b} &\hspace{-2mm}\mathop{\equiv}_{\text{(iv) terms}} \hspace{-3mm} - g_{bu} \nabla^2_{11} (\log g)_{1u} + \tfrac12 g_{b1} \nabla^2_{11} (\log g)_{pp} \\
&  = - V_b (\log g)''_{11} - (E+M)_{bu} (\log g)''_{1u} + \tfrac12 V_b \tr (\log g)'' \displaybreak[0] \\
& = \left( - V_b (\log g)'_{11} - M_{bu} (\log g)'_{1u} - (\log g)'_{1b} + V_b' + \tfrac12 V_b \tr (\log g)' \right)' - V_b'' \\
& \qquad  + (A'+V'+M')*(A'+V'+M')
\end{alignat*}
Note that the terms $-2(\log g)_{11} + 2 A$, $\tr (\log g - \log \widehat{g}) - A$ and $-(\log g)_{1b} + V_b$ are of type $(A+V)*(A+V+M)$, so their derivatives are of the form $(A + V)*(A' + V' + M')+(A+V+M)*(A'+V')$.
Hence both divergence terms are of this form.
By the same argument as used to analyze $\Ric_{ab} - \widehat{\Ric}_{ab}$, we conclude that $(a,b \geq 1)$
\begin{multline*}
 \mathcal{L}_{ab} - \widehat{\mathcal{L}}_{ab} = - \Mat{A''}{V''}{0} + S_2' + 1*(A+V) + 1*(A'+V') + (A+V)* M' \\
 + (A' + V' + M') * (A' + V' + M'),
\end{multline*}
where $S_2 = (A + V)*(A' + V' + M')+(A+V+M)*(A'+V')$.

We finally compute $\widehat{\mathcal{L}}_{ab}$.
First note that
\[ \widehat X^1 = - \tr \log (E+M) + \sfrac12 \tr (E+M)^{-1} M', \qquad \widehat X^2 = \ldots = \widehat X^{n-1} = 0. \]
Hence
\begin{alignat*}{1}
\widehat{\mathcal{L}}_{11} &= - 2 \tr (E+M)^{-1} M' + \tr (E+M)^{-1} M'' + M' * M' \\
\widehat{\mathcal{L}}_{1b} &= 0 \\
\widehat{\mathcal{L}}_{ab} &= \left( \tr \log (E+M) - \sfrac12 \tr (E+M)^{-1} M' \right) (- M' + 2M + 2E).
\end{alignat*}

Combining all these results, we obtain the structure of the right hand side of (\ref{eq:evoleq}) without the input term: $D_{ab} = - 2 \Ric_{ab} - 2(n-1) g_{ab} - \mathcal{L}_{ab}$.
Note that by (\ref{eq:Linv1})-(\ref{eq:Linv3}), we can identify all occuring $1*(A+V)$ and $1*(A'+V')$ terms.
Before we write down the structure of the $D_{ab}$, we introduce another parameter $F = \tr \log (E+M)$ which is well defined and bounded since we assumed $|M|<0.1$.
Note that $F' = \tr (E+M)^{-1} M'$.
Now for $a,b > 1$ we have
\begin{alignat*}{1}
D_{11} &= A'' + S_A' - (n-1) A' - 2(n-1)A \\
& \qquad + (A+V+M) * (A+ V) + (A+V+M)*(A'+V') \\
& \qquad + (A+V)*M' + (A'+V'+M')*(A'+V'+M') \displaybreak[1]  \\
D_{1b} &= V_b'' + (S_{V_b})' - (n-1) V' - n V \\
& \qquad + (A+V+M) * (A+ V) + (A+V+M)*(A'+V') \\
& \qquad + (A+V)*M' + (A'+V'+M')*(A'+V'+M') \displaybreak[1] \\
D_{ab} &= M_{ab}'' + (S_{M_{ab}})' - (n-1) M'- 2F E_{ab} \qquad\quad  + F * M + F*M' \\
& \qquad + (A+V+M) * (A+ V) + (A+V+M)*(A'+V') \\
& \qquad + (A+V)*M' + (A'+V'+M')*(A'+V'+M')
\end{alignat*}
where $S_A$, $S_V$ and $S_M$ are of the form $(A + V)*(A' + V' + M')+(A+V+M)*(A'+V')$.
It will be essential later that the three expressions above, do not contain any $M * M'$ term which is not already $(A+V) * M * M'$.\label{page:MRdF}
For this property it is important that we are dealing with \emph{modified} Ricci deTurck flow instead of standard Ricci deTurck flow.
Otherwise, the term $(\tr (E+M)^{-1} M' )(M+E)$ in $2 \widehat{\Ric}_{ab}$ would not have canceled with the corresponding term in $\widehat{\mathcal{L}}_{ab}$ and would have created a term of the form $( \tr M' - \tr (E+M)^{-1} M' )(M+E)$ which we are not able to deal with by our methods.

The flow equation (\ref{eq:evoleq}) is equivalent to
\begin{equation} \label{eq:flowAVM}
\begin{split} 
\dot A = D_{11} + I_{11} + \nabla^* J_{11}, \quad \dot V_b &= D_{1b} + I_{1b} + \nabla^* J_{1b}, \\
&\qquad\qquad \dot M_{ab} = D_{ab} + I_{ab} + \nabla^* J_{ab}.
\end{split}
\end{equation}
We will now determine the influence of $I$ and $J$.
Express
\[ I = \Mat{I_{11}}{ e^{-s} I_{1i}}{ e^{-2s} I_{ij} } \]
and 
\[ J = \partial_s \otimes \Mat{J^1_{11}}{ e^{-s} J^1_{1i}}{ e^{-2s} J^1_{ij} } + \sum_{k=2}^n e^s \partial_{x_k} \otimes \Mat{J^1_{11}}{ e^{-s} J^k_{1i}}{ e^{-2s} J^k_{ij} }. \]
Then
\[ \nabla^* J = - \Mat{(J^1_{11})'}{ e^{-s} ( J^1_{1i})'}{ e^{-2s} (J^1_{ij})'} + \Mat{J^1_{11} - 2 \sum_{k=2}^n  J^k_{1k} }{e^{-s} (J^1_{1i} + J^i_{11} - \sum_{k=2}^n J^k_{ki} ) }{e^{-2s} (J^1_{ij} + J^i_{1j} + J^j_{1i})} . \]
So we can redecompose the input term as
\begin{alignat*}{3}
 I + \nabla^* J &= &- & \Mat{(J^1_{11})'}{ e^{-s} (J^1_{1i})'}{ e^{-2s} (J^1_{ij})'} && + \Mat{J^1_{11} + \ldots + I_{11}}{e^{-s} ( J^1_{1i} + \ldots + I_{1i}) }{e^{-2s} ( J^1_{ij} + \ldots + I_{ij})} \\
 &= && \Mat{\widetilde I_{11}}{\widetilde I_{1i}}{\widetilde I_{ij}} && + \Mat{\widetilde J_{11}'}{e^{-s} \widetilde J_{1i}'}{e^{-2s} \widetilde J_{ij}'}. 
\end{alignat*}

From the flow equations (\ref{eq:flowAVM}) we deduce the evolution for $F$ (note that $F \tr (E+M)^{-1} = (n-1) F + F*M$):
\begin{multline*}
\dot F = \tr (E+M)^{-1} \dot M = F'' + S_F' - (n-1) F' - 2(n-1) F  \\
 \qquad + (A+V+M) * (A+V) + (A+V+M)*(A'+V') + (A+V)*M'  \\
 \qquad  + (A'+V'+M')*(A'+V'+M') + F * M + F*M' \\
 \qquad + \tr (E+M)^{-1} \widetilde I_{M} + \tr [ (E+M)^{-1} M' (E+M)^{-1} \widetilde J_M] + [ \tr (E+M)^{-1} \widetilde J_{M} ]',
\end{multline*}
where $S_F = (A+V)*(A'+V'+M') + (A+V+M)*(A'+V')$ and $\widetilde I_M$ resp. $\widetilde J_M$ denote the lower-right block of $e^{2s} \widetilde I$ resp. $e^{2s} \widetilde J$.

We can finally conclude the discussion of the structure of the invariant Ricci deTurck flow equation (\ref{eq:RdTIJ}).
Observe that $(A, F ,V,M)$ satisfies a system of nonlinear heat equations (of one rank higher than the original equation) with input terms $\widetilde I, \widetilde J$.
Group $(A, F, V)$ into one $n+1$-dimensinal quantity $v$ and denote $M$ by $u$.
The input terms $\widetilde I_M, \widetilde J_M$ are now denoted by $I_u, J_u$ and the terms $\widetilde I_{A}, \widetilde I_{V}, \tr (E+M)^{-1} \widetilde I_M$ resp. $\widetilde J_A, \widetilde J_V, \tr (E+M)^{-1} \widetilde J_M$ are denoted by $I_v$ resp. $J_v$.
Then the modified Ricci deTurck flow equation is of the form
\begin{alignat*}{4}
\dot u &= u'' - (n-1) u' \;\; &&  + R_u + S'_u \;\; &&+ I_u \;\; &&+ J_u' \\
\dot v &= v'' - (n-1) v' - \twocoeff{2(n-1)}{n} v \;\; && + R_v + S'_v \;\; &&+ I_v + u' * J_u \;\; &&+ J_v', 
\end{alignat*}
where $\twocoeff{2(n-1)}{n}$ means that we have to choose the coefficient $2(n-1)$ for the $A$- and $F$-component and $n$ for the $V$-component of $v$.
One can think of it as a diagonal matrix.
Furthermore, the nonlinear terms are
\begin{alignat}{1}
R_u &= 1*v + (u + v) * v' + v * u' + (u'+v')*(u'+v') \label{eq:struc1} \\
R_v &= ( u + v) * v + (u+v) * v' + v * u' + (u'+v')*(u'+v') \label{eq:struc2} \\
S_u, S_v &= v * (u' + v') + u * v'. \label{eq:struc3}
\end{alignat}
We can simplify these equations by using instead of $(s,t)$ the coordinates $(x,t)$ with $x = s - (n-1)t$:
\begin{subequations}
\begin{alignat}{3}
\dot u &= u''  \;\; && + R_u + S'_u  \;\; &&+ I_u + J_u' \label{eq:flow1} \\
\dot v &= v'' - \twocoeff{2(n-1)}{n} v \;\; && + R_v + S'_v \;\; &&+ I_v + J_v' + u' * J_u. \label{eq:flow2}
\end{alignat}
\end{subequations}
From now on, we will be dealing with these two equations only.

We can now reformulate Theorem \ref{Thm:cuspinv} as a statement for the system (\ref{eq:flow1}), (\ref{eq:flow2}).
Observe that in the $(x,t)$ coordinates, the domain, on which the quantities $u,v$ etc. are defined, is
\[ D = \{ (x,t) \;\; : \;\; x \geq -(n-1)t, \;\; 0 \leq t < T \}. \]
Its parabolic boundary $\partial_p D$ consists of the lines $\{ (-(n-1) t, t) \; : \; 0 \leq t < T \}$ and $\{ (x,0) \; : \; x \geq 0 \}$.
Denote by
\[ B_{\sigma}(\partial_p D) = \{ (x,t) \in D \;\; : \;\; x < -(n-1)t +  \sigma \} \cup  (D \cap \IR \times [0,\sigma^2) ) \]
the $\sigma$-tubular neighborhood around $\partial_p D$.
Observe that in order to prove Theorem \ref{Thm:cuspinv}, it suffices to establish the following fact:

\begin{Proposition} \label{Prop:cuspinv2}
 Assume that $u,v$ and $I_u, I_v, J_u, J_v$ are defined on the domain $D$ and satisfy equations (\ref{eq:flow1}) and (\ref{eq:flow2}).
 Moreover, assume that $u,v$ and their spatial derivatives up to third order are bounded on $D$ by some constant. \\
 Let $\delta > 0$.
 Then there are constants $\varepsilon_{inv} > 0$ and $C_{inv} < \infty$ depending only on $\delta$ and $n$ such that if
\begin{multline*}
 H = \sup_{B_{9 \sigma}(\partial_p D)} \big(|u| + |u'| + |v| + |v'| \big) + \sup_{(x,t) \in D} e^{x + (n-1 + \delta) t} \times \\ \big( |I_u|(x,t) + |I_v|(x,t) + |J_u|(x,t) + |J_v|(x,t) \big) < \varepsilon_{inv},
\end{multline*}
then $\Vert u \Vert_{L^\infty(D)} + \Vert v \Vert_{L^\infty (D)} \leq C_{inv} H$.
\end{Proposition}

\subsection{Introduction to the analytical problem}
Our analysis of the flow equations (\ref{eq:flow1}), (\ref{eq:flow2}) will be based on the following idea:
Looking at the linear part of (\ref{eq:flow1}), we expect $u$ to slowly converge towards a constant function.
The linear part of (\ref{eq:flow2}) suggests an exponential decay of $v$ which is however dominated by the slower decay of its nonlinear part.
More precisely, we expect the following behaviour for $t \to \infty$:
\begin{equation*} 
u \sim 1, \qquad u' \sim \frac1{t^{1/2}}, \qquad v,v' \sim \frac1t. 
\end{equation*}
This would imply that $S_u, S_v \sim \frac1{t^{1/2}}$ and $R_u, R_v \sim \frac1t$ which are in turn exactly the critical exponents to ensure the correct decays for $u, u', v, v'$.

In order to make this rigid, we will adapt the method of Koch and Lamm from \cite{KL} to our case.
It is recommended to understand first their proof for equations of the form $\dot f = f'' + S_f' + R_f$, where $S_f = f * f'$ and $R_f = f' * f'$.
We also remark, that the following proof also works if the input terms $I_u, I_v, J_u, J_v$ are zero and $D = \IR \times [0,T)$.
So in a second step, it might be helpful to go through the following proof, while having this simplified setting in mind.

Consider again the coordinate system $(x,t)$ and the domain $D$.
We define a function $r : D \to [0,\infty)$ which gives us a local scale:
\[ r(x,t) = \max \{ r \;\; : \;\; [x-2r,x+2r] \times [t-r^2,t] \subset D \} \]
If $D$ was $\IR \times [0,T)$, then $r(x,t)$ would be just $\sqrt{t}$.
Furthermore, using the notation $x^- = \min \{ x, 0 \}$, we define the parabolic domains
\begin{alignat*}{2}
 P_{r'}(x) &= [x-r',x+r'] \times [0, (r')^2 - (n-1)^{-1} x^-] \cap D \qquad &&\text{and} \\
 Q(x,t) &= [x-r, x+r] \times [t - r^2/2,t], \quad &&\text{where $r = r(x,t)$.}
\end{alignat*}

In order to simplify our notation, we will make use of the symbol $\lesssim$.
By $a \lesssim b$ we will mean $a \leq C b$ for some constant $C$ which only depends on the constants $\delta, \sigma, \mu_1, \mu_2$ ($\mu_1, \mu_2$ will be introduced in subsection \ref{subsec:Intronorms}).

\subsection{The heat kernel}
In the following, we will denote by $\zeta$ one of the two numbers $2(n-1)$ or $n$ depending on which component of $v$ we analyze (recall the $\twocoeff{2(n-1)}{n}$-coefficient in (\ref{eq:flow2})).
Then $\Phi, \ov \Phi \in C^\infty(\IR \times \IR_+)$ with
\[ \Phi (x,t) = (4 \pi t)^{-1/2} \exp \Big({- \frac{x^2}{4t} }\Big), \qquad \ov \Phi (x,t) = e^{- \zeta t} \Phi(x,t) \]
are the heat kernels of the linear part of (\ref{eq:flow1}) resp. (\ref{eq:flow2}).
Note that the ambiguity in the definition of $\ov\Phi$ will not create any problems, because it will only be important to us that $\ov\Phi$ has \emph{some} exponential decay.

We will need the following bounds for $\Phi$ and $\ov \Phi$:
\begin{Lemma} \label{Lem:hk}
We have the following estimates on $\Phi$ resp. $\ov{\Phi}$:
 \begin{enumerate}[(a)]
  \item For all $r > 0$
\[ \Vert \Phi \Vert_{L^{5/3}([-r,r] \times [0,r^2])} < C r^{4/5}, \quad \Vert \Phi' \Vert_{L^{5/4}([-r,r] \times [0,r^2])} < C r^{2/5}. \]
The same holds for $\ov\Phi$.
  \item Assume that $(x_0, t_0) \in D$ and set $r_0 = r(x_0,t_0)$.
Let $\ov r \geq \max \{ \sigma, r_0 \}$ and consider $(x,t) \in \big( \IR \times [0,t_0) \setminus [x_0 - \ov r, x_0 + \ov r] \times [t_0 - \ov r^2 /2, t_0] \big) \cap D$.
Then for some universal $c > 0$
\begin{multline*}
 \quad \Phi (x_0-x,t_0-t) < C \ov r^{-1} e^{-c|x_0-x|/\ov r}, \quad | \Phi'|(x_0-x,t_0-t) < C \ov r^{-2} e^{-c|x_0-x|/\ov r}, \\ |\Phi''|(x_0-x,t_0-t) < C \ov r^{-3} e^{-c|x_0-x|/\ov r}. 
\end{multline*}
The same holds for $\ov \Phi$, $\ov \Phi'$ resp. $\ov \Phi''$ even when we replace $\ov r^{-1}$, $\ov r^{-2}$ resp. $\ov r^{-3}$ by higher powers in $\ov r^{-1}$.
  \item If $(x,t) \in \IR \times [0, \infty) \setminus [-\sigma,\sigma] \times [0,\sigma^2]$, then
\[ |\ov \Phi|(x,t), \; |\ov \Phi'|(x,t), \; |\ov \Phi''|(x,t) < C e^{-c|x|-ct}. \]
 \end{enumerate}
\end{Lemma}
\begin{proof}
Parts (a) and (c) can be checked easily.
The statement on $\ov\Phi$ in part (b) follows from from the statement on $\Phi$ and part (c).
So we only have to prove the estimates on $\Phi$, $\Phi'$ and $\Phi''$.

First observe that
\begin{alignat}{1}
 |\Phi'|(x_0-x,t_0-t) &\leq C (t_0-t)^{-1} \exp \Big({ -\frac{(x_0-x)^2}{8(t_0-t)} }\Big) \label{eq:Phiprime} \\  
|\Phi''|(x_0-x,t_0-t) &\leq C (t_0-t)^{-3/2} \exp \Big({ -\frac{(x_0-x)^2}{8(t_0-t)} }\Big). \label{eq:Phiprimeprime}
\end{alignat}

In case $t \in [t_0 - \ov r^2/2, t_0]$ and hence $|x_0 -x | \geq \ov r$, we have
\begin{multline*} \Phi(x_0-x,t_0-t) \leq C \ov r^{-1} \Big( \frac{t_0-t}{\ov r^2} \Big)^{-1/2} \exp \Big( - \frac{\ov r^2}{16(t_0-t)} \Big) \exp \Big({ - \frac{|x_0-x|^2}{16(t_0-t)} }\Big) \\
\leq C \ov r^{-1} \exp \Big({ - \frac{(x_0-x)^2}{16(t_0 - t)} }\Big) \leq C \ov r^{-1} \exp \Big({ - c \frac{|x_0 - x|}{\ov r} }\Big).
\end{multline*}
For some $c > 0$.
The estimates for $\Phi'$ and $\Phi''$ follow analogously by (\ref{eq:Phiprime}) and (\ref{eq:Phiprimeprime}).

Now assume $t < t_0 - \ov r^2/2$.
Note that by the definition of $r_0$ and by $\ov r \geq \max\{ \sigma, r_0 \}$, we can conclude that the vertical line through $(x_0, t_0)$ intersects the boundary of $D$ in a point $(x_0, t_1)$ such that $t_0 - t_1 \leq  \big( 1+ 2 (n-1)^{-1}\sigma^{-1} \big) \ov r^2 =: (n-1)^{-1} A \ov r^2$.
So 
\[ x_0^- \leq - (n-1) t_0 + A \ov r^2. \]
Hence
\begin{equation} 
\label{eq:x0minusxandA} |x_0 - x| \geq x^- -x_0^- \geq (n-1) (t_0 - t) - A \ov r^2.  
\end{equation}
Observe that
\[ \Phi (x_0 - x, t_0 - t) \leq C \ov r^{-1} \exp \Big( - \frac{(x_0 - x)^2}{8 (t_0-t) } \Big) \]
and analogously, for $\Phi'$ and $\Phi''$ (here we will get an $\ov r^{-2}$ resp. $\ov r^{-3}$ factor in front of the exponential function).
So it remains to show that
\begin{equation} \label{eq:expeq} 
\exp \Big( - \frac{(x_0-x)^2}{8(t_0-t)} \Big) \leq C \exp \Big( - c \frac{|x_0-x|}{\ov r} \Big)
\end{equation}
for some universal $c > 0$.
Since the function $f(y) = \frac{y}{y/\sigma + 1} + \frac{1}{y}$ is bounded from below by some constant $c > 0$ for positive $y$ and (using (\ref{eq:x0minusxandA}))
\[ f(A^{-1} \ov r^{-1} |x_0-x|) = \frac{|x_0-x|}{\sigma^{-1} |x_0-x| + A \ov r} + \frac{A \ov r}{|x_0 - x|} \leq \frac{\ov r |x_0-x|}{|x_0-x| + A \ov r^2} + \frac{A \ov r}{|x_0 - x|} , \]
we get
\[ c \frac{|x_0 - x|}{\ov r} \leq \frac{(x_0 - x)^2}{|x_0 - x| +  A \ov r^2} + A \leq \frac{(x_0 - x)^2}{t_0-t} + A. \]
Exponentiating this equation yields (\ref{eq:expeq}).
\end{proof}

\subsection{Representing $u$ and $v$ using the heat kernel} \label{subsec:representu}
Let $\widetilde\varphi : \IR \to [0,1]$ be a smooth function with $\widetilde\varphi \equiv 0$ on $(-\infty, 0]$ and $\widetilde\varphi \equiv 1$ on $[1,\infty)$.
Define $\varphi \in C^{\infty}(D)$ by
\[ \varphi (x,t) = \widetilde\varphi ( (x+ (n-1) t) / \sigma ) \widetilde\varphi (t/\sigma^2). \]
$\varphi$ is a cutoff function whose support lies in the interior of $D$ and which is constant outside $B_\sigma(\partial_p D)$.

Let $(x_0,t_0) \in \Int D$ and assume that $\varphi(x_0,t_0) = 1$.
Recall that $u$ and $v$ as well as their first derivatives were assumed to be bounded on $D$.
Hence, we can use integration by parts and (\ref{eq:flow1}) to compute that for $0 \leq t < t_0$
\begin{alignat*}{1}
 &\partial_t \int_{- (n-1) t}^\infty \varphi^2 \Phi(x_0-x,t_0-t) u (x,t) dx = \int_{-(n-1) t}^{\infty} 2\dot \varphi \varphi \Phi(x_0-x,t_0-t) u(x,t) dx \\
& \qquad + \int_{-(n-1) t}^\infty \varphi^2 \bigl[ - \Phi''(x_0-x,t_0-t) u(x,t) + \Phi(x_0-x,t_0-t) u''(x,t) \bigr] dx \\
& \qquad + \int_{-(n-1) t}^\infty \varphi^2 \Phi(x_0 - x, t_0 - t) (R_u + S'_u + I_u + J'_u)(x,t) dx \displaybreak[2] \\
& = \int_{-(n-1) t}^{\infty} 2\varphi \dot \varphi \Phi(x_0-x,t_0-t) u(x,t) dx \\ 
&\qquad - \int_{-(n-1) t}^{\infty} 2 \bigl[ (\varphi \varphi'' +  (\varphi')^2) u + 2 (\varphi \varphi') u' \bigr] \Phi(x_0 - x, t_0 - t) dx \\
&\qquad - \int_{-(n-1) t}^{\infty} 2\varphi \varphi' \Phi(x_0-x,t_0-t) (S_u + J_u) dx \\
&\qquad + \int_{-(n-1) t}^\infty \varphi^2 \bigl[ \Phi(x_0-x,t_0-t) (R_u + I_u) + \Phi'(x_0-x,t_0-t)(S_u + J_u) \bigr] dx.
\end{alignat*}
Integration over $t$ from $0$ to $t_0$ yields
\[ u(x_0,t_0) = u^*(x_0,t_0) + u^{**}(x_0,t_0), \]
where
\begin{multline} u^* (x_0,t_0) = \int_{D \cap \IR \times [0,t_0]} \varphi^2 \bigl[ \Phi(x_0-x,t_0-t) (R_u + I_u) \\ 
+ \Phi'(x_0-x,t_0-t)(S_u + J_u) \bigr] dx dt \label{eq:ustar}
\end{multline}
and
\begin{multline}
 u^{**}(x_0,t_0) = \int_{D \cap \IR \times [0,t_0]} 2\bigl[ ( \dot\varphi \varphi -  \varphi \varphi'' - (\varphi')^2 ) u - 2 \varphi \varphi' u' \\
 -  \varphi\varphi' (S_u+J_u)  \bigr] \Phi(x_0 - x, t_0 - t) dx dt  \label{eq:ustarstar}
\end{multline}
Analogously, we find
\[ v(x_0,t_0) = v^*(x_0,t_0) + v^{**}(x_0,t_0), \]
where
\begin{multline} v^* (x_0,t_0) = \int_{D \cap \IR \times [0,t_0]} \varphi^2 \bigl[ \ov \Phi(x_0-x,t_0-t) (R_v + I_v + u' * J_u) \\ 
+ \ov \Phi'(x_0-x,t_0-t)(S_v + J_v) \bigr] dx dt \label{eq:vstar}
\end{multline}
and
\begin{multline*} 
 v^{**}(x_0,t_0) = \int_{D \cap \IR \times [0,t_0]} 2\bigl[ ( \dot\varphi \varphi -  \varphi \varphi'' - (\varphi')^2 ) v - 2 \varphi \varphi' v' \\
 -  \varphi\varphi' (S_v+J_v) \bigr] \ov\Phi(x_0 - x, t_0 - t)
 dx dt. 
\end{multline*}

\subsection{Estimating $u^{**}$ and $v^{**}$} \label{subsec:ususs}
We have the following estimates on $u^{**}$ and $v^{**}$:
\begin{Lemma} \label{Lem:starstar}
Assume that $H < 0.1$.
Then for $(x_0,t_0) \in D$ we have
\begin{enumerate}[(a)]
 \item $| u^{**} |(x_0,t_0), \; | v^{**} |(x_0,t_0), \; | (u^{**})'|(x_0,t_0), \; | (v^{**})'|(x_0,t_0) \lesssim H$
 \item $|(u^{**})'|(x_0,t_0) \lesssim r^{-1}(x_0,t_0) H$
 \item $|v^{**}|(x_0,t_0), |(v^{**})'|(x_0,t_0) \lesssim r^{-1}(x_0,t_0) \exp ( -c r(x_0,t_0) ) H$ for some $c > 0$.
\end{enumerate}
\end{Lemma}
\begin{proof}
Observe that since we have the bounds $|R_u|, |S_u|, |I_u|, |J_u| \lesssim H$ on $B_\sigma (\partial_p D)$, we can estimate for each $(x_0,t_0) \in B_{\sigma}(\partial_p D)$ using (\ref{eq:ustarstar}):
\begin{multline*}
  |u^{**}|(x_0,t_0) \leq C \int_{B_{\sigma}(\partial_p D) \cap \IR \times [0,t_0]} H \Phi(x_0-x,t_0-t) dx dt \displaybreak[1] \\
\leq C \int_{\IR \times [t_0-1,t_0]} H \Phi(x_0-x,t_0-t) dx dt +  \int_{\IR \times [0,\sigma^2]} CH \Phi(x_0 - x, t_0 - t) dx dt \displaybreak[1] \\
 + \int_{B_{\sigma}(\partial_p D) \cap \IR \times [\sigma^2,t_0-1]} C H \exp \Big({ - \frac{(x_0-x)^2}{4(t_0-t)} }\Big) dx dt \\
\leq C H + C H \int_{B_{\sigma}(\partial_p D) \cap \IR \times [\sigma^2,t_0-1]} \exp \Big({ - \frac{(n-1)^2 (t_0 - t)^2}{4(t_0-t)} }\Big) dx dt \leq C H.
\end{multline*}
The same is true for $|(u^{**})'|$, $|v^{**}|$ and $|(v^{**})'|$ (note that $\Phi'$ is integrable around the origin).
Since $u^{**}, v^{**}$ and their derivatives satisfy the linear heat equations on $D \setminus B_\sigma(\partial_p D)$:
\[ \partial_t u^{**} - (u^{**})'' = 0, \qquad
\partial_t v^{**} - (v^{**})'' + \zeta v^{**} = 0,\]
we conclude (a) by the maximum principle.

For (b) assume that $(x_0,t_0) \in \Int D$ and set $r_0 = r(x_0,t_0)$.
First note that we can assume that $[x_0 - 2 r_0, x_0 + 2 r_0] \times [t_0 - r_0^2/2, t_0] \subset D \setminus B_{\sigma}(\partial_p D)$, because otherwise, $r_0$ is smaller than some constant and we can simply use part (a).
Let $\eta : \IR \to [0,1]$ be a cutoff function which is $\equiv 1$ on $[-1,1]$ and $\equiv 0$ outside $[-2,2]$ and set $\eta_{r_0, x_0}(x) = \eta((x-x_0)/r_0)$.
Then by the same method as used in subsection \ref{subsec:representu}, we can compute that
\begin{alignat*}{1}
 (u^{**})'(x_0,t_0) &= \int_{t_0 - r_0^2/2}^{t_0} \int_{x_0-2r_0}^{x_0 + 2r_0} \bigl[ 2\bigl( \eta_{r_0,x_0}\eta_{r_0,x_0}'' + (\eta_{r_0,x_0}')^2 \bigr) \Phi'(x_0-x,t_0-t) \\ &\qquad\qquad\quad - 4  \eta_{r_0,x_0} \eta_{r_0,x_0}' \Phi''(x_0-x,t_0-t) \bigr] u^{**}(x,t) dxdt \\
&\quad + \int_{x_0 - 2 r_0}^{x_0 + 2 r_0} \eta_{r_0,x_0}^2 \Phi'(x_0-x, \sfrac{r_0^2}2) u^{**} (x,t_0 - \sfrac{r_0^2}2) dx.
\end{alignat*}
So by Lemma \ref{Lem:hk} (b) its absolute value is bounded by
\[  \int_{t_0 - r_0^2/2}^{t_0} \int_{x_0 - r_0}^{x_0 + r_0} C H r_0^{-4} dxdt + \int_{x_0 - r_0}^{x_0 + r_0} C H r_0^{-2} dx \leq CH r_0^{-1}. \]

Part (c) can be proved in the same way, except that we now have to employ Lemma \ref{Lem:hk} (c).
\end{proof}

\subsection{The $L^p_{\mu}$-norm} \label{subsec:Lpnorm}
We will need a norm which is slightly stronger than the $L^p$-norm.
Assume $\mu > 0$.
For any $r > 0$, $x \in \IR$ and function $f \in L^{p+p\mu}_{loc}(D)$ we set
\[ \Vert f \Vert_{L^p_{\mu}(P_r(x))} = \left( \int_{P_r(x)} \Vert f \Vert_{L^{p+p\mu}([x'-\sigma,x'+\sigma] \times [t'-\sigma^2/2,t'+\sigma^2/2] \cap P_r(x))}^p d x' d t' \right)^{1/p}, \]
where the norm under the integral sign should be understood as the norm of the restriction of $f$ to the indicated parabolic domain.
It is easy to see that the $L^p_{\mu}$-norm is stronger than the $L^p$-norm, i.e. for $r \geq \sigma$.
\[ \Vert f \Vert_{L^p(P_r(x))} \lesssim \Vert f \Vert_{L^p_{\mu} (P_r(x))}. \]

\subsection{Introduction of the norms} \label{subsec:Intronorms}
Fix some arbitrary constants $\mu_1, \mu_2 > 0$ such that $\frac{1}{1+\mu_1} = \frac1{2+2\mu_2} + \frac12$ and $\mu_1, \mu_2 < \frac14$.
Assume that $\sigma^2 < T' \leq T$.
We are going to control the following norms:
\begin{alignat*}{1}
\alpha_{u,T'} &= \Vert u \Vert_{L^\infty (D \cap \IR \times [0,T'))} \\
\alpha_{v,T'} &= \Vert v \Vert_{L^\infty (D \cap \IR \times [0,T'))} \displaybreak[1] \\
\beta_{u,T'} &= \sup_{r \geq \sigma} \sup_x r^{-1/2} \Vert u' \Vert_{L^2(P_r(x) \cap \IR \times [0,T'))} \\
\beta_{v,T'} &= \sup_{r \geq \sigma} \sup_x r^{-1} \left( \Vert v \Vert_{L^1_{\mu_1}(P_r(x) \cap \IR \times [0,T'))} + \Vert v' \Vert_{L^1_{\mu_1}(P_r(x) \cap \IR \times [0,T'))} \right) \\
&\qquad + \sup_{r \geq \sigma} \sup_x r^{-1/2} \left( \Vert v \Vert_{L^2_{\mu_2} (P_r(x) \cap \IR \times [0,T'))} + \Vert v' \Vert_{L^2(P_r(x) \cap \IR \times [0,T'))} \right) \displaybreak[1] \\
\gamma_{u,T'}  &= \sup_{\stackrel{\scriptstyle (x,t) \in D}{0 \leq t < T'}} r^{2/5}(x,t) \Vert u' \Vert_{L^5(Q(x,t))} \\
\gamma_{v,T'}  &= \sup_{\stackrel{\scriptstyle (x,t) \in D}{0 \leq t < T'}} r^{2/5}(x,t) \left( \Vert v \Vert_{L^{5}(Q(x,t))} + \Vert v' \Vert_{L^{5}(Q(x,t))} \right) \\
& \qquad + \sup_{\stackrel{\scriptstyle (x,t) \in D}{0 \leq t < T'}} r^{4/5}(x,t) \left( \Vert v \Vert_{L^{5/2}(Q(x,t))} + \Vert v' \Vert_{L^{5/2}(Q(x,t))} \right)
\end{alignat*}
Observe that by the derivative bounds on $u$ and $v$, these norms vary continuously in $T'$.

To simplify notation, we set $\alpha_{T'} = \alpha_{u, T'} + \alpha_{v, T'}$, $\beta_{T'} = \beta_{u, T'} + \beta_{v, T'}$ and $\gamma_{T'} = \gamma_{u, T'} + \gamma_{v, T'}$.
Moreover, we will most often leave out the $T'$ in the index.

\subsection{The estimates}
In this subsection, we will derive inequalities for these norms that are independent of $T'$ (see Lemmas \ref{Lem:alpha}, \ref{Lem:beta} and \ref{Lem:gamma}).
Since, we can always restrict the solutions $u, v$ to the time interval $[0,T')$, we can assume without loss of generality that $T' = T > \sigma^2$.
We will make use of the following identities:

\begin{Lemma} \label{Lem:RSbound}
Assume that $\alpha_T < 0.1$.
Then the quantities $R_u, R_v, S_u, S_v$ and $I_u, I_v, J_u, J_v$ satisfy the following bounds:
\begin{enumerate}[(a)]
 \item  If $(x,t) \in D$, $t < T$ and $r = r(x,t) \geq \sigma$, then
\begin{alignat*}{1}
 r^{4/5} \Vert R_u \Vert_{L^{5/2}(Q(x,t))} &\lesssim \gamma^2 + \gamma_v \displaybreak[1] \\ 
 r^{4/5} \Vert R_v \Vert_{L^{5/2}(Q(x,t))} &\lesssim \alpha^2 + \gamma^2 \displaybreak[2] \\ 
 r^{2/5} \Vert S_u \Vert_{L^5(Q(x,t))}, \;\; r^{2/5} \Vert S_v \Vert_{L^5(Q(x,t))} &\lesssim  \alpha^2 + \gamma^2 \displaybreak[3] \\
  r^{4/5} \Vert S_v \Vert_{L^{5/2}(Q(x,t))} &\lesssim  \alpha^2 + \gamma^2 \displaybreak[3] \\
 r^{4/5} \Vert I_u \Vert_{L^{5/2}(Q(x,t))}, \;\; r^{4/5} \Vert I_v \Vert_{L^{5/2}(Q(x,t))} &\lesssim H \\
 r^{2/5} \Vert J_u \Vert_{L^5 (Q(x,t))}, \;\; r^{2/5} \Vert J_v \Vert_{L^5 (Q(x,t))} &\lesssim H \\
 r^{4/5} \Vert J_v \Vert_{L^{5/2} (Q(x,t))} &\lesssim H \\
 r^{4/5} \Vert u' * J_u \Vert_{L^{5/2}(Q(x,t))} & \lesssim \gamma_u H
\end{alignat*}
\item If $x \in \IR$ and $r \geq \sigma$, then
\begin{alignat*}{1}
 r^{-1} \Vert R_u \Vert_{L^1(P_r(x))} & \lesssim \alpha^2 + \beta^2 + \beta_v \displaybreak[1] \\
 r^{-1} \Vert R_v \Vert_{L^1(P_r(x))} & \lesssim \alpha^2 + \beta^2 \displaybreak[2] \\
 r^{-1/2} \Vert S_u \Vert_{L^2(P_r(x))}, \;\; r^{-1/2} \Vert S_v \Vert_{L^2(P_r(x))} &\lesssim \alpha^2 + \beta^2 \displaybreak[1] \\
 r^{-1} \Vert S_u \Vert_{L^1(P_r(x))}, \;\; r^{-1} \Vert S_v \Vert_{L^1_{\mu_1}(P_r(x))} &\lesssim \alpha^2 + \beta^2 \displaybreak[3] \\
 r^{-1} \Vert I_u \Vert_{L^1(P_r(x))}, \;\; r^{-1} \Vert I_v \Vert_{L^1(P_r(x))} &\lesssim H \\
 r^{-1/2} \Vert J_u \Vert_{L^2(P_r(x))}, \;\; r^{-1/2} \Vert J_v \Vert_{L^2(P_r(x))} &\lesssim H \\
 r^{-1} \Vert J_u \Vert_{L^1(P_r(x))}, \;\; r^{-1} \Vert J_v \Vert_{L^1_{\mu_1}(P_r(x))} &\lesssim H \\
 r^{-1} \Vert u' * J_u \Vert_{L^1(P_r(x))} & \lesssim \beta_u H
\end{alignat*}
\end{enumerate}
\end{Lemma}
\begin{proof}
 The bounds on $R_u, R_v, S_u, S_v$ follow from their algebraic structure (see (\ref{eq:struc1}), (\ref{eq:struc2}), (\ref{eq:struc3})) using Cauchy-Schwarz.
Note that in order to bound the term $r^{-1} \Vert S_v \Vert_{L^1_{\mu_1}(P_r(x))}$, we have to make use of $\frac{1}{1+\mu_1} = \frac1{2+2\mu_2} + \frac12$.

The bounds on the terms $I_u, I_v, J_u, J_v$ follow immediately from the hypothesis of Proposition \ref{Prop:cuspinv2} which actually asserts a stronger exponential decay with respect to a stronger $L^p$-norm. 
\end{proof}

\begin{Lemma} \label{Lem:alpha}
If $\alpha < 0.1$, then we have
\begin{alignat*}{1}
\alpha_u &\lesssim \alpha^2 + \beta^2 + \gamma^2 + \beta_v + \gamma_v + H\\
\alpha_v &\lesssim \alpha^2 + \beta^2 + \gamma^2 + (\beta_u + \gamma_u)H + H
\end{alignat*}
\end{Lemma}

\begin{figure}[t]
\caption{The domains $P_{br_0}(x_0 - 2kbr_0)$ and $P_{br_k}(z_k)$ cover the domain $D \cap \IR \times [0,t_0]$ in the proof of Lemma \ref{Lem:alpha}}
\label{fig:alpha}
\begin{center}
\begin{picture}(0,0)%
\hspace{2mm}\includegraphics[width=14cm]{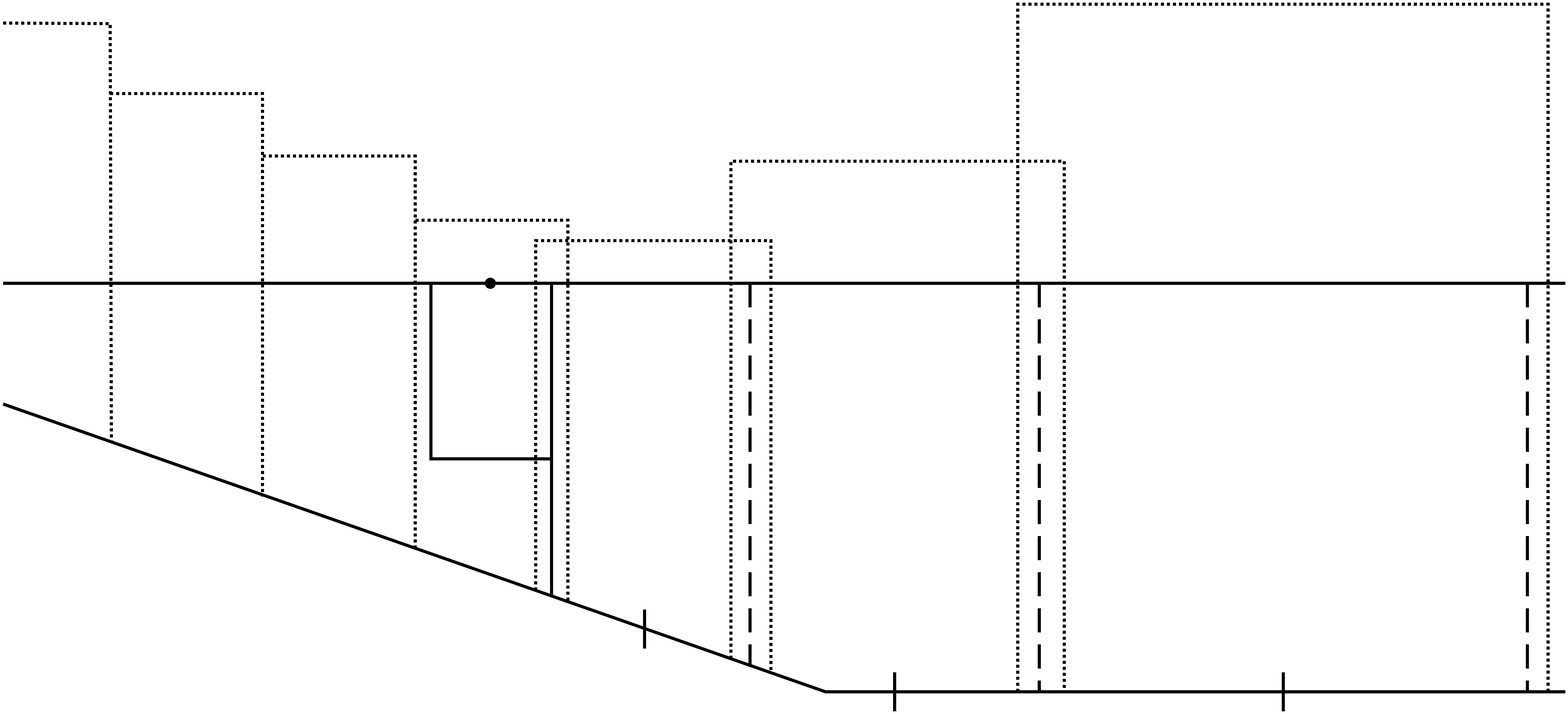}%
\end{picture}%
\setlength{\unitlength}{2863sp}%
\begingroup\makeatletter\ifx\SetFigFont\undefined%
\gdef\SetFigFont#1#2#3#4#5{%
  \reset@font\fontsize{#1}{#2pt}%
  \fontfamily{#3}\fontseries{#4}\fontshape{#5}%
  \selectfont}%
\fi\endgroup%
\begin{picture}(11195,4100)(518,-5050)
\put(1000,-4210){\makebox(0,0)[lb]{\smash{{\SetFigFont{12}{14.4}{\familydefault}{\mddefault}{\updefault}$x= -(n-1) t$}}}}
\put(3360,-3300){\makebox(0,0)[lb]{\smash{{\SetFigFont{12}{14.4}{\familydefault}{\mddefault}{\updefault}$\Omega_0$}}}}
\put(2850,-3000){\makebox(0,0)[lb]{\smash{{\SetFigFont{12}{14.4}{\familydefault}{\mddefault}{\updefault}$\frac{r_0^2}2$}}}}
\put(3100,-2450){\makebox(0,0)[lb]{\smash{{\SetFigFont{12}{14.4}{\familydefault}{\mddefault}{\updefault}$(x_0,t_0)$}}}}
\put(1900,-3300){\makebox(0,0)[lb]{\smash{{\SetFigFont{12}{14.4}{\familydefault}{\mddefault}{\updefault}$\Omega_-$}}}}
\put(5700,-3300){\makebox(0,0)[lb]{\smash{{\SetFigFont{12}{14.4}{\familydefault}{\mddefault}{\updefault}$\Omega_+$}}}}
\put(1640,-1250){\makebox(0,0)[lb]{\smash{{\SetFigFont{12}{14.4}{\familydefault}{\mddefault}{\updefault}$P_{br_0}(x_0-2kbr_0)$}}}}
\put(4000,-4000){\makebox(0,0)[lb]{\smash{{\SetFigFont{12}{14.4}{\familydefault}{\mddefault}{\updefault}$P_{br_1}(z_1)$}}}}
\put(5400,-4000){\makebox(0,0)[lb]{\smash{{\SetFigFont{12}{14.4}{\familydefault}{\mddefault}{\updefault}$P_{br_2}(z_2)$}}}}
\put(7600,-4000){\makebox(0,0)[lb]{\smash{{\SetFigFont{12}{14.4}{\familydefault}{\mddefault}{\updefault}$P_{br_3}(z_3)$}}}}
\put(4310,-4860){\makebox(0,0)[lb]{\smash{{\SetFigFont{12}{14.4}{\familydefault}{\mddefault}{\updefault}$z_1$}}}}
\put(5750,-5230){\makebox(0,0)[lb]{\smash{{\SetFigFont{12}{14.4}{\familydefault}{\mddefault}{\updefault}$z_2$}}}}
\put(8030,-5230){\makebox(0,0)[lb]{\smash{{\SetFigFont{12}{14.4}{\familydefault}{\mddefault}{\updefault}$z_3$}}}}
\put(3710,-4660){\makebox(0,0)[lb]{\smash{{\SetFigFont{12}{14.4}{\familydefault}{\mddefault}{\updefault}$y_1$}}}}
\put(4950,-5050){\makebox(0,0)[lb]{\smash{{\SetFigFont{12}{14.4}{\familydefault}{\mddefault}{\updefault}$y_2$}}}}
\put(6630,-5230){\makebox(0,0)[lb]{\smash{{\SetFigFont{12}{14.4}{\familydefault}{\mddefault}{\updefault}$y_3$}}}}
\put(9430,-5230){\makebox(0,0)[lb]{\smash{{\SetFigFont{12}{14.4}{\familydefault}{\mddefault}{\updefault}$y_4$}}}}

\put(3350,-3750){\makebox(0,0)[lb]{\smash{{\SetFigFont{12}{14.4}{\familydefault}{\mddefault}{\updefault}$\scriptstyle 2 r_0$}}}}
\end{picture}%
\end{center}
\vspace{0.5cm}
\end{figure}

\begin{proof}
Let $(x_0, t_0) \in D$ and set $r_0 = r_0(x_0,t_0)$.
If $r_0 < 4\sigma$ (and if $\sigma$ is sufficiently small), then $(x_0,t_0) \in B_{9\sigma}(\partial_p D)$ and the hypothesis of Proposition \ref{Prop:cuspinv2} already gives us $|u|(x_0,t_0), |v|(x_0,t_0) \lesssim H$.
So assume in the following $r_0 \geq 4\sigma$ and hence $(x_0, t_0) \not\in B_{\sigma}(\partial_p D)$.

We first establish the bound on $\alpha_u$.
As in subsection \ref{subsec:ususs}, decompose $u = u^* + u^{**}$.
Lemma \ref{Lem:starstar} (a) gives us $|u^{**}|(x_0,t_0) \lesssim H$.
So we just have to bound $|u^*|(x_0,t_0)$.
Recall from (\ref{eq:ustar}) that
\begin{multline*} u^* (x_0,t_0) = \int_{D \cap \IR \times [0,t_0]} \varphi^2 \bigl[ \Phi(x_0-x,t_0-t) (R_u + I_u) \\ 
+ \Phi'(x_0-x,t_0-t)(S_u + J_u) \bigr] dx dt
\end{multline*}
We split the domain of integration $D \cap \IR \times [0,t_0]$ into disjoint subsets $\Omega_-, \Omega_0, \Omega_+$ where
\begin{alignat*}{1}
 \Omega_0 &= Q(x_0,t_0) = [x_0 - r_0, x_0 + r_0] \times [t_0 - r_0^2/2, t_0], \\
 \Omega_- &= (-\infty,x_0+r_0] \times [0,t_0] \cap D \setminus \Omega_0, \\
 \Omega_+ &= (x_0+r_0, \infty) \times [0,t_0] \cap D
\end{alignat*}
(see also Figure \ref{fig:alpha}) and estimate the integral over each of these subdomains.
On $\Omega_0$, we use H\"older's inequality, Lemma \ref{Lem:hk} (a) and Lemma \ref{Lem:RSbound} (a) to find
\begin{alignat*}{1} 
\left| \int_{\Omega_0} \% \right| &\lesssim \Vert \Phi(x_0 - x,t_0 - t) \Vert_{L^{5/3}(\Omega_0)} \Vert R_u + I_u \Vert_{L^{5/2}(\Omega_0)} \\
& \qquad + \Vert \Phi'(x_0 - x, t_0 -t) \Vert_{L^{5/4} (\Omega_0)} \Vert S_u + J_u \Vert_{L^5(\Omega_0)} \displaybreak[1] \\
 &\lesssim r_0^{4/5} \bigl( \Vert R_u \Vert_{L^{5/2}(\Omega_0)} + \Vert I_u \Vert_{L^{5/2}(\Omega_0)} \bigr) + r_0^{2/5} \bigl( \Vert S_u \Vert_{L^5(\Omega_0)} + \Vert J_u \Vert_{L^5(\Omega_0)} \bigr) \\
& \lesssim \alpha^2 + \gamma^2 + \gamma_v + H.
\end{alignat*}
On $\Omega_-$ we apply Lemma \ref{Lem:hk} (b)
\[ \left| \int_{\Omega_-} \% \right| \lesssim \int_{\Omega_-} e^{-c|x-x_0|/r_0} \bigl( r_0^{-1} (|R_u| + |I_u|) + r_0^{-2} (|S_u| + |J_u|) \bigr) dx dt. \]
Now choosing $b = (2(n-1)^{-1} \sigma^{-1} + 1)^{1/2}$, we can guarantee that $\Omega_0 \subset P_{br_0}(x_0)$.
So $\Omega_- \subset \bigcup_{k=0}^\infty P_{br_0}(x_0-2k b r_0)$, and hence using Lemma \ref{Lem:RSbound} (b), we can bound the integral above by (observe that $r_0 \geq 4 \sigma$)
\begin{multline*} 
\sum_{k=0}^\infty \int_{P_{br_0}(x_0 - 2kbr_0) \cap \IR \times [0,t_0]} e^{-2cb (k-1)}  \bigl( r_0^{-1} (|R_u| + |I_u|) + r_0^{-2} (|S_u| + |J_u|) \bigr) \\
 \lesssim \alpha^2 + \beta^2 + \beta_v + H.
\end{multline*}
For the integral over the domain $\Omega_+$ we have to be a bit more careful since the local scale needs to change with $x$.
Set for $k \geq 1$
\[ y_k = r_0 k^2 + x_0, \qquad r_k = r_0 (k+1), \qquad z_k = \sfrac12 (y_k + y_{k+1}) \]
We check that $D \cap [y_k, y_{k+1}] \times [0,t_0] \subset P_{br_k}(z_k)$:
First note that
\[ y_{k+1} - y_k = r_0 (2k+1) < 2 b r_0 (k+1) = 2 b r_k. \]
Secondly, we show that $b^2 r_k^2 - (n-1)^{-1} z_k^- \geq t_0$.
Since $\Omega_0 \subset P_{br_0}(x_0)$, we already know that $b^2 r_0^2 - (n-1)^{-1} x_0^- \geq t_0$.
Then with
\[ b^2 (r_k^2 - r_0^2) = b^2 r_0^2 (k^2 + 2k) \geq b^2 \sigma r_0 (k^2 + k + {\textstyle \frac12} ) \geq (n-1)^{-1} (z_k - x_0) \]
the desired inequality follows.

Now observe that $\frac12 r_k \leq y_k - x_0$ and $(y_k-x_0)/ (\frac12 r_k) \geq k$.
So by Lemma \ref{Lem:hk} (b) (with $\ov r= \frac12 r_k$) we have for $(x,t) \in D \cap [y_k,y_{k+1}] \times [0,t_0]$
\[ |\Phi|(x_0 - x, t_0 - t) \lesssim r_k^{-1} e^{- c (y_k - x_0)/(\frac12 r_k)} \leq r_k^{-1} e^{-ck}, \quad |\Phi'|(x_0 - x, t_0 - t) \lesssim r_k^{-2} e^{- c k}. \]
So we can split up the integral over $\Omega_+$ and conclude
\begin{multline*}
 \left| \int_{\Omega_+} \% \right| \leq \sum_{k=1}^{\infty} \biggl| \int_{D \cap [y_k,y_{k+1}] \times [0,t_0]} \% \biggr| \\
  \lesssim \sum_{k=1}^{\infty} e^{-ck} \int_{P_{br_k}(z_k)} r_k^{-1} ( |R_u| + |I_u| ) + r_k^{-2} (|S_u| + |J_u|) \lesssim \alpha^2 + \beta^2 + \beta_v + H.
\end{multline*}

The bound on $v(x_0,t_0)$ is derived in the same way.
Observe here that the bounds which were used for $\Phi$, also apply for $\ov\Phi$.
\end{proof}

\begin{Lemma} \label{Lem:beta}
If $\alpha < 0.1$, then we have
\begin{alignat*}{1}
\beta_u &\lesssim \alpha^2 + \beta^2 + \alpha_u + \alpha_u^{1/2} \left( \alpha + \beta + \beta_v^{1/2} + H^{1/2} \right) + H\\
\beta_v &\lesssim  \alpha^2 + \beta^2 + \alpha_v + \beta_u H + \alpha_v^{1/2}\left( \alpha + \beta + \beta_u^{1/2} H^{1/2} + H^{1/2} \right) + H
\end{alignat*}
\end{Lemma}

\begin{figure}[bt]
\caption{The parabolic domains used in the proof of Lemma \ref{Lem:beta}.}
\label{fig:beta}
\begin{center}
\begin{picture}(0,0)%
\includegraphics[width=12cm]{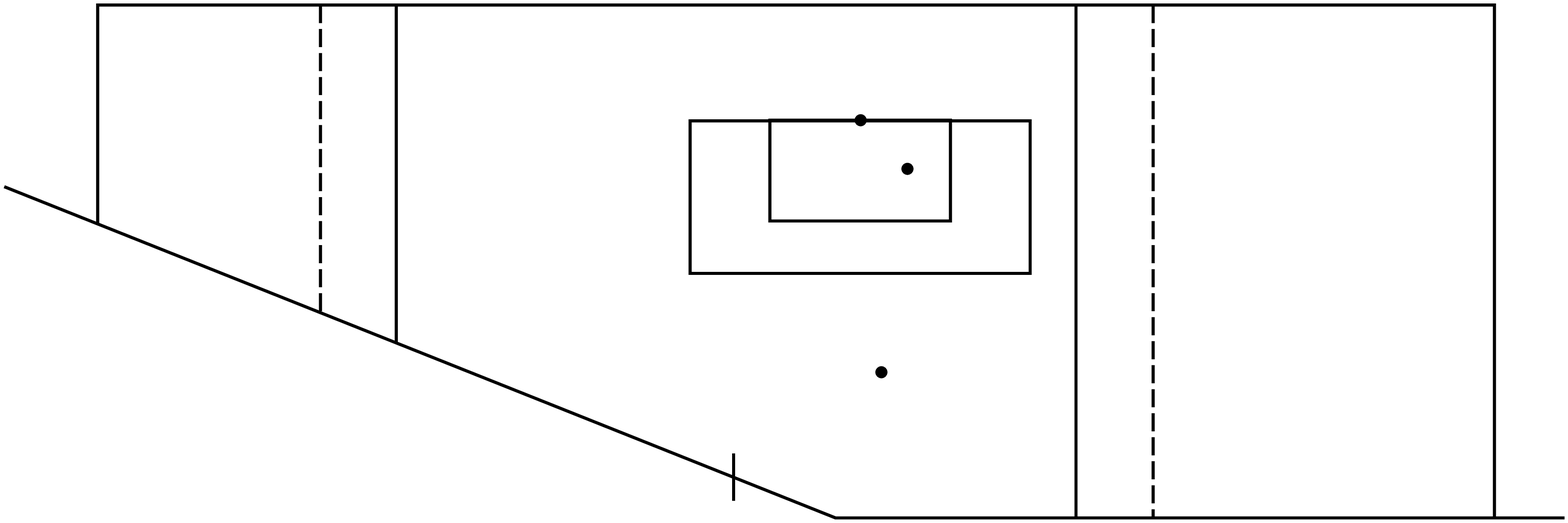}%
\end{picture}%
 \setlength{\unitlength}{3947sp}%
 \setlength{\unitlength}{0.6\unitlength}
 \begingroup\makeatletter\ifx\SetFigFont\undefined%
 \gdef\SetFigFont#1#2#3#4#5{%
 \reset@font\fontsize{#1}{#2pt}%
 \fontfamily{#3}\fontseries{#4}\fontshape{#5}%
 \selectfont}%
 \fi\endgroup%
 \begin{picture}(9800,3500)(518,-4964) \put(930,-4210){\makebox(0,0)[lb]{\smash{{\SetFigFont{12}{14.4}{\familydefault}{\mddefault}{\updefault}$x=-(n-1) t$}}}}
 \put(600,-2000){\makebox(0,0)[lb]{\smash{{\SetFigFont{12}{14.4}{\familydefault}{\mddefault}{\updefault}$t_0$}}}}
 \put(4400,-5000){\makebox(0,0)[lb]{\smash{{\SetFigFont{12}{14.4}{\familydefault}{\mddefault}{\updefault}$x_0$}}}}
 \put(5100,-2400){\makebox(0,0)[lb]{\smash{{\SetFigFont{12}{14.4}{\familydefault}{\mddefault}{\updefault}$(x_1,t_1 + \frac{\sigma^2}2)$}}}}
 \put(5100,-2900){\makebox(0,0)[lb]{\smash{{\SetFigFont{12}{14.4}{\familydefault}{\mddefault}{\updefault}$P_1$}}}}
 \put(4600,-2900){\makebox(0,0)[lb]{\smash{{\SetFigFont{12}{14.4}{\familydefault}{\mddefault}{\updefault}$P'_1$}}}}
 \put(1540,-2900){\makebox(0,0)[lb]{\smash{{\SetFigFont{12}{14.4}{\familydefault}{\mddefault}{\updefault}$P'$}}}}
 \put(2400,-2900){\makebox(0,0)[lb]{\smash{{\SetFigFont{12}{14.4}{\familydefault}{\mddefault}{\updefault}$P_0'$}}}}
 \put(2900,-2900){\makebox(0,0)[lb]{\smash{{\SetFigFont{12}{14.4}{\familydefault}{\mddefault}{\updefault}$P_0$}}}}
 \put(5450,-3100){\makebox(0,0)[lb]{\smash{{\SetFigFont{12}{14.4}{\familydefault}{\mddefault}{\updefault}$(x,t)$}}}}
 \put(5300,-4400){\makebox(0,0)[lb]{\smash{{\SetFigFont{12}{14.4}{\familydefault}{\mddefault}{\updefault}$(x',t')$}}}}
 \put(5200,-5300){\makebox(0,0)[lb]{\smash{{\SetFigFont{12}{14.4}{\familydefault}{\mddefault}{\updefault}$(0,0)$}}}}
 \put(6500,-5300){\makebox(0,0)[lb]{\smash{{\SetFigFont{12}{14.4}{\familydefault}{\mddefault}{\updefault}$x_0 + r_0$}}}}
 \put(8800,-5300){\makebox(0,0)[lb]{\smash{{\SetFigFont{12}{14.4}{\familydefault}{\mddefault}{\updefault}$x_0 + 2r_0$}}}}
 \end{picture}%
\end{center}
\vspace{7mm}
\end{figure}

\begin{proof}
We first bound $\beta_u$.
Let $x_0 \in \IR$ and $r_0 \geq \sigma$ be given and set $t_0 = \min \{ r_0^2 - (n-1)^{-1} x_0^-, T \}$.
Let $\eta : \IR \to [0,1]$ be a cutoff function which is $\equiv 1$ on $[-1,1]$ and $\equiv 0$ outside $[-2,2]$ and set $\eta_{r_0, x_0}(x) = \eta((x-x_0)/r_0)$.
Then (\ref{eq:flow1}) and integration by parts gives for each $t \in [0,t_0)$:
\begin{align*}
& \partial_t \int_{-(n-1) t}^\infty \eta^2_{r_0, x_0} |u|^2 + \int_{-(n-1) t}^\infty \eta^2_{r_0, x_0} |u'|^2   \\
& \quad = (n-1) \bigl( \eta^2_{r_0, x_0} |u|^2 \bigr) (- (n-1) t, t)
+ \int_{-(n-1) t}^\infty \eta^2_{r_0, x_0} \big(2 u u'' + 2 u (R_u + I_u)  \\
&\quad\qquad\hspace{9cm} + 2 u (S'_u + J'_u) + |u'|^2  \big) \displaybreak[1] \\ 
&\quad = (n-1) \bigl( \eta^2_{r_0, x_0} |u|^2 \bigr) (-(n-1) t, t) - 2 \eta_{r_0, x_0}^2 (u u' + u S_u + u J_u ) (-(n-1) t, t) \\
& \quad\qquad - 4 \int_{-(n-1) t}^\infty   \eta_{r_0, x_0} \eta_{r_0, x_0}' ( u u' + u S_u + u J_u)  \\
& \quad\qquad + \int_{-(n-1) t}^\infty \eta^2_{r_0,x_0} \big(- |u'|^2 + 2u (R_u + I_u) - 2u' (S_u + J_u) \big) 
\intertext{The first two terms can be bounded by $C H^2 \eta_{r_0, x_0}^2(-(n-1)t,t)$ and using the fact that the integrand of the third term is bounded by $\eta_{r_0,x_0}^2 (|u'|^2 + |S_u|^2 + |J_u|^2) + 12 (\eta'_{r_0, x_0})^2 |u|^2$, we continue} \displaybreak[1]
& \quad \lesssim H^2 \eta_{r_0, x_0}^2(-(n-1) t,t) + \int_{x_0 - 2r_0}^{x_0 + 2r_0} \frac1{r_0^2} |u|^2 + |u| (|R_u| + |I_u|) + |S_u|^2 + |J_u|^2. 
\end{align*}
Integrating this over $t$ from $0$ to $t_0$ and using Lemma \ref{Lem:RSbound} (b) yields for $P' = [x_0 - 2r_0, x_0 + 2r_0] \times [0,t_0] \cap D$
\begin{multline*}
 \Vert u' \Vert^2_{L^2(P_{r_0}(x_0))} \lesssim r_0 H^2 + r_0 \alpha_u^2 + \alpha_u \int_{P'} \bigl( |R_u| + |I_u| \bigr) + \int_{P'} \bigl( |S_u|^2 + |J_u|^2 \bigr) \\
 \lesssim r_0 H^2 + r_0 \alpha_u^2 + r_0 \alpha_u \left( \alpha^2 + \beta^2 + \beta_v + H \right) + r_0 \left( (\alpha^2 + \beta^2)^2 + H^2 \right) .
\end{multline*}
This establishes the bound on $\beta_u$.
If we carry out the same procedure for equation (\ref{eq:flow2}) instead of (\ref{eq:flow1}), we get an additional $\Vert v \Vert_{L^2(P_{r_0}(x_0))}$-term on the left hand side by the exponential decay property of the linearization.
Moreover, there will be no $\beta_v$-term on the right hand side, but the extra $u' * J_u$-term produces a $\beta_u H$-term:
\begin{multline*}
 \Vert v' \Vert^2_{L^2(P_{r_0}(x_0))} + \Vert v \Vert^2_{L^2(P_{r_0}(x_0))} \\ \lesssim r_0 H^2 + r_0 \alpha_v^2 + r_0 \alpha_v \left( \alpha^2 + \beta^2 + \beta_u H + H \right) + r_0 \left( (\alpha^2 + \beta^2)^2 + H^2 \right)
\end{multline*}

It remains to bound $r_0^{-1} \Vert v \Vert_{L^1_{\mu_1}(P_{r_0}(x_0))}$, $r_0^{-1/2} \Vert v \Vert_{L^2_{\mu_2}(P_{r_0}(x_0))}$ and $r_0^{-1} \Vert v' \Vert_{L^1_{\mu_1}(P_{r_0}(x_0))}$.
We first establish the corresponding bounds for $v^{**}$.
For this note that for any $(x,t) \in D$ with $r(x,t) \geq \sigma$, the vertical distance $s$ to the parabolic boundary $\partial_p D$ can be estimated by $s = t - (n-1)^{-1} x^- \leq C r^2(x,t)$.
The bounds for $v^{**}$ then follow from Lemma \ref{Lem:starstar} (c) and the fact that $\int_0^\infty \exp ( - s^{1/2} ) d s < \infty$.

It remains to establish the bounds for $v^*$.
We first discuss the bound on $r_0^{-1} \Vert v^* \Vert_{L^1_{\mu_1}(P_{r_0}(x_0))}$.
Choose $(x_1, t_1) \in P_0 = P_{r_0}(x_0)$ and set $P_1 = [x_1 - \sigma, x_1 + \sigma] \times [t_1 - \frac12 \sigma^2, t_1 + \frac12 \sigma^2] \cap P_0$ and $P'_1 = [x_1 - 2 \sigma, x_1 + 2\sigma] \times [t_1- \frac32 \sigma^2, t_1 + \frac12 \sigma^2] \cap P_0'$ where $P'_0 = [x_0-r_0-\sigma,x_0+r_0+\sigma] \times [0,t_0] \cap D$.
Let $(x,t) \in P_1$. 
By (\ref{eq:vstar}) we have
\begin{multline} v^*(x,t) = \int_{D \cap \IR \times [0,t]} \varphi^2 \bigl[ \ov \Phi(x-x',t-t') \bigl(R_v + I_v + u' * J_u \bigr)\\
+ \ov \Phi'(x-x',t-t') \bigl( S_v + J_v \bigr) \bigr] dx' dt'. \label{eq:no1}
\end{multline}
We can represent $v^* = v_1 + v_2$, where $v_1$ denotes the integral above over the domain $P'_1$ and $v_2$ the integral over the domain $D \cap \IR \times [0,t] \setminus P'_1$.
Since $\Vert \ov \Phi \Vert_{L^{1+ \mu}([-3\sigma,3\sigma] \times [0, 2\sigma^2])}$ and $\Vert \ov \Phi' \Vert_{L^{1+\mu}([-3\sigma,3\sigma] \times [0, 2\sigma^2])}$ are finite for $\mu < 1/2$, Young's inequality yields
\begin{equation}  
 \Vert v_1 \Vert_{L^{1+\mu_1}(P_1)} \lesssim \Vert R_v + I_v + u' * J_u \Vert_{L^1(P'_1)} + \Vert S_v + J_v \Vert_{L^1(P'_1)} \label{eq:no2} 
\end{equation}
We now integrate both sides over $(x_1,t_1) \in P_0$ and obtain by Lemma \ref{Lem:RSbound} (b)
\begin{alignat}{1} 
\Vert v_1 \Vert_{L^1_{\mu_1}(P_0)} &\lesssim \Vert R_v + I_v + u' * J_u \Vert_{L^1(P'_0)} +  \Vert S_v + J_u \Vert_{L^1(P'_0)} \notag \\
& \qquad \lesssim r_0 \left( \alpha^2 + \beta^2 + \beta_u H + H \right) \label{eq:no3}
\end{alignat}

We will now bound $v_2$.
Fix $(x_1,t_1) \in P_0$ again. 
By  Lemma \ref{Lem:hk} (c)
\[ |v_2|(x,t) \lesssim \int_{D \cap \IR \times [0,t]} e^{-c|x-x'|-c|t-t'|} \left( |R_v + I_v + u' * J_u| + |S_v + J_v| \right) dx' dt' \]
for all $(x,t) \in P_1$ and hence
\begin{multline*}
 \Vert v_2 \Vert_{L^{1+\mu_1}(P_1)} \lesssim \int_{D \cap \IR \times [0,t_1]} e^{-c|x_1-x'|-c|t_1-t'|} \times \\
\bigl( |R_v + I_v + u' * J_u| + |S_v + J_v| \bigr) dx' dt' = \int_{P' \cap \IR \times [0,t_1]} \% + \int_{(D \setminus P') \cap \IR \times [0,t_1]} \%.
\end{multline*}
Now let $(x_1,t_1)$ vary over $P_0$ and compute the $L^1$-norm of $\Vert v_2 \Vert_{L^{1+\mu_1}(P_1)}$ (recall that $P_1$ depends on $(x_1,t_1)$).
By the inequality above, we can estimate this $L^1$-norm by the sum of both $L^1$-norms of the two integrals on the right hand side.
Since $e^{-c |x_1 - x'| - c|t_1-t'|}$ is bounded in $L^1$, we can use Young's inequality to bound the $L^1$-norm of the first integral by
\[ \Vert |R_v + I_v + u' * J_u| + |S_v + J_v| \Vert_{L^1(P')} \lesssim r_0 (\alpha^2 + \beta^2 + \beta_u H + H). \]
As for the second integral, it suffices to show the even stronger $L^\infty$-bound
\[ \biggl| \int_{(D \setminus P') \cap \IR \times [0,t_1]} \% \biggr| \lesssim r_0^{-2} (\alpha^2 + \beta^2 + \beta_u H + H). \]
In order to derive this inequality, we cover the domain $(D \setminus P') \cap \IR \times [0,t_1]$ by regions $P_{br_0}(x_0 - 2k br_0)$ and $P_{br_k}(z_k)$ as in the proof of Lemma \ref{Lem:alpha}.
Note that we have since $r_0 \geq \sigma$
\[ e^{-c|x_1 - x| - c|t_1-t|} \lesssim r_0^{-3} e^{-c k} \qquad \text{on} \qquad P_{b r_0}(x_0 - 2k b r_0) \setminus P' \]
for $k \geq 0$ and
\begin{multline*} e^{-c|x_1 - x| - c|t_1 - t|} \leq \min \{ e^{-cr_0}, e^{-c r_0 k^2 + c r_0} \} \leq e^{- c' r_0 - c' r_0 (k+1) - c' \sigma k} \\ \lesssim r_0^{-2} r_k^{-1} e^{-c'k} \qquad \text{on} \qquad [y_k, y_{k+1}] \times [0,t_1] \setminus P'
\end{multline*}
for $k \geq 1$.
So
\begin{multline*} \biggl| \int_{(D \setminus P') \cap \IR \times [0,t_1]} \% \biggr| \lesssim \sum_{k=0}^\infty \int_{P_{br_0}(x_0 - 2kbr_0)}   r_0^{-3} e^{-ck} (|R_v + I_v + u' * J_u| + |S_v + J_v|)  \\
+ \sum_{k=1}^\infty \int_{P_{br_k}(z_k)} r_0^{-2} r_k^{-1} e^{-c'k} (|R_v + I_v + u' * J_u| + |S_v + J_v|) \\
 \lesssim r_0^{-2} \left( \alpha^2 + \beta^2 + \beta_u H + H \right).
\end{multline*}
This establishes the required bound.

Next, we establish the bound on $r_0^{-1/2} \Vert v \Vert_{L^2_{\mu_2}(P_{r_0}(x_0))}$.
Observe for this that by the argument above with $\mu_1$ replaced by $\mu_2$, we obtain
\[ r_0^{-1} \Vert v \Vert_{L^1_{\mu_2}(P_{r_0}(x_0))} \lesssim \alpha^2 + \beta^2 + \beta_u H + H. \]
Using the interpolation inequality, we can conclude
\begin{multline*}
 r_0^{-1/2} \Vert v \Vert_{L^2_{\mu_2}(P_{r_0}(x_0))} \leq \Vert v \Vert_{L^\infty(P_{r_0}(x_0))}^{1/2} \big( r_0^{-1} \Vert v \Vert_{L^1_{\mu_2}(P_{r_0}(x_0))} \big)^{1/2} \\
 \lesssim \alpha_v^{1/2} \left( \alpha^2 + \beta^2 + \beta_u H + H \right)^{1/2}.
\end{multline*}

Finally, we explain how the bound on $r_0^{-1} \Vert (v^*)' \Vert_{L^1_{\mu_1}(P_{r_0}(x_0))}$ is derived.
The argument is almost the same as for $r_0^{-1} \Vert v^* \Vert_{L^1_{\mu_1}(P_{r_0}(x_0))}$ with the following modifications:
In (\ref{eq:no1}) we have to replace $\ov \Phi$ by $\ov \Phi'$ and $\ov \Phi'$ by $\ov \Phi''$.
In the estimate (\ref{eq:no2}) for $\Vert v_1' \Vert_{L^{1+\mu_1}(P_1)}$ we now have to apply Lemma \ref{Lem:CZ} on the second term to find
\[ \Vert v'_1 \Vert_{L^{1+\mu_1}(P_1)} \lesssim \Vert R_v + I_v + u' * J_u\Vert_{L^1(P'_1)} + \Vert S_v + J_v \Vert_{L^{1+\mu_1}(P'_1)} \]
and thus in (\ref{eq:no3}), we get
\begin{align*}
\Vert v'_1 \Vert_{L^1_{\mu_1}(P_0)} &\lesssim \Vert R_v + I_v + u' * J_u\Vert_{L^1(P'_0)} + \Vert S_v + J_v \Vert_{L^1_{\mu_1}(P'_0)} \\
&\qquad \lesssim r_0 (\alpha^2 + \beta^2 + \beta_u H + H).
\end{align*}
The estimate on $v_2'$ stays the same.
\end{proof}

\begin{Lemma} \label{Lem:gamma}
If $\alpha < 0.1$, then we have
\begin{alignat*}{1}
 \gamma_u &\lesssim \alpha^2 + \beta^2  + \gamma^2 + \beta_v + \gamma_v + H \\
 \gamma_v &\lesssim \alpha^2 + \beta^2  + \gamma^2 + \alpha_v + (\beta_u + \gamma_u)H +  H
\end{alignat*}
\end{Lemma}

\begin{figure}[bt]
\caption{The parabolic domains used in the proof of Lemma \ref{Lem:gamma}.}
\label{fig:gamma}
\begin{center}
\begin{picture}(0,0)%
\hspace{2mm}\includegraphics[width=14cm]{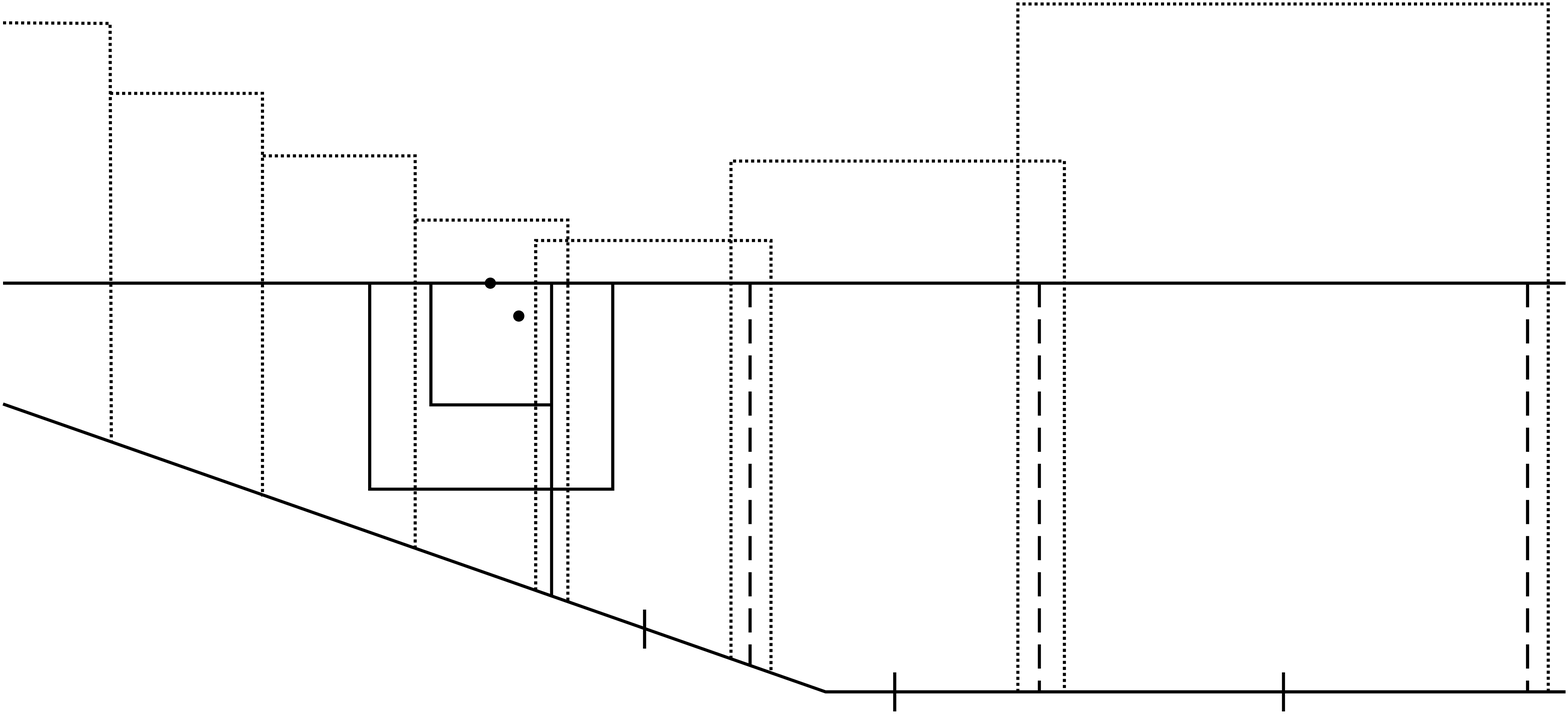}%
\end{picture}%
\setlength{\unitlength}{2863sp}%
\begingroup\makeatletter\ifx\SetFigFont\undefined%
\gdef\SetFigFont#1#2#3#4#5{%
  \reset@font\fontsize{#1}{#2pt}%
  \fontfamily{#3}\fontseries{#4}\fontshape{#5}%
  \selectfont}%
\fi\endgroup%
\begin{picture}(11195,4100)(518,-5050)
\put(1000,-4210){\makebox(0,0)[lb]{\smash{{\SetFigFont{12}{14.4}{\familydefault}{\mddefault}{\updefault}$x= -(n-1) t$}}}}
\put(3300,-3200){\makebox(0,0)[lb]{\smash{{\SetFigFont{12}{14.4}{\familydefault}{\mddefault}{\updefault}$\Omega$}}}}
\put(3300,-3690){\makebox(0,0)[lb]{\smash{{\SetFigFont{12}{14.4}{\familydefault}{\mddefault}{\updefault}$\Omega'$}}}}
\put(2350,-3750){\makebox(0,0)[lb]{\smash{{\SetFigFont{12}{14.4}{\familydefault}{\mddefault}{\updefault}$\scriptstyle \frac34 r_0^2$}}}}
\put(2350,-3300){\makebox(0,0)[lb]{\smash{{\SetFigFont{12}{14.4}{\familydefault}{\mddefault}{\updefault}$\scriptstyle \frac12 r_0^2$}}}}
\put(3350,-2450){\makebox(0,0)[lb]{\smash{{\SetFigFont{12}{14.4}{\familydefault}{\mddefault}{\updefault}$(x_0,t_0)$}}}}
\put(3200,-2800){\makebox(0,0)[lb]{\smash{{\SetFigFont{12}{14.4}{\familydefault}{\mddefault}{\updefault}$\scriptstyle (x,t)$}}}}
\put(2600,-2480){\makebox(0,0)[lb]{\smash{{\SetFigFont{12}{14.4}{\familydefault}{\mddefault}{\updefault}$\scriptstyle 2r_0$}}}}
\put(3050,-2480){\makebox(0,0)[lb]{\smash{{\SetFigFont{12}{14.4}{\familydefault}{\mddefault}{\updefault}$\scriptstyle r_0$}}}}
\put(1640,-1250){\makebox(0,0)[lb]{\smash{{\SetFigFont{12}{14.4}{\familydefault}{\mddefault}{\updefault}$P_{br_0}(x_0-2kbr_0)$}}}}
\put(4000,-4150){\makebox(0,0)[lb]{\smash{{\SetFigFont{12}{14.4}{\familydefault}{\mddefault}{\updefault}$P_{br_1}(z_1)$}}}}
\put(5400,-4150){\makebox(0,0)[lb]{\smash{{\SetFigFont{12}{14.4}{\familydefault}{\mddefault}{\updefault}$P_{br_2}(z_2)$}}}}
\put(7600,-4150){\makebox(0,0)[lb]{\smash{{\SetFigFont{12}{14.4}{\familydefault}{\mddefault}{\updefault}$P_{br_3}(z_3)$}}}}
\put(4310,-4860){\makebox(0,0)[lb]{\smash{{\SetFigFont{12}{14.4}{\familydefault}{\mddefault}{\updefault}$z_1$}}}}
\put(5750,-5230){\makebox(0,0)[lb]{\smash{{\SetFigFont{12}{14.4}{\familydefault}{\mddefault}{\updefault}$z_2$}}}}
\put(8030,-5230){\makebox(0,0)[lb]{\smash{{\SetFigFont{12}{14.4}{\familydefault}{\mddefault}{\updefault}$z_3$}}}}
\put(3710,-4660){\makebox(0,0)[lb]{\smash{{\SetFigFont{12}{14.4}{\familydefault}{\mddefault}{\updefault}$y_1$}}}}
\put(4950,-5050){\makebox(0,0)[lb]{\smash{{\SetFigFont{12}{14.4}{\familydefault}{\mddefault}{\updefault}$y_2$}}}}
\put(6630,-5230){\makebox(0,0)[lb]{\smash{{\SetFigFont{12}{14.4}{\familydefault}{\mddefault}{\updefault}$y_3$}}}}
\put(9430,-5230){\makebox(0,0)[lb]{\smash{{\SetFigFont{12}{14.4}{\familydefault}{\mddefault}{\updefault}$y_4$}}}}
\end{picture}%
\end{center}
\vspace{2mm}
\end{figure}

\begin{proof}
 Let $(x_0,t_0) \in D$ be given and set $r_0 = r(x_0,t_0)$ as well as $\Omega = Q_{r_0}(x_0)$.
If $r_0 < 4 \sigma$, then $\Omega \subset B_{9\sigma}(\partial_p D)$ and the bounds follow easily.
So assume that $r_0 \geq 4 \sigma$.

We first derive the bound on $\gamma_u$.
As in subsection \ref{subsec:ususs}, we decompose $u = u^* + u^{**}$ where $(u^{**})'$ satisfies the required bound by Lemma \ref{Lem:starstar} (b).
So we only have to derive the bound for $(u^*)'$.
For any $(x,t) \in \Omega$ we have by (\ref{eq:ustar})
\begin{multline*} 
(u^*)' (x,t) = \int_{D} \varphi^2 \bigl[ \Phi'(x-x',t-t') (R_u + I_u) \\ 
+ \Phi''(x-x',t-t')(S_u + J_u) \bigr] dx' dt'.
\end{multline*}
Represent $(u^*)' = u'_1 + u'_2$ where $u'_1$ denotes the integral above over the domain $\Omega' = [x_0 - 2 r_0, x_0 + 2 r_0] \times [t_0 - \frac34 r_0^2, t_0]$ and $u'_2$ the integral over the domain $D \setminus \Omega'$.
Using Young's inequality, Lemma \ref{Lem:hk} (a), Lemma \ref{Lem:CZ} and Lemma \ref{Lem:RSbound} (a), we find
\begin{multline}
 \Vert u'_1 \Vert_{L^5(\Omega)} \lesssim r_0^{2/5} \Vert R_u + I_u \Vert_{L^{5/2}(\Omega')} + \Vert S_u + J_u \Vert_{L^5(\Omega')} \\ \lesssim r_0^{-2/5} \left( \alpha^2 + \gamma^2 + \gamma_v + H \right). \label{eq:no4}
\end{multline}

The bound on $u'_2(x,t)$ is derived analogously as the bound on $\int_{\Omega_- \cup \Omega_+} \%$ in Lemma \ref{Lem:alpha}:
Observe that $D \cap \IR \times [0, t_0]  \subset \bigcup_{k=0}^\infty P_{br_0}(x_0 - 2kbr_0) \cup \bigcup_{k=1}^\infty P_{br_k}(z_k)$.
We estimate the heat kernel on each of these domains away from $\Omega'$.
By Lemma \ref{Lem:hk} (b) (with $\ov{r} = \frac12 r_0$) we have
\begin{multline*}
 |\Phi'|(x-x',t-t') \lesssim r_0^{-2} e^{-ck} \quad \text{and} \quad |\Phi''|(x-x',t-t') \lesssim r_0^{-2} e^{-ck} \\ \qquad \text{for} \qquad (x',t') \in P_{b r_0}(x_0 - 2kbr_0) \setminus \Omega'.
\end{multline*}
Moreover, again by Lemma \ref{Lem:hk} (b) (with $\ov{r} = \max \{ \frac12 r_0, y_k - x_0 - r_0 \}$)
\begin{multline*}
|\Phi'|(x-x',t-t') \lesssim (r_0 k)^{-2} e^{-c k} \lesssim r_0^{-1} r_k^{-1} e^{-c k}  \quad \text{and} \\ |\Phi''|(x-x',t-t') \lesssim (r_0 k)^{-3} e^{-c k} \lesssim r_0^{-2} r_k^{-1} e^{-c k} \quad \text{for} \quad (x',t') \in [y_k, y_{k+1}] \times [0,t_0] \setminus \Omega'.
\end{multline*}
Hence, we obtain
\begin{multline*} |u'_2|(x,t) \leq \sum_{k=0}^\infty \int_{P_{br_0}(x_0 - 2 k b r_0)} r_0^{-2} e^{-c k} \big(|R_u| + |I_u| + |S_u| + |J_u| \big) \\
+ \sum_{k=1}^\infty \int_{P_{br_k}(z_k)} r_0^{-1} r_k^{-1} e^{-c k} \big(|R_u| + |I_u| + |S_u| + |J_u| \big) \\
\lesssim r_0^{-1} \big( \alpha^2 + \beta^2 + \beta_v + H \big)
\end{multline*}
and thus
\[ \Vert u'_2 \Vert_{L^5(\Omega)} \lesssim r_0^{-2/5} \left( \alpha^2 + \beta^2 + \beta_v + H \right). \]
Hence, we have bounded $\gamma_u$.

We now bound $\gamma_v$.
The bound on $\Vert v' \Vert_{L^5(\Omega)}$ is derived in the same way as above with $\Phi$ replaced by $\ov\Phi$.

The bound on $\Vert v' \Vert_{L^{5/2}(\Omega)}$ also follows by the same arguments except that in (\ref{eq:no4}) we now have to use the $L^1$-boundedness of $\ov\Phi'$ in Young's inequality and Lemma \ref{Lem:CZ} to show
\begin{multline*}
 \Vert v'_1 \Vert_{L^{5/2}(\Omega)} \lesssim \Vert R_v + I_v + u' * J_u \Vert_{L^{5/2}(\Omega')} + \Vert S_v + J_v \Vert_{L^{5/2}(\Omega')} \\ \lesssim r_0^{-4/5} \left( \alpha^2 + \gamma^2 + \gamma_u H + H \right).
\end{multline*}
The exponential decay of the heat kernel $\ov \Phi$ now allows us to use a higher power of $r_0^{-1}$ in the bound for $v'_2$ (see the remark in Lemma \ref{Lem:hk} (b)):
\[ |v'_2|(x,t) \lesssim r_0^{-2} \left( \alpha^2 + \beta^2 + \beta_u H + H \right) \]

As for $\Vert v \Vert_{L^{5/2}(\Omega)}$, we replace $\ov \Phi'$ by $\ov \Phi$ and $\ov \Phi''$ by $\ov \Phi'$ in the argument above.
Again, using the $L^1$-boundedness of $\ov\Phi$ and $\ov\Phi'$, we get
\begin{multline*}
 \Vert v_1 \Vert_{L^{5/2}(\Omega)} \lesssim \Vert R_v + I_v + u' * J_u   \Vert_{L^{5/2}(\Omega')} + \Vert S_v + J_v \Vert_{L^{5/2}(\Omega')} \\
  \lesssim r_0^{-4/5} \left( \alpha^2 + \gamma^2 + \gamma_u H + H \right).
\end{multline*}
And by the same reason as before, we can bound 
\[ |v_2|(x,t) \lesssim r_0^{-2} \left(\alpha^2 + \beta^2 + \beta_u H + H \right). \]

Finally, for $\Vert v \Vert_{L^5(\Omega)}$ we make use of the interpolation inequality
\[ \Vert v \Vert_{L^5(\Omega)} \leq \Vert v \Vert_{L^\infty(\Omega)}^{1/2} \Vert v \Vert_{L^{5/2}(\Omega)}^{1/2} \leq r_0^{-2/5} \alpha_v + r_0^{2/5} \Vert v \Vert_{L^{5/2}(\Omega)}. \qedhere \]
\end{proof}

We can now use these inequalities to prove Theorem \ref{Thm:cuspinv}.
\begin{proof}[Proof of Theorem \ref{Thm:cuspinv}]
As pointed out in subsection \ref{subsec:Calculations}, it suffices to establish Proposition \ref{Prop:cuspinv2}.

Allow again $T'$ to vary, i.e. $\sigma^2 \leq T' \leq T$ and set $\chi_{T'} = \alpha_{T'} + \beta_{T'} + \gamma_{T'}$.
As long as $\chi < 0.1$, we conclude from Lemmas \ref{Lem:alpha}, \ref{Lem:beta}, \ref{Lem:gamma}
\begin{alignat*}{1}
 \alpha_u &\lesssim \chi^2 + \beta_v + \gamma_v + H \\
 \alpha_v &\lesssim \chi^2 + H \displaybreak[1] \\
 \beta_u &\lesssim \chi^{3/2} + \alpha_u + \beta_v + H \\
 \beta_v &\lesssim \chi^{3/2} + \alpha_v + H \displaybreak[1]\\
 \gamma_u &\lesssim \chi^2 + \beta_v + \gamma_v + H \\
 \gamma_v &\lesssim \chi^2 + \alpha_v + H
\end{alignat*}
Plug the second inequality into the fourth and sixth to get $\alpha_v, \beta_v, \gamma_v \lesssim \chi^{3/2} + H$.
This implies with the first and fifth inequality $\alpha_u, \gamma_u \lesssim \chi^{3/2} + H$.
Eventually, plugging everything into the third inequality, yields
\begin{equation}  \chi_{T'} \leq C_0 ( \chi_{T'}^{3/2} + H ). \label{eq:chi} \end{equation}
By the hypothesis of the Proposition $\chi_{\sigma^2} \leq C H$.
Choose $\varepsilon_0 = (2 C_0)^{-2}$ and assume that $\varepsilon_0 < 0.1$.
Moreover, set $\varepsilon_{inv} =  \min \{ (2 C_0)^{-1}, C^{-1} \} \varepsilon_0$ and assume $H < \varepsilon_{inv}$.
This implies $\chi_{\sigma^2} <  \varepsilon_0$.

We conclude now that $\chi_{T} < \varepsilon_0$:
Observe that $\chi_{T'}$ is continous in $T'$.
So if the hypothesis was wrong, then there would be some time $T' \in (\sigma^2, T]$ with $\chi_{T'} = \varepsilon_0$, and this would imply
\[ \varepsilon_0 < C_0 (\varepsilon_0^{3/2} + \varepsilon_{inv}) \leq \tfrac12 \varepsilon_0  + \tfrac12 \varepsilon_0. \] 
Plugging the bound $\chi_{T} < \varepsilon_0$ into (\ref{eq:chi}) yields
\[ \chi_{T} \leq 2 C_0 H \]
and hence the Proposition.
\end{proof}

\section{Ricci flow on the whole manifold} \label{sec:whole}
In this section we finally present the proof of Theorem \ref{Thm:main}.
In the following, denote the given hyperbolic manifold by $(M, \ov g)$.
The constants $\varepsilon$ and $C$ will only depend on $\sigma$ and an upper bound on $\vol M$ in dimension $n \geq 4$ resp. and upper bound on $\diam M_{cpt}$ in dimension $n = 3$.

\subsection{The heat kernel estimate}
Let $E = \Sym_2 T^* M$ and consider the heat kernel $(k_t) \in C^{\infty}(M \times M \times \IR_+; E \boxtimes E^*)$ associated to $L$ on $M$, i.e. for all $x \in M$
\[
(\partial_t + L)k_t(\cdot, x) = 0, \qquad
k_t(\cdot, x) \xrightarrow[t \to 0]{} \id_{E_x} \delta_x.
\]
Consider the decomposition $M = M_{cpt} \dot\cup M_{ncpt}$ with $M_{ncpt} = \bigcup_l N_l$ and let $s : M \to [0, \infty)$ such that it restricts to the coordinate $s$ on each cusp $N_l$ and $s=0$ on $M_{cpt}$.

\begin{Lemma} \label{Lem:wholehkest}
\begin{enumerate}[(a)]
 \item We have for all $x \in M$
\[ \Vert k_t(\cdot, x) \Vert_{L^1(B_\sigma(x) \times [0,\sigma^2])} < C. \]
 \item For $x, y \in M$ with either $d(x,y) \geq \sigma$ or $t \geq \sigma^2$
\begin{alignat*}{1}
  |k_t|(x,y) &< C \exp \left( {\textstyle \frac12} (n-1) (s(x) + s(y)) -(n-2)t \right). \\
  |k_t|(x,y) &< C \exp \left( (n-1) s(x) \right). 
\end{alignat*}
\end{enumerate}
\end{Lemma}

\begin{proof}
Part (a) follows in the same way as in the proof of Lemma \ref{Lem:oschk}.

Moreover, as in the proof of Lemma \ref{Lem:oschk}, we can derive the following bound for $t < 1$ (see (\ref{eq:CLYderder}) and (\ref{eq:volest}))
\begin{equation*} |\nabla^m k_t|(x,y) < C_m t^{-(n+m)/2} \exp \left( - \frac{d^2(x,y)}{8t} + \sfrac12 (n-1) (s(x) + s(y)) \right)
\end{equation*}
and the following bound for $t \geq \sigma^2/5$ (see (\ref{eq:ktnminus2}))
\begin{multline*} 
| k_t|(x,y) \leq \Vert k_{t/2}(x, \cdot) \Vert_{L^2(M)} \Vert k_{t/2}(\cdot,y) \Vert_{L^2(M)} \\
 < C \exp \left( {\textstyle \frac12} (n-1) (s(x) + s(y)) -(n-2)t \right).
 \end{multline*}
Hence, we have established the first inequality of part (b) and the second one in the case $t < 1$.

Observe that for the second inequality for $t \geq 1$, we only have to consider the case $s(x) < s(y)$, i.e. $y \in N_l$ for some $l$.
We fix $x \in M$ and analyze the function $q_t (y) = k_t^*(y,x) = k_t(x,y)$ on the cusp $N_l$.
It satisfies the linear equation $(\partial_t + L)q_t(y) = 0$.
As in subsection \ref{subsec:invosc}, we can split $q_t = q_t^{inv} + q_t^{osc}$.
By the first inequality of part (b) and (\ref{eq:trick}) in the proof of Lemma \ref{Lem:oschk}, we get
\begin{multline*} 
|q_t^{osc}| (y) \leq C \exp(-\sfrac12 (n-1) s(y))  \exp(\sfrac12 (n-1)(s(x) + s(y)) - (n-2)t) \\
\leq C \exp (\sfrac12 (n-1) s(x)). 
\end{multline*}
Hence it remains to bound $q_t^{inv}$.
Note that $g_t^{inv}$ only depends on $s$ and $t$ and satisfies the system of heat equations with right-hand side (\ref{eq:Linv1})-(\ref{eq:Linv3}).
Moreover, by the first inequality of part (b), we have $|q^{inv}_t| \leq C \exp ( (n-1) s(x))$ on the parabolic boundary $\{ s \geq s(x) + 1 \} \times \{ 0 \} \cup \{ s = s(x) + 1 \} \times [0, \infty)$.
So by the maximum principle
\[ |q^{inv}_t |(y) \leq C \exp ((n-1)s(x)) \]
which establishes the last inequality.
\end{proof}

\subsection{The final argument}
\begin{proof}[Proof of Theorem \ref{Thm:main}]
 By Proposition \ref{Prop:MDTisRF}, it suffices show Proposition \ref{Prop:main}, i.e. that we have convergence for modified Ricci deTurck flow $(g_t)$ (see (\ref{eq:MRdTflow})).
Set  $h_t = g_t - \ov g$.
By (\ref{eq:Qtequation}), we can write down the flow equation as
\[ \partial_t h_t + L h_t = Q_t = R[h_t] + \nabla^* S[h_t],
\]
where $|Q_t| \leq C (|h_t|^2+|\nabla h_t|^2+|\nabla^2 h_t|^2)$ if $|h_t|<0.1$.
Let $[0,T_{\max})$ be the maximal time interval on which a solution $(h_t)$ to the modified Ricci deTurck flow equation exists which is uniformly bounded on compact time intervals.
If $H = \Vert h_0 \Vert_{L^\infty(M)} < \varepsilon_{s.e.}$, then Proposition \ref{Prop:shortex} implies $T_{\max} \geq \tau_{s.e.} > 200 \sigma^2$ and $\Vert h_t \Vert_{L^\infty(M \times [0, \tau_{s.e.}])} < C H$ as well as $\Vert h_t \Vert_{C^{4;2}(M \times [\frac12 \tau_{s.e.}, \tau_{s.e.}])} < C H$ (so assume from now on $\varepsilon < \varepsilon_{s.e.}$).
Moreover, if $T_{\max} < \infty$, it follows that we cannot have $\Vert h_t \Vert_{L^\infty(M)} < \varepsilon_{s.e.}$ for all $t \in [0,T_{\max})$, since otherwise the solution could be extended to the time interval $[0,T_{\max} + \tau_{s.e.})$. 
By Corollary \ref{Cor:Shi}, this implies that we even cannot have $\Vert h_t \Vert_{L^\infty(M)} < \varepsilon_0$ for all $t \in [0,T_{\max})$ where $\varepsilon_0$ has to be sufficiently small.
In the following we will show that for small enough $H$, we can bound this norm by $CH$ for some $C$ which is independent of $T_{\max}$.
This implies then that $T_{\max} = \infty$ if $H$ is sufficiently small.

Choose and fix two constants $\lambda, \beta$ which satisfy the inequalities $\frac12(n-1) < \beta < n-1$, $0 < \lambda < n-2$ and $(n-1)(n-2) > \frac{\lambda}2 (n-1) + \beta (n-2)$.
We introduce a time dependent weight function on $[0,\infty)$ which we will use to bound $h_t$:
\[ W_t(s) = \min \big\{ \exp(\beta s - \lambda t), 1 \big\}. \]
For any $T \leq T_{\max}$ set
\[ \omega_T = \sup_{(x,t) \in M \times [0,T)} W^{-1}_t(s(x)) | h_t |(x). \]
Observe that if $H < \varepsilon_{s.e.}$, then
\begin{equation} \label{eq:omegasigmasq}
 \omega_{\sigma^2} \leq C_1 H.
\end{equation}
By Corollary \ref{Cor:Shi}, we conclude that there is an $\varepsilon_1 > 0$ such that if $\omega_T < \varepsilon_1$, then for $t \geq \sigma^2$
\[
| \nabla h_t |(x), \; | \nabla^2 h_t |(x) \leq C \omega_T W_t(s(x)). 
\]
In this case, we can estimate
\begin{equation} 
| Q_t |(x) \leq C \omega_T^2 W_t^2(s(x))  \quad \text{for} \quad t \in [\sigma^2, T). \label{eq:estQ}
\end{equation}

Lemma \ref{Lem:wholehkest} (b) implies that whenever $d(x,y) \geq \sigma$ or $t \geq \sigma^2$, the heat kernel obeys the bound $|k_t|(x,y) \leq C K_t(s(x), s(y))$ where
\[ K_t(s_1, s_2) = \min \big\{ \exp( \sfrac12 (n-1) (s_1 + s_2) -(n-2) t ), \; \exp ( (n-1) s_1) \big\}. \]

Let now $\sigma^2 < T \leq T_{\max}$ and assume $\omega_T < \varepsilon_1$.
Choose $(x_0, t_0) \in M \times [0,T)$, set $s_0 = s(x_0)$ and consider the case
\[ \beta (s_0 - 11 \sigma) \leq \lambda t_0. \]
We will derive a better bound on $h_{t_0}(x_0)$.
If $t_0 \leq \tau_{s.e.}$, then $|h_{t_0}(x_0)| \leq C H$.
So assume in the following that $t_0 > \tau_{s.e.} > 200 \sigma^2$.
Analogous to subsections \ref{subsec:representh} and \ref{subsec:representu}, we can derive the representation
\[ h_{t_0}(x_0) = \int_{M \times [\sigma^2,t_0]} k_{t_0-t}(x_0, x) Q_t (x) dx dt + \int_M k_{t_0-\sigma^2}(x_0, x) h_{\sigma^2}(x) dx. \]
If we split the last integral into integrals over $M_{cpt}$ and the cusps, we find that its absolute value is bounded by
\[ C H K_{t_0 - \sigma^2}(s_0,0) + C H \int_0^\infty K_{t_0 - \sigma^2}(s_0,s) e^{-(n-1)s} ds \leq C H \exp (\beta s_0 - \lambda t_0). \]
The first integral can be split into integrals over $B_\sigma(x_0) \times [t_0-\sigma^2, t_0]$ and its complement in $M \times [0,t_0]$, so its absolute value is bounded by
\begin{multline*} 
\int_{B_\sigma(x_0) \times ([t_0 - \sigma^2, t_0] \cap [\sigma^2,t_0])} |k_{t_0-t}|(x_0, x) |Q_t| (x) dx dt \\
+  \int_{M \times [\sigma^2, t_0]} K_{t_0-t}(s_0,s(x)) |Q_t|(x) dx dt
\end{multline*}
By Lemma \ref{Lem:wholehkest} (a) and (\ref{eq:estQ}), the first integral is bounded by $C \omega_T^2 W_{t_0}^2(s_0)$ and if we split the second integral into integrals over $M_{cpt}$ and each of the cusps, we get the bound
\[ C \omega_T^2 \left( \int_0^{t_0} K_{t_0-t}(s_0,0) W_t^2(0) dt + \int_0^{t_0} \int_0^\infty K_{t_0-t}(s_0,s) W_t^2(s) e^{-(n-1)s} ds dt \right). \]

The next Lemma implies then
\begin{equation} |h_{t_0}|(x_0) \leq C (\omega_T^2 + H) \exp (\beta s_0 - \lambda t_0 ). \label{eq:htest1} \end{equation}
So around the boundaries $\partial N_l$ of the cusps, we have the estimate $|h_t| \leq C ( \omega_T^2 + H ) e^{-\lambda t}$.
Proposition \ref{Prop:Shi} again implies that if $C (\omega_T^2 + H) < \varepsilon_1$, we have the bounds 
\[ |h_t|, \; |\nabla h_t|, \; |\nabla^2 h_t| \leq C ( \omega_T^2 + H ) e^{-\lambda t}  \quad \text{for} \quad t \in [2\sigma^2,T).\]
So using the derivative bounds on the time interval $[\frac12 \tau_{s.e.}, \tau_{s.e.}]$, we can invoke Theorem \ref{Thm:RdToncusp} to find that if $C(\omega_T^2 + H) < \min \{ \varepsilon_{cusp}, \varepsilon_1 \}$ and $\omega_T < \varepsilon_1$, then
\[ \Vert h_t \Vert_{L^\infty(M_{ncpt} \times [0,T))} \leq C_{cusp} C (\omega_T^2 + H). \]
This together with (\ref{eq:htest1}) yields
\begin{equation} \label{eq:omega}
\omega_T \leq C_0 ( \omega_T^2 + H ). 
\end{equation}
For some constant $C_0$ which is independent of $T_{\max}$.

Now let $0 < \varepsilon' < (2C_0)^{-1}$ and small enough such that $\omega_T, H \leq \varepsilon'$ implies that we can carry out all steps of the argument above (in particular apply Proposition \ref{Prop:shortex}).
We remark that $\omega_T$ depends continuously on $T$.
Let $\varepsilon = \min \{ (2C_0)^{-1}, (2C_1)^{-1}, 1 \} \varepsilon'$ (where $C_1$ is the constant from (\ref{eq:omegasigmasq})).
If $H < \varepsilon$, then we cannot have $\omega_T = \varepsilon'$ for any $T \in [\sigma^2,T_{\max})$ since this would contradict (\ref{eq:omega}).
So if $H < \varepsilon$, we have $\omega_{\sigma^2} \leq C_1 H < \varepsilon'$ and hence by continuity $\omega_{T_{\max}} < \varepsilon'$.
This implies
\[ \omega_{T_{\max}} \leq 2 C_0 H \]
which concludes the proof.
\end{proof}

\begin{Lemma} \label{Lem:C0est}
There is a constant $C$ which does not depend on $s_0$ or $t_0$ such that
\begin{alignat*}{1} 
 \int_0^{t_0} K_{t_0-t}(s_0,0) W_t^2(0) dt &\leq C \exp(\beta s_0 - \lambda t_0 ) \\
 \int_0^{t_0} \int_0^\infty K_{t_0-t}(s_0,s) W_t^2(s) e^{-(n-1)s} ds dt &\leq C \exp(\beta s_0 - \lambda t_0 )
\end{alignat*}
\end{Lemma}
\begin{proof}
First observe that since 
\[ K_{t_0-t} (s_0, 0) W^2_t(0) < C \min_{0 \leq s \leq 1} K_{t_0-t}(s_0, s) W^2_t(s) e^{-(n-1)s}, \]
we have the estimate 
\[ K_{t_0-t} (s_0, 0) W^2_t(0) < C \int_0^1 K_{t_0-t}(s_0, s) W^2_t(s) e^{-(n-1)s} ds \]
and thus we only have to prove the second inequality since it implies the first one.

In the following, we will simply estimate $W^2_t(s) \leq \exp(\beta s - \lambda t)$ and prove that $\int_0^{t_0} \int_0^\infty P(s,t) ds dt \leq C$ where
\[ P(s,t) = K_{t_0-t}(s_0,s) \exp ( \beta (s-s_0) - \lambda (t-t_0) - (n-1) s ). \]
The desired inequality then follows immediately.

\begin{figure}[t] \label{fig:regions}
\caption{The regions $R_1$, $R_2$.}
\vspace{15mm}
\begin{center}
\begin{picture}(0,0)%
\includegraphics[width=7cm]{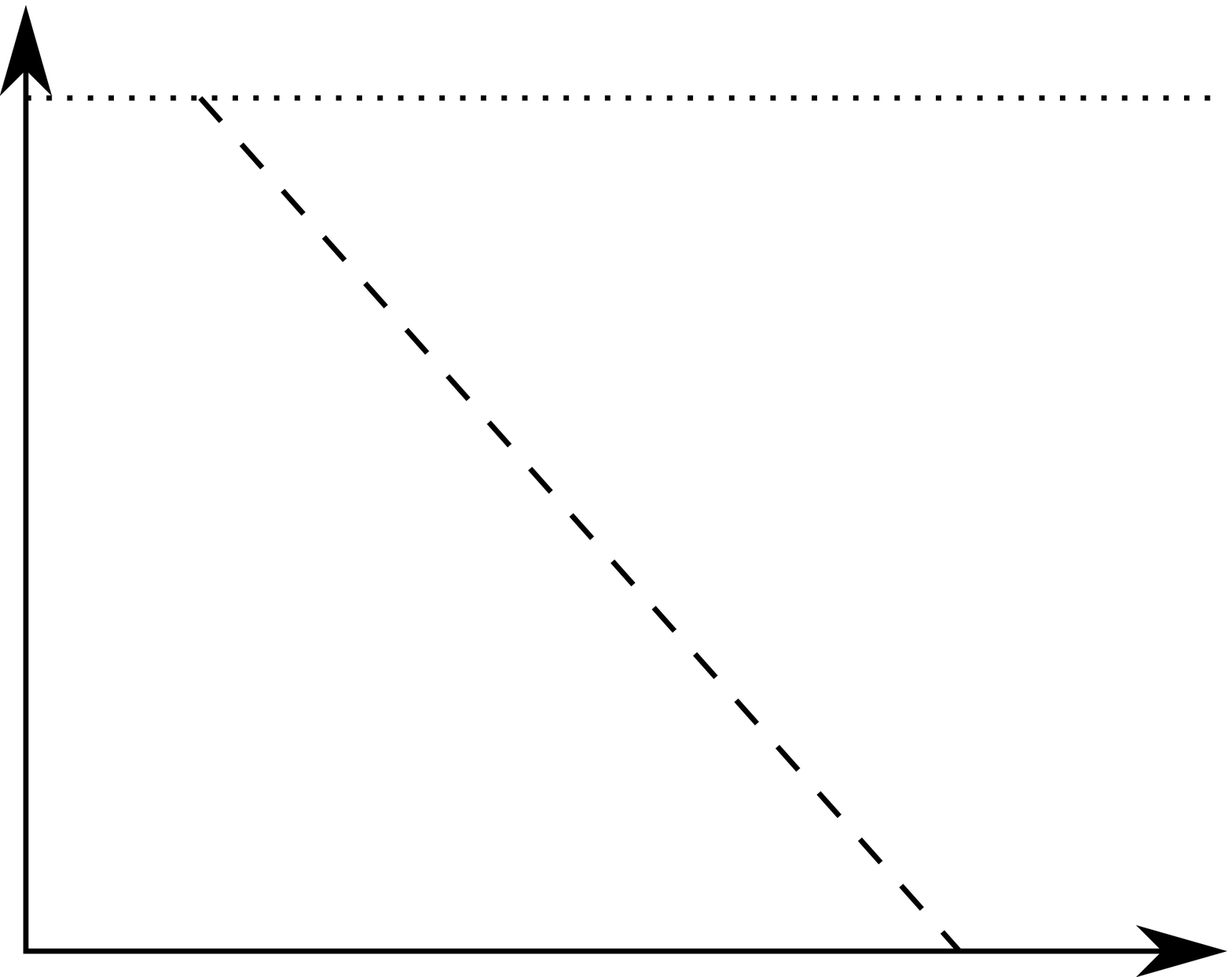}%
\end{picture}
 \setlength{\unitlength}{3947sp}%
 \setlength{\unitlength}{0.6\unitlength}
 \begingroup\makeatletter\ifx\SetFigFont\undefined%
 \gdef\SetFigFont#1#2#3#4#5{%
 \reset@font\fontsize{#1}{#2pt}%
 \fontfamily{#3}\fontseries{#4}\fontshape{#5}%
 \selectfont}%
 \fi\endgroup%
 \begin{picture}(5300,3500)(518,-4964) \put(4500,-3200){\makebox(0,0)[lb]{\smash{{\SetFigFont{12}{14.4}{\familydefault}{\mddefault}{\updefault}$R_2$}}}}
 \put(1300,-3200){\makebox(0,0)[lb]{\smash{{\SetFigFont{12}{14.4}{\familydefault}{\mddefault}{\updefault}$R_1$}}}}
 \put(2700,-2600){\makebox(0,0)[lb]{\smash{{\SetFigFont{12}{14.4}{\familydefault}{\mddefault}{\updefault}$L$}}}}
 \put(900,-850){\makebox(0,0)[lb]{\smash{{\SetFigFont{12}{14.4}{\familydefault}{\mddefault}{\updefault}$(s_0,t_0)$}}}}
 \put(5500,-4600){\makebox(0,0)[lb]{\smash{{\SetFigFont{12}{14.4}{\familydefault}{\mddefault}{\updefault}$s$}}}}
 \put(550,-700){\makebox(0,0)[lb]{\smash{{\SetFigFont{12}{14.4}{\familydefault}{\mddefault}{\updefault}$t$}}}}
 \end{picture}%
\end{center}
\vspace{1mm}
\end{figure}

Divide the domain $[0,\infty) \times [0,t_0]$ into two regions $R_1$ and $R_2$ by the line $L$
\[ \sfrac12 (n-1) (s - s_0) -(n-2) (t_0-t) = 0. \]
It indicates where the two terms in the minimum of the definition of $K_{t_0-t}(s_0,s)$ agree.
Then on region $R_1$
\begin{multline*}
 P(s,t) = \exp\big( \tfrac12 (n-1) (s_0 - s) - (n-2) (t_0 -t) + \beta(s - s_0) - \lambda (t - t_0) \big) \\
 = \exp \big( (\beta - \tfrac12(n-1)) (s - s_0) + (n-2+\lambda)(t_0-t) \big)
\end{multline*} 
and on region $R_2$
\begin{multline*}
 P(s,t) = \exp \big( (n-1)(s_0 - s) + \beta(s - s_0) - \lambda (t - t_0) \big) \\
 = \exp \big( - ( n-1 - \beta) (s - s_0) + \lambda (t_0 - t) \big).
\end{multline*}
On the line $L$, we have
\[ P(s,t) = \exp \Big( - \frac{2}{n-1} \big( (n-1)(n-2) - \tfrac{\lambda}2 (n-1) -  \beta (n-2) \big) (t_0 - t) \Big). \]

Hence, since $\beta > \tfrac12(n-1)$ and $\delta =  \frac{2}{n-1} ( (n-1)(n-2) - \tfrac{\lambda}2 (n-1) - \beta (n-2) ) > 0$, the integral over $R_1$ can be estimated by 
\[ \int_{R_1} P(s,t) ds dt \leq C \int_0^{t_0} \exp (- \delta (t_0 - t) ) dt \leq C. \]
Analogously, we can bound the integral over $R_2$.
Here we use the fact that $\beta < n-1$.
\end{proof}

\end{document}